\documentclass[12pt]{amsart}
 \pdfoutput=1
\usepackage{microtype}
\usepackage{etex} 
\usepackage[active]{srcltx}
\usepackage{calc,amssymb,amsthm,amsmath,amscd, eucal,ulem}
\usepackage{alltt}
\usepackage{mathtools}
\usepackage[left=1in,top=1in,right=1in,bottom=1in]{geometry}
\usepackage{bbding}
\RequirePackage[dvipsnames,usenames]{color}

\input{mabliautoref.sty}
\input{xy}
\xyoption{all}
\input{kmacros3.sty}
\usepackage{tikz}
\usepackage{amsfonts, mathrsfs}
\usepackage[cal=boondox, calscaled=1.05]{mathalfa}
\usepackage{calligra}
\usepackage{amssymb}
\usepackage{stmaryrd} 
\usepackage{dcpic, pictexwd}
\usepackage[left=1in,top=1in,right=1in,bottom=1in]{geometry}
\usepackage{bm}
\usepackage{verbatim}
\usepackage{upgreek}
\usepackage{systeme}
\usepackage{colonequals}
\usepackage{url}

\numberwithin{equation}{theorem}

\newcommand{\kay}{\mathcal{k}}
\newcommand{\el}{\mathcal{l}}

\newcommand{\F}{\mathbb{F}}

\renewcommand{\:}{\colon}
\newcommand{\eg}{{\itshape e.g.} }

\renewcommand{\m}{\mathfrak{m}}

\newcommand{\p}{\mathfrak{p}}
\newcommand{\q}{\mathfrak{q}}

\DeclareMathOperator{\Ass}{Ass}

\newcommand{\Gal}{\textnormal{Gal}}

\DeclareMathOperator{\Cl}{Cl}

\usepackage{setspace}
\usepackage{hyperref}
\usepackage{enumerate}
\usepackage{enumitem}
\usepackage{graphicx}
\usepackage[all,cmtip]{xy}

\usepackage{verbatim}

\theoremstyle{theorem}

\renewcommand{\sB}{\mathcal{B}}

\renewcommand{\sF}{\mathcal{F}}

\renewcommand{\sH}{\mathcal{H}}

\renewcommand{\sK}{\mathcal{K}}

\renewcommand{\sO}{\mathcal{O}}

\renewcommand{\sT}{\mathcal{T}}
\renewcommand{\sU}{\mathcal{U}}

\usepackage{marvosym}

\newcommand{\unp}{{1/{p}}}

\newcommand{\fe}{F^e}
\newcommand{\fstare}{\fe_*}
\newcommand{\fstar}{F_*}

\newcommand{\rphi}{(R,\phi)}
\newcommand{\spsi}{(S,\psi)}
\newcommand{\phin}{{\phi^n}}
\newcommand{\transpose}{^\frT}
\newcommand{\transp}{\transpose}
\newcommand{\sphit}{(S,\phi\transp)}

\newcommand{\phit}{\phi\transp}
\newcommand{\pt}{\p\transp}

\newcommand{\pS}{\p S}
\newcommand{\thetat}{(\theta, \frT)}

\newcommand{\inv}{^{-1}}

\begin{document}
\title{On tame ramification and centers of $F$-purity}
\author[J.~Carvajal-Rojas]{Javier Carvajal-Rojas}
\address{Centro de Investigaci\'on en Matem\'aticas, A.C., Callej\'on Jalisco s/n, 36024 Col. Valenciana, Guanajuato, Gto, M\'exico}
\email{\href{mailto:javier.carvajal@cimat.mx}{javier.carvajal@cimat.mx}}

\author[A.~Fayolle]{Anne Fayolle}
\address{Department of Mathematics\\ University of Utah\\ Salt Lake City\\ UT 84112\\USA}
\email{\href{mailto:fayolle@math.utah.edu}{fayolle@math.utah.edu}}

\thanks{Carvajal-Rojas was partially supported by the grants ERC-STG \#804334, FWO \#G079218N, and CONAHCYT \#CBF2023-2024-224. Fayolle was partially supported by an NSERC doctoral fellowship.
}

\begin{abstract}
We introduce a notion of tame ramification for general finite covers. When specialized to the separable case, it extends to higher dimensions the classical notion of tame ramification for Dedekind domains and curves and sits nicely in between other notions of tame ramification in arithmetic geometry. However, when applied to the Frobenius map, it naturally yields the notion of center of $F$-purity (aka compatibly $F$-split subvariety). As an application, we describe the behavior of centers of $F$-purity under finite covers---it all comes down to a transitivity property for tame ramification in towers.
\end{abstract}
\maketitle

\section{Introduction} \label{sec.Introduction}
The Cohen--Seidenberg theorems \cite{CohenSeidenberg} on the behavior of prime ideals in an integral extension are among the most fundamental results in commutative algebra. For instance, both dimension theory in algebraic geometry and ramification theory in algebraic number theory rely on them. On the other hand, Schwede's centers of $F$-purity \cite{SchwedeCentersOfFPurity} (aka compatibly $F$-split subvarieties in the literature \cite{BrionKumarFrobeniusSplitting}) are arguably the most basic objects in the theory of $F$-singularities. They are to $F$-singularity theory what log canonical centers are to the theory of singularities in the Minimal Model Program. However, do Cohen--Seidenberg-type theorems hold for centers of $F$-purity? Schwede and Tucker had warned us in \cite{SchwedeTuckerExplicitlyExtending} to be cautious with wild ramification. But what else is blocking the way? We uncover it here.

To do so, we establish a bridge between the two seemingly independent notions of tame ramification in arithmetic geometry and $F$-singularities in commutative algebra. 
The core of this work is the development of an unifying notion of tame ramification under which these are nothing but two sides of the same coin; see \autoref{section tame ramification}, \autoref{Section CFPs and Cartier Core map}. With this in place, the relationship between tame ramification and good behavior of $F$-singularities in finite covers (\eg Cohen--Seidenberg-type theorems for centers of $F$-purity) becomes transparent; see \autoref{sec.FiberedTranspositions}. As an appetizer, let us glimpse at our main application on curves:

\begin{mainthm*}[{Teaser version, \cf \autoref{cor.MainTheorem}, \autoref{cor.GaloisCase}}]
Let $R$ be an $F$-finite Dedekind domain with field of fractions $K$ and $\phi \: F_* R \to R$ be a Frobenius splitting. Let $L/K$ be a finite Galois extension and $S$ be the normalization of $R$ in $L$. Suppose that $\phi$ extends to $\psi \: F_* S \to S$. Then, the following statements are equivalent:
\begin{enumerate}
    \item $S/R$ is tamely ramified over the centers of $F$-purity of $(R,\phi)$.
    \item The centers of $F$-purity of $(S,\psi)$ are exactly those prime ideals of $S$ lying over centers of $F$-purity of $(R,\phi)$.
    \item $\Tr_{L/K}(S)=R$, \ie $S/R$ is everywhere tamely ramified.
\end{enumerate}
In that case, Cohen--Seidenberg-type theorems apply to centers of $F$-purity.
\end{mainthm*}

Our actual main theorem is much more general but, to properly formulate it beyond curves, we need to introduce a few notions. Additionally, we would like to do justice to its history by summarizing the ideas behind it present in earlier works. Indeed, we took great inspiration from the works of Polstra, Schwede, Speyer, St\"abler, and Tucker.

\begin{convention}
All rings are noetherian. We assume $0 \in \bN$ and $\bN_{+} \coloneqq \bN \setminus \{0\}$. Further, $p$ is a prime number and $q \coloneqq p^e$, $q'\coloneqq p^{e'}$, $q_0 \coloneqq p^{e_0}$, etc. We let $\bF_q$ be the finite field with $q$ elements. Given an $\bF_p$-algebra $R$, we denote the $e$-th iteration of its Frobenius endomorphism by $F^e = F^e_R \: R \to R$, \ie $F^e \: r \mapsto r^q$. If $F\:R \to R$ is finite one says that $R$ is $F$-finite, \ie the $R$-module $F_*^e R=\{F^e_* r \mid r \in R\}$ obtained by restriction of scalars along $F^e$ is finitely generated. So $r'F^e_*r=F^e_*r'^qr$ and $F^e_* r + F^e_* r'=F^e_*(r+r')$ for all $r',r \in R$. 
\end{convention}

Let us fix a normal integral $F$-finite $\bF_p$-algebra $R$ with field of fractions $K$ as well as a surjective map $\phi \in \Hom_R(F^e_* R,R)$, \eg an \emph{$F$-splitting}. According to Schwede \cite{SchwedeFAdjunction}, this is tantamount to the choice of an effective $\bZ_{(p)}$-divisor $\Delta_{\phi}$ on $X \coloneqq \Spec R$ such that $K_X + \Delta_{\phi} \sim_{\bZ_{(p)}} 0$, where $K_X$ is a canonical divisor on $X$. We refer to any such divisor as a \emph{tame boundary} on $X$. Since $\phi$ is surjective, $(X,\Delta_{\phi})$ is $F$-pure and so we may think of it as a \emph{tame} log Calabi--Yau pair. This is part of why $F$-splittings matter in birational geometry. 

\subsection{A bit of history and background} In \cite{SchwedeTuckerTestIdealFiniteMaps}, Schwede and Tucker studied the following fundamental question. Given a finite extension $R \subset S$ with $S$ a normal integral domain with field of fractions $L$, when does $\phi$ lift to a map $\psi \: F^e_* S \to S$? As it turns out, $L/K$ must be separable; see \autoref{cor.FirstCriterionSeparability}, \cite[Proposition 5.2]{SchwedeTuckerTestIdealFiniteMaps}. In that case, the trace map $0\neq \Tr \: L \to K$ restricts to an $R$-linear map $\Tr \: S \to R$ and the corresponding finite cover $f\: Y\coloneqq \Spec S \to X$ has a ramification divisor $\Ram_f \sim K_Y-f^*K_X$. They established that $\psi$ lifting $\phi$ is the exact same thing as having a commutative diagram
\[
\xymatrixrowsep{1.8pc}\xymatrix{ 
F^e_* S \ar[r]^-{\psi} \ar[d]_-{F_*^e\Tr_{}} & S \ar[d]^-{\Tr_{}} \\
F^e_* R  \ar[r]^-{\phi} & R,
}
\]
and therefore to an equality $\Delta_{\psi}=f^*\Delta_{\phi}-\Ram_f$. In particular, $K_Y+\Delta_{\psi} \sim_{\bZ_{(p)}} f^*(K_X+\Delta_{\phi})$. By choosing $K_Y \coloneqq \Ram_f + f^*K_X$, we can write $K_Y+\Delta_{\psi} = f^*(K_X+\Delta_{\phi})$, which is to say that $f\:(Y,\Delta_{\psi}) \to (X,\Delta_{\phi})$ is a \emph{(log) crepant} finite cover between log pairs.

Given an extension of pairs $(R,\phi) \subset (S,\psi)$ as above, it is natural and fruitful to wonder about the relationships between the ($F$-)singularities of $(R,\phi)$ and those of $(S,\psi)$. A neat way to investigate them are transformation rules for the invariants measuring $F$-singularities. The case of \emph{test ideals} was originally treated in \cite{SchwedeTuckerTestIdealFiniteMaps} whereas local invariants such as \emph{$F$-splitting ratios} and \emph{$F$-splitting primes} were treated in \cite{CarvajalStablerFsignaturefinitemorphisms}; after \cite{CarvajalSchwedeTuckerEtaleFundFsignature,CarvajalFiniteTorsors}. For example, for test ideals, $\Tr \big( \uptau(S,\psi) \big) = \uptau(R,\phi)$. Likewise, if $(R,\fram) \subset (S,\fran)$ is further a local extension and $\Tr \: S \to R$ is surjective, $\upbeta(S,\psi)  \cap R = \upbeta(R,\phi)$, where $\upbeta$ is the $F$-splitting prime of Aberbach and Enescu \cite{AberbachEnescuStructureOfFPure}. Moreover, for the $F$-splitting ratios, we have:
\[
\big[\kappa(\fran):\kappa(\fram)]r(S,\psi) = \big[\kappa\big(\upbeta(S,\psi)\big):\kappa\big( \upbeta(R,\phi)\big)\big] r(R,\phi),
\]
where $\kappa(\p)$ denotes the residue field at a prime $\p$. Importantly, the surjectivity of the trace is a notion of tame ramification known as cohomological tame ramification \cite{KerzSchmidtOnDifferentNotionsOfTameness,ChinburgErezPappasTaylorTameActions}. 

These transformation rules have been used, for instance, in the study of the topology of singularities; see \cite{CarvajalSchwedeTuckerEtaleFundFsignature,CarvajalRojasStaeblerTameFundamentalGroupsPurePairs,BhattCarvajalRojasGrafSchwedeTucker}, \cf \cite{JeffriesSmirnovTransformationRuleForNaturalMultiplicities,CaiLeeMaSchwedeTuckerPerfectoidSignature}. Likewise, they have also been key tools in other works such as \cite{PolstraSimpsonFPurityDeformsQGorenstein,LeePandeFSignatureAmpleCone,LiedtkeMartinMatsumotoTorsorsRationalsDoublePoints,JeffriesNakajimaSmirnovWatanabeYoshidaLowerBoundsHKMULTIPLICITIES,TaylorInversionAdjunctionFSignature}.

The test ideal $\uptau(R,\phi)$ and the $F$-splitting prime $\upbeta(R,\phi)$ (when $R$ is local) are the most prominent examples of \emph{$\phi$-ideals} (aka $\phi$-compatible ideals). An ideal $\fra \subset R$ is referred to as a $\phi$-ideal if $\phi(F^e_*\fra) \subset \fra$. More geometrically, their closed subschemes have also been studied by many under the name of compatibly ($F$-)split; see \cite{BrionKumarFrobeniusSplitting}.

We denote the set of $\phi$-ideals by $I(R,\phi)$, which is a sublattice of the bounded lattice of ideals of $R$ (containing $(0)$ and $(1)$). Since $\phi$ is surjective, $I(R,\phi)$ is a finite lattice of radical ideals (\cite{KumarMehtaFiniteness,SchwedeFAdjunction,SchwedeTuckerNumberOfFSplit}) whose least nonzero element is $\uptau(R,\phi)$ and, if $R$ is further local, whose greatest proper element is $\upbeta(R,\phi)$. Summing up, what we have up to date is a good understanding of how the least nonzero and greatest proper elements of $I(R,\phi)$ and $I(S,\psi)$ compare to one another. However, until now, the whole relationship between $I(R,\phi)$ and $I(S,\psi)$ remained unknown. 

\subsection{Overview of the paper}
Classically, tame ramification is a notion that pertains to finite extensions of Dedekind domains. Defining a suitable notion in higher dimensions, however, has been rather tricky. Quoting Kerz and Schmidt \cite{KerzSchmidtOnDifferentNotionsOfTameness}, ``The notion of a tamely ramified covering is canonical only for curves.'' Nevertheless, there have been several different notions of tame ramification in higher dimension in arithmetic geometry, which are all well-dissected by Kerz and Schmidt in \cite{KerzSchmidtOnDifferentNotionsOfTameness}. Unfortunately, none of them are well suited for our job at hand nor; the authors believe, for other commutative algebra settings. 

In \autoref{section tame ramification}; the heart of this paper, we fill this gap. We provide a notion of tame ramification in higher dimensions extending the one for Dedekind domains (\autoref{prop.TameRamificationParticularCases}) that is purely trace-theoretic and tailored to our work. In \autoref{subsection.TameRamificationSeparableCase}, we also detail how nicely it fits in between the notions of tame ramification discussed in  \cite{KerzSchmidtOnDifferentNotionsOfTameness}. In \autoref{setting workout example}, we test our notions of tame ramification against a hybrid between the separable case and Frobenius, \eg cyclic covers by torsion divisor classes. 

Unlike in this introduction, our definition and theory do not require our rings to be normal or of any particular characteristic. Moreover, our notion is notably 
elementary and is solely based on the most basic commutative algebra concepts. Indeed, our definition of tame ramification is as follows: let $\theta\colon A\to B$ be a finite homomorphism of rings such that $\ker \theta \subset \sqrt{(0)}$, $\frT\in \Hom_A(B,A)$, and $\p\in \Spec A$. Then, we let $\pt\coloneqq\{b\in B\mid \frT(bB)\subset \p\}$ be the $\frT$-\emph{transposition} of $\p$ (\autoref{def.TranspositionOfIdeals}). We say that $B/A$ is \emph{tamely $\frT$ ramified} over $\p$ if $\sqrt{\p B}=\pt$, \ie $\sqrt{\p B}$ is the largest $B$-ideal $\frab$ such that $\frT(\frab) \subset \p$; see \autoref{definition tamely T ramified}.
If $\frT=\Tr$ in the separable normal integral case, we say that $B/A$ is tamely ramified over $\p$.

Many experts will find our definition familiar for particular cases of it lurk in the literature already. For example, it shows up in \cite[Theorem 5.12]{CarvajalStablerFsignaturefinitemorphisms} as a key hypothesis. Moreover, it may be familiar to those who know the definition of $F$-splitting primes/numbers. In fact, the primes $\p \in \Spec R$ over which $F^e\: R \to R$ is tamely $\phi$-ramified are precisely the centers of $F$-purity of $(R,\phi)$, which are the prime ideals in $I(R,\phi)$ (see \autoref{definition phi ideals}). 

To prove our main theorem, we focus our attention on these primes. We treat them as their very own spectrum (rather than isolated integral closed subschemes), which we denote by $\Schpec(R,\phi) \subset \Spec R$ and endow with the subspace Zariski topology. As a set, $\Schpec(R,\phi)$ tells us where the $F$-singularities of $(R,\phi)$ are the most severe. Its topology, however, captures the basic shape of the $F$-singularities of $(R,\phi)$. For instance, $\dim \Schpec(R,\phi)=0$ if and only if $(R,\phi)$ is $F$-regular. 

In order to study $\Schpec(R,\phi)$, we use (a slightly different version of) the \emph{Cartier core map} $\upbeta$ considered by Badilla-C\'espedes and Brosowsky in \cite{BadillaCespedesFInvariantsSRRings,BrosowskyCartierCoreMap}. Namely, we let
$\upbeta = \upbeta(R,\phi) \: \Spec R \to \Schpec(R,\phi)$ be the map $\p \mapsto \upbeta_{\p}=\upbeta_{\p}(R,\phi) \coloneqq \bigcap_{n \in \bN} \p^{\phi^n}$ which is a \emph{continuous} retraction of the subspace inclusion $\Schpec(R,\phi) \subset \Spec R$; see \autoref{proposition beta as a map}. Thus, to study how centers of $F$-purity transform under finite covers one should ask how $\upbeta$ transforms. That is, how functorial is $\upbeta$? To answer these questions, it will be insightful to consider the coarsest topology on $\Spec R$ that makes $\upbeta \: \Spec R \to \Schpec(R,\phi)$ continuous. This is the topology whose set of closed subsets is $\{V(\fra) \mid \fra \in I(R,\phi)\}$; see \autoref{definition CZpec}. We refer to this topology as the \emph{Cartier--Zariski topology}. We denote by $\CZpec(R,\phi)$ the topological space given by the set $\Spec R$ endowed with this topology. 

 Let us see why such an odd-looking topology matters. Recall that $\Spec \: \mathsf{Rings} \to \mathsf{Top}$ is a functor from rings to topological spaces, so we may consider the continuous map $f \coloneqq \Spec \theta \: \Spec S \to \Spec R$. This relies on two things. Namely, the contraction of a prime $S$-ideal is a prime $R$-ideal and $f$ is continuous as $f^{-1}\big(V(\fra)\big)=V\big(\sqrt{\fra S}\big)$ for all ideals $\fra \subset R$. As $\psi$ restricts to $\phi$ on $R$, a $\psi$-ideal contracts to a $\phi$-ideal, so $f$ restricts to a continuous map $\Schpec \theta \: \Schpec(S,\psi) \to \Schpec(R,\phi)$. However, it is not true that $\sqrt{\fra S} \in I(S,\psi)$ if $\fra \in I(R,\phi)$; see \autoref{ex.NonContinuity}. Hence, $f=\CZpec \theta \: \CZpec(S,\psi) \to \CZpec(R,\phi)$ is not necessarily continuous. Nonetheless:

\begin{theoremA*}[{\autoref{thm.FibrationsHomos}}]
With notation as above, the following statements are equivalent:
\begin{enumerate}
    \item The map $\CZpec \theta  \:\CZpec(S,\psi) \to \CZpec(R,\phi)$ is continuous.
    \item The following diagram is commutative 
    \[
\xymatrixcolsep{4.5pc}\xymatrix{
\CZpec(R,\phi)\ar[d]_-{\upbeta(R,\phi)} & \CZpec(S,\psi) \ar[l]_-{\CZpec \theta} \ar[d]^-{\upbeta(S,\psi)} \\
\Schpec(R,\phi) & \Schpec(S,\psi) \ar[l]_-{\Schpec \theta}
}
    \]
    \item The equality $\sqrt{\upbeta_{\p} S} = \upbeta_{\sqrt{\p S}} \coloneqq \bigcap_{\q \in \Ass_S S/\sqrt{\p S}} \ \upbeta_\q$ of $S$-ideals holds for all $\p\in \Spec R$. 
    \item The equality $f^{-1}\{\upbeta_{\p}\} = \upbeta \big(f^{-1}\{\p\}\big)$ of subsets of $\Spec S$ holds
    for all $\p \in \Spec R$.
    \item The equality $
\Schpec(S,\psi) = f^{-1} \big(\Schpec(R,\phi)\big)$ of subsets of $\Spec S$ holds. 
\end{enumerate}
\end{theoremA*}
We shall say that $(R,\phi) \subset (S,\psi)$ is a \emph{fibered} extension if any of the equivalent conditions in Theorem A hold. In such case, the Cohen--Seidenberg theorems hold for centers of $F$-purity due to (e). In particular, $\dim \Schpec(S,\psi)=\dim \Schpec(R,\phi)$ and so $(S,\psi)$ is $F$-regular if and only if so is $(R,\phi)$ (akin to \cite[Proposition 5.7]{AtiyahMacdonald}). Note that condition (b) resembles one of a natural transformation. We make it precise in \autoref{sec.CategoricalFormulation} by introducing a suitable category of fibered crepant extensions. 

It is worth noting that Speyer studied the extent to which condition (a) in Theorem A holds in arithmetic families; see \cite{SpeyerFrobeniusSplit}. His ideas and techniques have been of great value in our work.

We may then rephrase \cite[Theorem 5.12]{CarvajalStablerFsignaturefinitemorphisms} by St\"abler and C-R as follows. If $S/R$ is tamely ramified over $\p \in \Schpec(R,\phi)$ then every $\q$ lying over $\p$ belongs to $\Schpec(S,\psi)$ and $\Tr(\uptau_{\q}(S,\psi))=\uptau_{\p}(R,\phi)$. Here, $\uptau_{\p}(R,\phi)$ denotes the \emph{test ideal along a center of $F$-purity} as treated in \cite{SmolkinSubadditivity,TakagiAdjointIdealsAndACorrespondence}. Using the transitivity property of tame ramification (\autoref{pro.TransitivityTameRamificationGeneralCase}), our main result provides a converse as well as a further generalization.

\begin{mainthm*}[{\autoref{cor.MainTheorem}, \autoref{cor.GaloisCase}, \autoref{cor.MainResult}}] \label{thm.MainTheorem Intro}
With notation as above, the following statements are equivalent:
\begin{enumerate}
    \item $S/R$ is tamely ramified over every $\p \in \Schpec(R,\phi)$.
    \item $(R,\phi) \subset (S,\psi)$ is fibered and $\Tr(\uptau_{\q}(S,\psi))=\uptau_{\q \cap R}(R,\phi)$ for all $\q \in \Schpec(S,\psi)$.
    \item $(R,\phi) \subset (S,\psi)$ is fibered and $\Tr(\uptau_{\q}(S,\psi))=\uptau_{\q \cap R}(R,\phi)$ for all $\q \in \Schpec(S,\psi)$ of height $\geq 2$.
\end{enumerate}
These conditions imply that $\Tr \: S \to R$ is surjective and the converse holds if $L/K$ is Galois.
\end{mainthm*}

In \autoref{section contracting compatible stuff}, we provide as an application another way to show that compatible ideals are trace ideals on $\bQ$-Gorenstein rings from the Gorenstein case, which was done recently in \cite{PolstraSchwedeCompatibleIdealsGorensteinRings} by Polstra and Schwede.

Everything we have discussed so far is local in nature, so there is no harm nor loss of generality in working in the ring-theoretic setting. However, 
it all readily extends to general finite covers $f\: Y \to X$ between schemes by gluing on and restricting to affine charts.

\subsection{Acknowledgements} 
 The authors are extremely grateful to W\'agner~Badilla-C\'espedes, Anna~Brosowsky, Karl~Schwede, Karen~Smith, and Axel~St\"abler for very enlightening discussions and help in the preparation of this paper. We are also thankful to the anonymous referees for their helpful suggestions and corrections.

\section{Criteria for Separability} \label{sec.CriteriaForSeparability}

In this preliminary section, we isolate two useful yet rather non-standard characterizations of separability of field extensions. An arbitrary extension of fields $L/K$ is separable if it is formally smooth as a homomorphism of rings. If the characteristic is $p>0$, then $L/K$ is separable if and only if the $e$-th relative Frobenius homomorphism $F^e_{L/K} \: L \otimes_K K^{1/q} \to L^{1/q}$ is injective for some (equivalently all) $e\in \bN_+$. See \cite[\href{https://stacks.math.columbia.edu/tag/031U}{Tag 031U}]{stacks-project} for more on separability. 

\begin{lemma} \label{lem.LemmaForSeparability} Let $L/K$ and $K'/K$ be field extensions inside a common field $L'$ such that $L/K$ and $L'/K'$ are finite. Let $\xi \:L\otimes_K K' \to L'$ be the canonical homomorphism. Then, $\xi$ is injective if and only if there is a $0\neq \lambda \in \Hom_K(L,K)$ that extends to a map $\lambda'\in \Hom_{K'}(L',K')$.  
\end{lemma}
\begin{proof}
    By assumption, $\xi$ is a $K'$-linear map between finite dimensional $K'$-modules. Hence, $\xi$ is injective if and only if its $K'$-dual $\Xi \coloneqq \Hom_{K'}(\xi,K')$ is surjective. Observe that
     \[
    \Hom_{K'}(L',K') \xrightarrow{\Xi} \Hom_{K'}(L\otimes_K K',K') \xleftarrow{\cong} K' \otimes_K \Hom_{K}(L,K),
    \]
    where the displayed isomorphism is the canonical homomorphism, which is an isomorphism as $L/K$ is finite (see \cite[\href{https://stacks.math.columbia.edu/tag/0A6A}{Tag 0A6A}]{stacks-project}). Therefore, by $K'$-linearity, $\Xi$ is surjective if and only if for all $\lambda \in \Hom_K(L,K)$ there is $\lambda' \in \Hom_{K'}(L',K')$ such that $\Xi(\lambda') = 1 \otimes \lambda(=\lambda \otimes_K K')$. Note that $\Xi(\lambda') = 1 \otimes \lambda$ means that $\lambda'$ is an extension of $\lambda$. 

    Since $L/K$ is finite, the $L$-module $\Hom_K(L,K)$ is free of rank $1$ and is generated by any nonzero element in it. Therefore, every map in $\Hom_K(L,K)$ lifts to a map in $\Hom_{K'}(L',K')$ as long as a nonzero one does. The result then follows.
\end{proof}

\begin{corollary}[First criterion, {\cf \cite[Lemma 3.3, Proposition 5.2]{SchwedeTuckerTestIdealFiniteMaps}}] \label{cor.FirstCriterionSeparability}
    Let $L/K$ be an extension of $F$-finite fields of characteristic $p>0$. Then, $L/K$ is separable if and only if there is a nonzero map $\phi \in \Hom_K(K^{1/q},K)$ that extends to a map $\psi \in \Hom_{L}(L^{1/q},L)$. 
\end{corollary}

\begin{corollary}[Second criterion]\label{second criterion separability}
    Let $L/K$ be a finite extension of fields of characteristic $p>0$. Then, $L/K$ is separable if and only if there is $0 \neq \frT \in \Hom_{K}(L,K)$ such that $\frT$ commutes with Frobenius (\ie $\frT(l^p)=\frT(l)^p$ for all $l \in L$). In that case, $\Tr_{L/K} \cdot (\bF_{q} \cap L) \subset \Hom_K(L,K)$ is the set of maps commuting with the $e$-th iteration of Frobenius.
\end{corollary}
\begin{proof}
    By \autoref{lem.LemmaForSeparability}, $L/K$ is separable if and only if for all $\frT \in \Hom_K(L,K)$ there is $\tilde{\frT} \in \Hom_K(L, K)$ such that $\tilde{\frT}^{1/p} \in \Hom_{K^{1/p}}(L^{1/p},K^{1/p})$ is an extension of $\frT$. In such cases, since $F_{L/K}^e$ is an injective $K^{1/q}$-linear map between $K^{1/q}$-modules of dimension $[L:K]<\infty$, this implies that $F_{L/K}$ and its $K^\unp$-dual are isomorphisms and so $\tilde{\frT}$ is uniquely determined by $\frT$. If $\frT \neq 0$, then $\tilde{\frT}=\frT \cdot t$ for a uniquely determined $0 \neq t \in L$. The condition that $\tilde{\frT}^{1/p}=\frT^{1/p} \cdot t^{1/p}$ lifts $\frT$ means that $\frT(l)^p=\frT(t l^p)$  for all $l \in L$.

    Summing up, $L/K$ is separable if and only if there is $0 \neq \frT \in \Hom_{K}(L,K)$ for which there is a (necessarily) unique $0 \neq t \in L$ such that $\frT(l)^p=\frT(t l^p)$ for all $l \in L$. Note that $\frT$ commuting with Frobenius means that we must take $t=1$.

    The above shows that the existence of a  map $\frT\neq 0$ commuting with Frobenius implies separability of $L/K$. For the converse, note that if $L/K$ is separable then the trace map $\Tr_{L/K} \in \Hom_{K}(L,K)$ is a nonzero map that commutes with Frobenius (\cite[Lemma 6]{SpeyerFrobeniusSplit}). 

    For the final statement, observe that every $\frT \in \Hom_K(L,K)$ is of the form $\Tr_{L/K} \cdot l$ for some unique $l \in L$. One readily sees that $\Tr_{L/K} \cdot l$ commutes with the $e$-th iteration of Frobenius if and only if $\Tr_{L/K}(lx^{q})=\Tr_{L/K}(l^{q}x^{q})$ for all $x\in L$. Now, since $F_{L/K}$ is surjective, this further implies that $\Tr_{L/K}(lx)=\Tr_{L/K}(l^{q}x)$ for all $x\in L$. The uniqueness of $l$ then yields $l=l^{q}$ whence $l \in \bF_{q} \cap L$. See also \cite[Proposition 4.4]{SchwedeTuckerTestIdealFiniteMaps}.
\end{proof}

\section{Transpositions of Ideals and Tame Ramification}\label{section tame ramification}

We introduce next a notion of tame ramification along a section of the canonical module of a finite cover. We do this via an operation we call \emph{transposition of ideals}. We develop a general theory of transpositions and tame ramification first. As stated in the introduction, our theory is versatile and this general framework is the key to effortlessly amalgamate the notion of (separable) tame ramification with $F$-singularities. Moreover, having such a broad framework lets us treat other (less severe) purely inseparable covers such as cyclic covers by $p$-torsion divisor classes as in \cite{CarvajalFiniteTorsors}, which we do in \autoref{setting workout example} and \autoref{section contracting compatible stuff}. We hope it will find further applications in commutative algebra, algebraic geometry, and number theory.

\begin{notation} 
We let $I(R)$ denote the bounded lattice of ideals of a ring $R$ and $\mathbf{0},\mathbf{1}\in I(R)$ denote the zero and non-proper ideals; respectively. For $\fra \in I(R)$, we let $ Z_{\fra} \coloneqq \bigcup_{\p \in \Ass_R R/\fra } \p$ denote the set of zerodivisors modulo $\fra$ and $W_{\fra} \coloneqq R \setminus Z_{\fra}$ denote the multiplicative subset of regular elements modulo $\fra$. We let $\sK(R)\coloneqq W_{\bm{0}}^{-1} R$ denote the total ring of fractions of $R$. We let $\kappa(\p)$ denote the residue field at $\p \in \Spec R$. If $R$ is an integral domain and $L/\sK(R)$ is a field extension, we denote by $\bar{R}^L$ the normalization of $R$ in $L$. By abuse of notation, given a ring homomorphism $\theta \: R \to S$, we denote by $\frab \cap R$ the contraction $\theta^{-1}(\frab)$ of an ideal $\frab \in I(S)$. Likewise, we write the extension of an ideal $\fra \in I(R)$ to $S$ by $\fra S$. 
\end{notation}

We start off with the notion of transpositions of ideals and non-degeneracy. Our notion of tame ramification is essentially the good behavior of this transposition operation.

\begin{definition}[Transposition of ideals and degeneracy] \label{def.TranspositionOfIdeals}
    Let $\theta \: R \to S$ be a finite homomorphism of rings such that $\ker \theta \subset \sqrt{\mathbf{0}}$ and $f\: Y\to X$ be the corresponding \emph{finite cover}. Moreover, let $\mathfrak{T}$ be an element of the $S$-module $\omega_{f} \coloneqq \omega_{\theta} \coloneqq  \omega_{S/R}\coloneqq \Hom_R(S,R)$. 
    \begin{enumerate}
    \item For $\frab \in I(S)$, we let $\frab_{\frT} \in I(R)$ denote the image of $\frab$ along $\frT$. That is, $\frab_\frT\coloneqq \frT(\frab)$.
    \item For $\fra\in I(R)$, we define its \emph{$\frT$-transpose ideal $\fra^{\frT} \in I(S)$} as the largest ideal of $S$ whose image under $\frT$ is contained in $\fra$. That is, $\fra^{\frT} \coloneqq  \{s \in S \mid \langle s \rangle_{\frT} = \frT(sS) = (\frT\cdot s)(S) \subset \fra\}$.
    \item We say that $\frT$ is \emph{non-degenerate} along $\fra \in I(R)$ if $\mathbf{1}_{\frT}\not\subset Z_{\fra}$, \ie $V(\mathbf{1}_{\frT}) \cap \Ass_{R} R/\fra = \emptyset$. When $\fra=\mathbf{0}$, we simply say that $\frT$ is non-degenerate.
    \item We say that $\frT$ is \emph{non-singular (resp. strongly non-degenerate)} if the $S$-linear map $\sigma\: S \to \omega_{f}$ given by $1\mapsto \frT$ is an isomorphism (resp. injective).
\end{enumerate}
\end{definition}

\begin{remark}[On non-degeneracy] \label{rem.OnDegeneracy}
    With notation as in \autoref{def.TranspositionOfIdeals}, observe that $\frT$ is non-degenerate along $\fra$ if and only if $W^{-1}_{\fra}\frT$ is surjective. Hence, $\frT$ is non-degenerate if and only if $\sK(R) \otimes_R \frT \: \sK(S) \to \sK(R)$ is surjective.\footnote{Since $\theta$ is integral, $\sK(R) \otimes_R S = \sK(S)$.} In particular, if $\frT$ is non-degenerate then the ideal $\mathbf{0}^{\frT} = \ker \sigma$ is contained in the set of zerodivisors of $S$. The converse holds, for instance, if $R$ is an integral domain (\ie $\sK(R)$ is a field).
   
   On the other hand, $\frT$ defines a symmetric $R$-bilinear form on $S$; namely $\frT(-\cdot-)$. Our notion of strong non-degeneracy coincides with such bilinear form being non-degenerate (\ie $\mathbf{0}^{\frT}=\mathbf{0}$). We shall see below in \autoref{pro.BasicPropertiesTranspositions} that our condition of non-degeneracy means $\mathbf{0}^{\frT}\cap R=\mathbf{0}$. These notions coincide if $S$ and $R$ are both integral domains. Noteworthily, our definition of non-singularity is the same as the one from the corresponding bilinear form.
\end{remark}

We now establish the basic properties of transpositions.

\begin{proposition}[Basic properties] \label{pro.BasicPropertiesTranspositions}
With notation as in \autoref{def.TranspositionOfIdeals}, let $W \subset R$ be a multiplicative subset, $\fra \in I(R)$, $\frab \in I(S)$, and subsets $A\subset I(R)$, $B \subset I(S)$. Then:
\begin{enumerate}
\item \label{pro.TransposesLocalization}  
$\fra^{\frT} W^{-1}S = (\fra W^{-1}R)^{W^{-1}\frT}$.
\item $(\bigcap_{\fra \in A} \fra)^{\frT} = \bigcap_{\fra \in A} \fra^{\frT}$ and $(\sum_{\frab \in B} \frab)_{\frT} = \sum_{\frab \in B} \frab_{\frT}$.
 \item $\frab \subset \fra^{\frT}$ if and only if $\frab_{\frT} \subset \fra$.
 \item Consequently: $V(\frab_{\frT}) = \big\{\p \in X \mid \frab \subset \p^{\frT}\big\}$, $\fra^{\frT}\neq \mathbf{1}$  if and only if $\mathbf{1}_{\frT} \not\subset \fra$, and the locus of degeneracy for $\frT$ is $V(\mathbf{1}_{\frT}) = \big\{\p \in X \mid \p^{\frT} = \mathbf{1}\big\}$.
 \item \label{key.Item} $
\fra \subset \fra^{\frT} \cap R = \fra : \mathbf{1}_{\frT}$ and the inclusion is an equality if and only if $\frT$ is non-degenerate along $\fra$.
 \end{enumerate}   
 \end{proposition}
 \begin{proof}
     We only explain (e) as the other statements are direct from the definitions. Note that $\frT(as)=a\frT(s)\in \fra $ for all $a\in \fra$. Hence, $\fra \subset \fra^{\frT} \cap R$ and $\fra^{\frT} \cap R = \fra : \mathbf{1}_{\frT}$.
     In general, $\fra \subset \fra : \frab$ is an equality if and only if $\frab \not\subset Z_{\fra}$. Indeed, if $\frab$ contains an element $w \in W_{\fra}$ then $rw \in \fra$ and so $r \in \fra$ for all $r \in \fra : \frab$. Conversely, suppose that $\frab \subset Z_{\fra}$ so that $\frab$ is contained in an associated prime of $\fra$, say $\fra: x$. Then $x\not\in \fra$ but $x \in \fra:\frab$. 
 \end{proof}

The property of transitivity of transpositions in towers, which we now state,
will be the key to study the behavior of centers of $F$-purity in finite covers (see  \autoref{pro.betaCommutesWithTransposeAndApplications}). 

\begin{proposition}[Transitivity] \label{pro.TransitivityTransposition}
    Let $R \to S \to T$ be a tower of finite covers and let $\frT_{S/R} \in \omega_{S/R}$, $\frT_{T/S}\in \omega_{T/S}$. Set $\frT_{T/R} \coloneqq \frT_{S/R} \circ \frT_{T/S} \in \omega_{T/R}$. Then, $\fra^{\frT_{T/R}} = (\fra^{\frT_{S/R}})^{\frT_{T/S}}$ for all $\fra \in I(R)$.
\end{proposition}
\begin{proof}
 By \autoref{pro.BasicPropertiesTranspositions}, the inclusion ``$\subset$'' is equivalent to $\frT_{T/S}\big(\fra^{\frT_{T/R}} \big) \subset \fra^{\frT_{S/R}}$ and further to $\frT_{S/R}\big(\frT_{T/S}\big(\fra^{\frT_{T/R}} \big)  \big)  \subset \fra$,  which follows from $\frT_{T/R} = \frT_{S/R} \circ \frT_{T/S} $. Conversely, letting $\frab$ denote the ideal on the right hand side of the equality, the inclusion ``$\supset$'' is equivalent to $\frT_{T/R}(\frab) \subset \fra$, which follows from: $\frT_{T/R}(\frab) = \frT_{S/R}\big( \frT_{T/S} (\frab) \big) \subset  \frT_{S/R}\big( \fra^{\frT_{S/R}} \big) \subset \fra$. 
\end{proof}

\begin{notation} \label{notation.Ranks}
   We introduce next some notation for restricting maps $\frT$ on fibers and residue fields. With notation as in \autoref{def.TranspositionOfIdeals}, let $\p \in X$. We will let $Y_{\p}\coloneqq f^{-1}\{\p\} \subset Y$ denote the fiber at $\p$. By incomparability and going-up, the points of $Y_{\p}$ are the minimal/associated primes of $\pS$. We consider the artinian $\kappa(\p)$-algebras: $S_{\kappa(\p)} \coloneqq S \otimes_R \kappa(\p) = S_{\p}/\p S_{\p}$, $\bar{S}_{\kappa(\p)}\coloneqq (S_{\kappa(\p)})_{\mathrm{red}}=S_{\p}/\sqrt{\p S_{\p}}=\prod_{\q \in Y_{\p}} \kappa(\q)$, and $S_{\kappa(\p)}^{\frT} \coloneqq  S_{\p}/\p^{\frT}S_{\p}$. We let $\mu(\p),\rho(\p),\eta(\p) \in \bN$ denote their respective $\kappa(\p)$-dimensions. Of course, $\mu(\p)\geq \rho(\p),\eta(\p)$ as $\p S$ is a sub-ideal of both $\sqrt{\p S}$ and $\p^{\frT}$. When $\sqrt{\pS} \subset \p^{\frT}$ (\eg $\p^{\frT}$ is radical), $\rho(\p) \geq \eta(\p)$ and $\frT_{\kappa(\p)} \coloneqq \frT \otimes \kappa(\p) \: S_{\kappa(\p)} \to \kappa(\p)$ factors as
 \[
\frT_{\kappa(\p)} \: S_{\kappa(\p)} \xrightarrow{\mathrm{can}} \bar{S}_{\kappa(\p)}  \xrightarrow{\bar{\frT}_{\p}} \kappa(\p),
 \]
 which let us define the following $\kappa(\p)$-linear maps for all $\q \in Y_{\p}$:
 \[
 \bar{\frT}_{\q} \:\kappa(\q) \xrightarrow{\mathrm{can}}  \bar{S}_{\kappa(\p)} \xrightarrow{\bar{\frT}_{\p}} \kappa(\p).
 \]
\end{notation}
\begin{remark} \label{rem.BeingExplictAboutT_q}
     We may describe $\bar{\frT}_{\q}$ explictly as follows. Let $s \in \kappa(\q)$, which we may think of as the class modulo $\q S_{\p}$ of an element $s \in S_{\p}$. By the Chinese reminder theorem, there is $s'\in S_{\p}$ such that $s' \equiv s \bmod \q S_{\p}$ and $s'\equiv 0 \bmod \q'S_{\p}$ for all $\q \neq \q'\in Y_{\p}$. Further, any such $s'$ is unique modulo $\sqrt{\p S_{\p}}$. Hence, $\frT_{\p}(s')$ is well-defined modulo $\p$ and $\bar{\frT}_{\q}(s)$ is such a class. In particular, since ``$(s^n)'=(s')^n$,'' we have that $\bar{\frT}_{\q}(s^n)$ equals the class of $\frT(s'^n)$ in $\kappa(\p)$.
\end{remark}

\begin{proposition}[{\cite[Lemma 3.6]{BlickleSchwedeTuckerFSigPairs1}}] \label{prop.OnterpretationOfLength}
   With notation as in \autoref{notation.Ranks}, $\eta(\p)$ is the largest $n \in \bN$ for which there is an exact sequence of $R_{\p}$-modules $S_{\p}\to R_{\p}^{\oplus n}  \to 0$ such that the corresponding $n$ projections $S_{\p} \to R_{\p}$ are in $\frT_{\p} \cdot S_{\p}$. In particular, if $\frT_{\p}$ is non-singular then $\eta(\p)$ is the free-rank of $S_{\p}$ as an $R_{\p}$-module.
\end{proposition}

We describe next the associated primes of the transposition of a prime ideal. 

\begin{proposition}[On associated primes] \label{MinimalPrimesTranspostions}
    Working with notation as in \autoref{def.TranspositionOfIdeals} and \autoref{notation.Ranks}, suppose that $\frT$ is non-degenerate along $\p \in X$. Then:
    \begin{enumerate} 
     \item $\Ass_S S/\p^{\frT} \subset Y_{\p}$ and so $\p^{\frT}$ has no embedded primes.
    \item The minimal primes of $\p^{\frT}$ are those $\q \in Y_{\p}$ for which there is $s \in S$ such that $\frT(s\q)\subset \p$ but $\frT(sS)\not\subset \p$.
    \item If $\sqrt{\pS} \subset \p^{\frT}$, the minimal primes of $\p^{\frT}$ are those $\q \in Y_{\p}$ such that $ \bar{\frT}_{\q} \neq 0$.
    \end{enumerate}
\end{proposition}
\begin{proof}
   Observe that $\p^{\frT} : s = \p^{\frT\cdot s}$. Therefore, $\Ass_S S/\pt = \{\p^{\frT \cdot s} \in I(S) \mid s \in S\} \cap \Spec S$. 
   Note that if $\frT$ is non-degenerate along $\p$ then $\pt$ is proper and so are its associated primes. Since these are of the form $\p^{\frT\cdot s}$, part (a) follows from \autoref{pro.BasicPropertiesTranspositions}.
    \begin{claim}  The ideal
    $\p^{\frT \cdot s}$ is prime if and only if $\mathbf{1}\neq\p^{\frT\cdot s} \supset \q$
    for some $\q\in Y_\p$.
    \end{claim}
    \begin{proof}[Proof of claim]
      By \autoref{pro.BasicPropertiesTranspositions}, $\p^{\frT \cdot s} \cap R$ is a proper ideal if and only so is $\p^{\frT \cdot s}$, in which case it equals $\p$. In other words, $\p^{\frT \cdot s}$ is a proper ideal if and only if it contracts to $\p$. Hence, if $\p^{\frT \cdot s}$ is a prime ideal it lies over $\p$. The forward implication ``$\Longrightarrow$'' then follows. For the converse, suppose that $\p^{\frT \cdot s}$ is a proper ideal and contains a prime $\q$ lying over $\p$. Observe that a minimal prime of $\p^{\frT \cdot s}$ would be of the form $\p^{\frT \cdot ss'}$ for some $s'\in S$ and so it would lie over $\p$. Therefore, it would be equal to $\q$ by incomparability. This means that $\q=\p^{\frT \cdot s}$.
    \end{proof}
    Now, (b) follows after recalling that $\p^{\frT \cdot s}$ is proper exactly when $\frT(sS) \not\subset \p$, and $\q \subset \p^{\frT\cdot s}$ exactly when $\frT(s\q) \subset \p$ (see \autoref{pro.BasicPropertiesTranspositions}).
    
   For (c), take $\q \in Y_{\p}$ such that there is $s \in S$ with $\frT(s\q)\subset \p$ and $\frT(sS)\not\subset \p$. Note that $\frT'\coloneqq \frT \cdot s $ also satisfies that $\sqrt{\p S} \subset \p^{\frT'}$ and so we can also define $\bar{\frT}'_{\p}$ and $\bar{\frT}'_{\q}$. However, since $\frT'(\q) \subset \p$ and $\frT'(S)\not\subset \p$, it follows that $\bar{\frT}'_{\p}$ admits a factorization as follows
\[
\xymatrix{
\kappa(\q) \ar[r]^-{\mathrm{can}} \ar[rd]_-{\id} & \prod_{\q \in Y_{\p}} \kappa(\q) \ar[d]^-{\mathrm{can}} \ar[r]^-{\bar{\frT}'_{\p}} & \kappa(\p)\\
& \kappa(\q) \ar[ru]_-{T\neq 0} &
}.
\]
Further, $T= \bar{\frT}'_{\q}$ so $\bar{\frT}'_{\q} \neq 0$. However, $\bar{\frT}'_{\q} = \bar{\frT}_{\q} \cdot \bar{s}$ where $\bar{s} \in \kappa(\q)$ is the class of $s$, so $\bar{\frT}_{\q} \neq 0$.

Conversely, let us suppose that $\bar{\frT}_{\q} \neq 0$. This means that there is $s \in S$ such that $s \in \q'$ for all $\q'\in Y_{\p} \setminus \{\q\}$ and so that $\frT(s) \notin \p$. In particular, $\frT(sS) \not\subset \p$. Note that $s\q \subset \bigcap_{\q \in Y_{\p}} \q = \sqrt{\pS} $ and so $\frT(s\q) \subset \p$ by hypothesis.
\end{proof}

\begin{corollary}
     With notation as in \autoref{def.TranspositionOfIdeals}, let $\fra \in I(R)$ be a radical ideal. Then, $\frT$ is non-degenerate along $\fra$ if and only if for every minimal prime $\p$ of $\fra$ there is $\q \in Y_{\p}$ such that $\fra^{\frT}\subset \q$. In particular, if $R$ is reduced, $\frT$ is non-degenerate if and only if for every minimal prime of $R$ there is a prime lying over it that contains $\mathbf{0}^{\frT}$.
\end{corollary}
\begin{question}
As mentioned in \autoref{rem.OnDegeneracy}, if $\frT$ is non-degenerate then $\mathbf{0}^{\frT}$ is contained in an associated prime of $S$. Is it possible to characterize the non-degeneracy of $\frT$ as a particular containment of $\mathbf{0}^{\frT}$ in associated primes of $S$ when $R$ is not reduced?     \end{question}

\begin{example}[{\cf \cite[Remark 5.14]{CarvajalStablerFsignaturefinitemorphisms}}] \label{ex.TransposesOFDivisorialIdeals}
    With notation as in \autoref{def.TranspositionOfIdeals}, suppose that $R$ and $S$ are normal integral domains. If $\frT$ is non-degenerate (\ie nonzero),
    we may attach to it an effective divisor $\Ram_{\frT}$ on $\Spec S$ linearly equivalent to $K_{S/R}\coloneqq f^*K_S-K_R$ (see \cite[Definition-Proposition 2.2]{CarvajalSchwedeTuckerEtaleFundFsignature}). For instance, $\frT$ being non-singular means $\Ram_{\frT}=0$. Let $D$ be an effective divisor on $\Spec R$. Then, $R(-D)^{\frT}=S(-D^{\frT})$ where $D^{\frT}$ is the effective part of $D^{*}\coloneqq f^*D-\Ram_{\frT}$. Indeed, for every $0\neq s \in S$, we have the following equivalences:
    \begin{align*}
    s \in R(-D)^{\frT}&\Leftrightarrow \frT \cdot s \in \Hom_R\big(S,R(-D)\big)=\Hom_R\big(S(f^*D),R\big)\\
    &\Leftrightarrow \Ram_{\frT \cdot s} = \Ram_{\frT}+\Div s \geq f^*D  \Leftrightarrow \Div s \geq D^* \Leftrightarrow \Div s \geq D^{\frT} \Leftrightarrow s \in S(-D^{\frT})
    \end{align*}
    For example, if $\frT$ is non-singular then $D^{\frT}=f^*D$. In that case, for $f^*D$ to be reduced if $D$ is reduced we need the ramification indexes over the primes supporting $D$ to be $1$.
\end{example}

\begin{example} \label{ex.TracaConmuteWithFrobenius}
    With notation as in \autoref{def.TranspositionOfIdeals}, suppose that $R$ and $S$ are integral, that the corresponding extension of fields of fractions $L/K$ is separable, and that $R$ is normal. Then, the trace map $\Tr_{L/K} \: L \to K$ restricts to an $R$-linear map $\Tr=\Tr_{S/R} \: S \to R$. Indeed, since $S/R$ is integral and $R$ is normal, the minimal polynomial of every $s \in S$ over $K$ has coefficients in $R$ (\cite[Proposition 5.15]{AtiyahMacdonald}). However, $\Tr_{L/K}(s)$ is an integer multiple of those coefficients (\cite[\href{https://stacks.math.columbia.edu/tag/0BIH}{Tag 0BIH}]{stacks-project}). Moreover, $\Tr(s) \in \sqrt{\fra}$ if $s \in \sqrt{\fra S}$ and so $\sqrt{\fra S} \subset \sqrt{\fra}^{\Tr}$ for all $\fra \in I(R)$ (use \cite[Propositions 5.14 and 5.15]{AtiyahMacdonald}). If $\Char K = p$, the above implies that the properties of \autoref{second criterion separability} also hold for $\Tr$, namely 
\[
\{\frT \in \Hom_R(S,R) \mid \frT(s^q)=\frT(s)^q,\forall s \in S\}=\Tr \cdot (\bF_{q}\cap S).
\]
\end{example}
\begin{lemma} \label{lem.CommWithFrobmodP}
With notation as in \autoref{ex.TracaConmuteWithFrobenius}, let $\p \in X$ be such that $\Char \kappa(\p)=p$. Then, $\Tr(s^p) \equiv \Tr(s)^p \bmod \p$ for all $s\in S$.
\end{lemma}
\begin{proof}
    Let $\bar{K}/K$ be a separable closure of $K$ and fix an embedding $L\subset \bar{K}$. Let $M \subset \bar{K}$ be the Galois closure of $L/K$ in $\bar{K}$.\footnote{That is, $M$ is the intersection of all Galois finite extensions of $K$ inside $\bar{K}$ containing $L$.} Then, $\Tr_{M/K} = \sum_{\sigma \in G} \sigma$ where $G \coloneqq \Gal(M/K)$. Letting $H \subset G$ be the stabilizer of $L$, we have $\Tr_{L/K} = \sum_{\sigma \in G/H} \sigma$. Let $T\coloneqq \bar{R}^M$. Then, for $s \in S\subset T$, we have the following in $T$:
    \[
    \Tr(s^p) = \sum_{\sigma \in G/H} \sigma(s^p) = \sum_{\sigma \in G/H} \sigma(s)^p \equiv \Bigg(\sum_{\sigma \in G/H} \sigma(s) \Bigg)^p = \Tr(s)^p \bmod pT.
    \]
    Since $p \in \p$, then $pT \subset \p T $ and so $\Tr(s^p) - \Tr(s)^p \in (\p T) \cap R = \p$.
\end{proof}

Our next task is to compare transpositions with the radical of extensions. Shockingly, the next result establishes opposite inclusions in the separable and purely inseparable cases.

\begin{proposition} \label{pro.SeparableVSpurelyInseparable}
With notation as in \autoref{def.TranspositionOfIdeals}, suppose that $\bF_p \subset R$. Then:
    \begin{enumerate}
 \item Suppose that $\frT$ commutes with $F^e$. If $\fra\in I(R)$ is radical then so is $\fra^{\frT}$ and so $\sqrt{\fra S} \subset \fra^{\frT}$. Moreover, for all $\p \in X$ and $\q \in Y_{\p}$, the induced map $\bar{\frT}_{\q} \: \kappa(\q) \to \kappa(\p)$ commutes with $F^e$. In particular, if $\bar{\frT}_{\q} \neq 0$ then $\kappa(\q)/\kappa(\p)$ is separable and $\bar{\frT}_{\q}=\Tr_{\kappa(\q)/\kappa(\p)} \cdot u$ for some unique $0\neq u\in \bF_q \cap \kappa(\q)$.  
 \item Suppose that $S^q \subset R \subset S$. Then, for all $\fra \in I(R)$ radical, $\fra^{\frT} \subset \sqrt{\fra S}$ if and only if $\frT$ is non-degenerate along $\fra$.
\end{enumerate}
\end{proposition}

\begin{proof}
Part (a) follows from noting that $\sqrt{\fra^{\frT}} \subset \sqrt{\fra}^{\frT}$ for all $\fra \in I(R)$. To see this inclusion, let $s\in \sqrt{\fra^{\frT}}$. Then, there is $n\gg0$ such that $s^{q^n}\in \fra^{\frT}$. In particular, $\frT(s)^{q^n}=\frT(s^{q^n})\in \fra$ and so $\frT(s)\in \sqrt{\fra}$. That is, $\big(\sqrt{\fra^{\frT}}\big)_{\frT} \subset \sqrt{\fra}$. To see why $\bar{\frT}_{\q}$ commutes with $F^e$, use \autoref{rem.BeingExplictAboutT_q}. 

For (b), let $x\in \fra^{\frT}$. Then, $x^q\in \fra^\frT\cap R$ which is contained in $\fra$ if $\frT$ is non-degenerate along $\fra$ by \autoref{pro.BasicPropertiesTranspositions}. This shows ``$\Longleftarrow$.'' Conversely, let $\fra$ be a radical ideal such that $\fra^{\frT} \subset \sqrt{\fra S}$. Contracting to $R$ yields $\fra^{\frT} \cap R \subset \sqrt{\fra}=\fra$, so we are done by \autoref{pro.BasicPropertiesTranspositions}.
\end{proof}

\begin{proposition} \label{pro.RadicalityAndObviousTameRamification}
    With notation as in \autoref{ex.TracaConmuteWithFrobenius}, let $\fra \in I(R)$ be a radical ideal. Then, $\fra^{\Tr}$ is radical. Moreover, for every $\p \in \Ass_R R/\fra$, the following statements hold:
    \begin{enumerate}
        \item If $\Char \kappa(\p) = p$ then $\Bar{\Tr}_{\q} \: \kappa(\q) \to \kappa(\p)$ commutes with Frobenius for all $\q\in Y_{\p}$.  In particular, if $\bar{\Tr}_{\q} \neq 0$ then $\kappa(\q)/\kappa(\p)$ is separable and $\bar{\Tr}_{\q}=\Tr_{\kappa(\q)/\kappa(\p)} \cdot u$ for some unique $0\neq u\in \bF_q \cap \kappa(\q)$.
        \item If $\Char \kappa(\p)$ is either $0$ or $>[L:K]$ then $0\neq \Bar{\Tr}_{\q}$ for all $\q\in Y_{\p}$ and so $\p^{\Tr}=\sqrt{\p S}$.
    \end{enumerate}
    In particular, if $\Char \kappa(\p)$ is either $0$ or $>[L:K]$ for all $\p \in \Ass_R R/\fra$ then $\fra^{\frT}=\sqrt{\fra S}$.
\end{proposition}
\begin{proof}
    We may assume that $\fra = \p$ is a prime ideal. Suppose first that $\Char\kappa(\p)= p$. Proving that $\p^{\Tr}$ is radical boils down to $\Tr\big(\sqrt{\p^{\Tr}}\big) \subset \p$. That is, we must show that $\Tr(s) \equiv 0 \bmod \p$ for all $s \in \sqrt{\p^{\Tr}}$. Take $e \gg 0$ such that $s^q \in \p^{\Tr}$, so that $\Tr(s^q) \equiv 0 \bmod \p$. By \autoref{lem.CommWithFrobmodP}, $\Tr(s)^{q} \equiv 0 \bmod \p$ and so $\Tr(s) \equiv 0 \bmod \p$ (for $\p$ is radical). To see why $\Bar{\Tr}_{\q}$ commutes with $F$, use \autoref{rem.BeingExplictAboutT_q} and\autoref{lem.CommWithFrobmodP}. The rest of part (a) then follows from \autoref{second criterion separability}.

    It remains to show that if $\Char \kappa(\p)$ is either $0$ or $>[L:K]$ then $\p^{\Tr} = \sqrt{\p S}$. This follows from the following claim and \autoref{MinimalPrimesTranspostions}:
    \begin{claim}[\cf{\cite[Proposition 10]{SpeyerFrobeniusSplit}}]
        $\bar{\Tr}_{\q}(1) = n \in \kappa(\p)$ where $n \in \bN_{+}$ is $\leq [L:K]$.
    \end{claim}
    \begin{proof}[Proof of claim]
       Let $s \in S_{\p}$ be $1'$ as in \autoref{rem.BeingExplictAboutT_q}. Let $T$ be as in the proof of \autoref{lem.CommWithFrobmodP} and set $V=\Spec T$. For every $\mathfrak{r} \in V_{\p}$, it follows that:
        \[
        s \equiv \begin{cases}
            1 \bmod \mathfrak{r} & \text{ if } \mathfrak{r} \cap S = \q,\\
            0 \bmod \mathfrak{r} & \text{ if } \mathfrak{r} \cap S \neq \q.
        \end{cases}
        \]    
        Fix $\mathfrak{r} \in V_{\p}$, 
        so that $\{\sigma^{-1}(\mathfrak{r}) \mid \sigma \in G\}=V_{\p}$. Thus, for all $\sigma \in G$:
        \[
        s \equiv \begin{cases}
            1 \bmod  \sigma^{-1}(\mathfrak{r}) & \text{ if } \sigma^{-1}(\mathfrak{r}) \cap S = \q,\\
            0 \bmod  \sigma^{-1}(\mathfrak{r}) & \text{ if } \sigma^{-1}(\mathfrak{r}) \cap S \neq \q,
        \end{cases}
\text{ and so }
        \sigma(s) \equiv \begin{cases}
            1 \bmod  \mathfrak{r} & \text{ if } \sigma^{-1}(\mathfrak{r}) \cap S = \q,\\
            0 \bmod  \mathfrak{r} & \text{ if } \sigma^{-1}(\mathfrak{r}) \cap S \neq \q.
        \end{cases}
        \]    
Let us further assume that $\mathfrak{r}$ lies over $\q$. Then, $\Tr_{\p}(s)= \sum_{\sigma \in G/H}\sigma(s) \equiv n \bmod \mathfrak{r}$, where $n \in \bN$ is the cardinality of $\{\sigma \in G/H \mid \sigma^{-1}(\mathfrak{r}) \cap S = \q\}$. Observe that $n\neq 0$ as $\sigma \in H$ is such that $\sigma^{-1}(\mathfrak r) \cap S = \q$. Moreover, $n \leq [G:H] = [L:K]$. Consequently, $
        \Tr_{\p}(s)-n \in \mathfrak{r} \cap R = \p$.
    \end{proof}
This concludes the proof.
\end{proof}

With the above in place, we are ready to introduce our trace-theoretic notion of tame ramification and its first basic properties.

\begin{definition}[Tame ramification]\label{definition tamely T ramified}
   With notation as in \autoref{def.TranspositionOfIdeals}, $S/R$ is \emph{tamely $\frT$-ramified over $\p$} if $\p^{\frT} = \sqrt{\p S}$. If $\frT=\Tr$ as in \autoref{ex.TracaConmuteWithFrobenius}, we say that $S/R$ is \emph{tamely ramified over $\p$}.
\end{definition}

\begin{corollary}\label{cor.DifferentWaysToThinkOfTameRamification}
 With notation as in \autoref{def.TranspositionOfIdeals} and \autoref{MinimalPrimesTranspostions}, then the following are equivalent:
\begin{enumerate}
        \item $\theta \: R \to S$ is tamely $\frT$-ramified over $\p \in \Spec R$.
        \item $\theta_{\p} \: R_{\p} \to S_{\p}$ is tamely $\frT_{\p}$-ramified over $\p R_{\p}$.
        \item $\p^{\frT}$ is radical and $\bar{\frT}_{\q} \neq 0$ for all $\q \in Y_{\p}$.
        \item $\frT(\sqrt{\pS}) \subset \p$ and $\bar{\frT}_{\q} \neq 0$ for all $\q \in Y_{\p}$.
        \item $\frT(\sqrt{\pS}) \subset \p$ and  for all $\q \in Y_{\p}$ there is $s \in \bigcap_{\q' \in Y_{\p} \setminus \{\q\}} \q'$ such that $\frT(s) \notin \p$.
         \item $\frT(\sqrt{\pS}) \subset \p$ and there is an exact sequence of $R_{\p}$-modules $S_{\p}\to R_{\p}^{\oplus \rho(\p)}  \to 0$ such that the $n$ projections $S_{\p} \to R_{\p}$ belong to $\frT_{\p}\cdot S_{\p}$.
\end{enumerate}
In that case, $\frT(S)\not\subset \p$,  \ie $\frT_{\p} \: S_{\p} \to R_{\p}$ is surjective.
\end{corollary}
\begin{proof}
The equivalence between (a), (c), (d), and (e) is direct from \autoref{MinimalPrimesTranspostions}. For (f), use \autoref{prop.OnterpretationOfLength}. Since transpositions (and radicals) commute with localizations (see \autoref{pro.BasicPropertiesTranspositions}), the equivalence between (a) and (b) means that the equality $\pt=\sqrt{\p S}$ holds if (and only if) it does after localizing at $\p$. Suppose that $\pt S_{\p}=\sqrt{\p S} S_{\p}$. Since the contraction of $\sqrt{\p S} S_{\p}$ back to $S$ is $\sqrt{\p S}$, it suffices to show that $(\pt S_{\p}) \cap S = \pt$. Note that $(\pt S_{\p}) \cap S = \bigcup_{r \in R\setminus \p} (\pt \colon r) = \bigcup_{r \in R\setminus \p} \p^{\frT \cdot r} = \pt$ as $\p^{\frT \cdot r} = \pt$ for all $r\in R \setminus \p$.
\end{proof}

\begin{remark} \label{rem.TheGaloisCase}
    There is a case in which the surjectivity of $\frT_{\p}$ in \autoref{cor.DifferentWaysToThinkOfTameRamification} implies any (equivalently all) of the conditions in \autoref{cor.DifferentWaysToThinkOfTameRamification}. Namely, for $\frT=\Tr$ as in \autoref{ex.TracaConmuteWithFrobenius} where $S$ is further normal and $L/K$ is Galois. Indeed, since $\Gal(L/K)$ acts transitively on the primes of $S$ lying over $\p$ by $R$-algebra automorphisms of $S$, it follows that $\p^{\Tr}$ is either supported everywhere or nowhere over $\p$. More precisely, either $\p^{\Tr} = S$ or $\p^{\Tr} = \sqrt{\pS}$. Therefore, tame $\Tr$-ramification is equivalent to $\Tr$ being non-degenerate along $\p$. We will elaborate more on this in \autoref{subsection.TameRamificationSeparableCase} below; see for instance \autoref{prop.DifferentNOTIONStameRamificationGaloisCase}.
\end{remark}

Next, we explain how tame ramification works in towers.

\begin{proposition}[Transitivity]\label{pro.TransitivityTameRamificationGeneralCase}
    With notation as in \autoref{pro.TransitivityTransposition}, let $\p \in \Spec R$. Then, $T/R$ is tamely $\frT_{T/R}$-ramified over $\p$ if the following two conditions hold:
    \begin{enumerate}
        \item $S/R$ is tamely $\frT_{S/R}$-ramified over $\p$, and
        \item $T/S$ is tamely $\frT_{T/S}$-ramified over every prime of $S$ lying over $\p$.
    \end{enumerate}
    Conversely, suppose that $T/R$ is tamely $\frT_{T/R}$-ramified over $\p$. Then, the ideal $\p^{\frT_{S/R}}$ is radical if and only if $\frT_{T/S}$ is non-degenerate along it. In case one and so both of them hold, then both conditions (a) and (b) above hold. In particular, if (a) holds then: $T/R$ is tamely $\frT_{T/R}$-ramified over $\p$ if and only if (b) holds.
\end{proposition}
\begin{proof}
    If (a) and (b) hold, then $T/R$ is tamely $\frT_{T/R}$-ramified over $\p$ by \autoref{pro.TransitivityTransposition} and \autoref{pro.BasicPropertiesTranspositions}. Indeed, setting $V \coloneqq \Spec T$: 
    \[
\p^{\frT_{T/R}}=\big(\p^{\frT_{S/R}}\big)^{\frT_{T/S}}= \Bigg(\bigcap_{\q \in Y_{\p}} \q\Bigg)^{\frT_{T/S}} = \bigcap_{\q \in Y_{\p}} \q^{\frT_{T/S}} = \bigcap_{\q \in Y_{\p}} \bigcap_{\mathfrak{r}\in V_{\q}} \mathfrak{r} = \bigcap_{\mathfrak{r}\in V_{\p}} \mathfrak{r}= \sqrt{\p T}.
\]
For the converse, suppose that $\sqrt{\p T}=\p^{\frT_{T/R}}=\big(\p^{\frT_{S/R}}\big)^{\frT_{T/S}} $. Contracting to $S$ yields
\[
\pS \subset \p^{\frT_{S/R}} \subset \p^{\frT_{S/R}} : \mathbf{1}_{\frT_{T/S}} = \sqrt{\p S}.
\]
By \autoref{pro.BasicPropertiesTranspositions}, $\frT_{T/S}$ is non-degenerate along $\p^{\frT_{S/R}}$ if and only if $\p^{\frT_{S/R}}$ is radical, in which case (a) and (b) hold.  
\end{proof}

The above conclude our general remarks on tame ramification. For the rest of this section, we present an instructive example that exhibits how tame ramification behaves with respect to separability or lack thereof (see \autoref{setting workout example}). Then, we explain how in the separable case our notion sits in between other notions of tame ramification in arithmetic geometry (see \autoref{subsection.TameRamificationSeparableCase}).

\subsection{A workout example: diagonalizable quotients}\label{setting workout example}
Fix a base field $\kay$ of characteristic $p>0$. Let $\Gamma$ be a finite abelian group and $S=\bigoplus_{\gamma \in \Gamma} S_{\gamma}$ be a $\Gamma$-graded $\kay$-algebra. Set $R \coloneqq S_0$. This is tantamount to an action of the diagonalizable group-scheme $G\coloneqq D(\Gamma) \coloneqq \Spec \kay[\Gamma]$ on $\Spec S$ where $R=S^G$ is the ring of invariants; see \cite[Example 1.6.7]{MontgomeryHopfAlgebras}. Denote by $\pi_{\gamma} \: S \to S_{\gamma}$ the canonical projection and set $\frT\coloneqq \pi_0 \: S \to R$. Suppose that $S_{\gamma}$ is a finitely generated $R$-module for all $\gamma \in \Gamma$, so that $S/R$ is a finite extension.\footnote{In general, it holds that $S/R$ is integral and so quasi-finite \cite[\S4.2]{MontgomeryHopfAlgebras}. Also see \cite[Remark 2.3]{CarvajalFiniteTorsors}.} Notably, $\frT$ is the trace map $\Tr_{S/R}\:S \to S^G$ constructed in \cite[\S3.1]{CarvajalFiniteTorsors}. Our goal here is to describe $\p^{\frT}$ for all $\p \in R$ and so tame $\frT$-ramification. Note that $\p^{\frT}$ is a proper ideal as $\frT$ is surjective. 

\begin{remark}
Observe the following:
\begin{enumerate}
    \item Suppose that $p\nmid \vert\Gamma \vert$. Then, $\cdot p\colon \Gamma\to \Gamma$ is an isomorphism and $\gamma/p \in \Gamma$ is well-defined for all $\gamma\in \Gamma$. Moreover, $\pi_{\gamma}(s^p)=\pi_{\gamma/p}(s)^p$ for all $\gamma \in \Gamma$. In particular, setting $\gamma=0$ yields $\frT(s^p)=\frT(s)^p$. Therefore, $\p^{\frT}$ is radical and so $\sqrt{\p S} \subset \p^{\frT}$.
    \item If $\Gamma$ is a $p$-group of order $q$, then $S^{q}\subset R\subset S$. Therefore, $\p^{\frT}\subset \sqrt{\p S}$. Note that $\q \coloneqq \sqrt{\p S}$ is the only prime ideal of $S$ lying over $\p$.
\end{enumerate}
\end{remark}

\begin{definition}
   With notation as above, the \emph{homogeneous part} of $\frab \in I(S)$ is the subideal $\frab_{\mathrm{h}} \coloneqq \bigoplus_{\gamma \in \Gamma} \frab \cap S_{\gamma} \subset \frab$. We say that $\frab$ is \emph{homogeneous} if $\frab_{\mathrm{h}} = \frab$, \ie $\pi_{\gamma}(\frab) \subset \frab$ for all $\gamma \in \Gamma$. 
\end{definition}

One readily sees that $\frab \in I(S)$ is homogeneous if and only if it is generated by homogeneous elements and if and only if the quotient $S \to S/\frab$ (or rather its spectrum) is $G$-equivariant. 

\begin{proposition} \label{transpose is equivariant}
    With notation as above, let $\p \in \Spec R$. Then, $\pt$ is homogeneous and $(\sqrt{ \p S})_{\mathrm{h}} \subset \p^{\frT}$. If $\Gamma$ is a $p$-group, then $\p^{\frT} = (\sqrt{ \p S})_{\mathrm{h}}$ and so: $S/R$ is tamely $\frT$-ramified over $\p$ if and only if $\p$ is the contraction of a homogeneous ideal and if and only if $\p^{\frT}$ is radical.
\end{proposition}
\begin{proof}
Let $s \in \pt$. We must show that $s_{\gamma} \coloneqq \pi_{\gamma}(s) \in \pt$ for all $\gamma \in \Gamma$. For this, it suffices to show that $\frT(s_{\gamma}x) \in \p$ for all homogeneous elements $x \in S$, say of degree $\gamma'$. Note that $\frT(s_{\gamma} x)=0$ unless $\gamma'= -\gamma$, so we may assume that $\gamma'= - \gamma$. Moreover, $\frT(sx)=\frT(s_{\gamma}x)$, which is in $\p$ as $s \in \pt$. The containment $(\sqrt{ \p S})_{\mathrm{h}} \subset \p^{\frT}$ follows from the facts that $\sqrt{\p S} \cap R= \p$ and $\frT$ is the projection onto $R=S_0$. The rest then follows.
\end{proof}

\begin{remark}
    By \cite[\S3.2]{CarvajalFiniteTorsors}, $S/R$ (or rather its spectrum) is a $G$-torsor over $\p \in \Spec R$ if and only if each $(S_{\gamma})_{\p}$ is a free $R_{\p}$-module of rank $1$ and $(\pi_{\gamma})_{\p} \: S_{\p} \to R_{\p}$ is an $S_{\p}$-multiple of $\frT_{\p}$ for all $\gamma \in \Gamma$. This is the case precisely when the multiplication map $S_{\gamma} \otimes_R S_{-\gamma} \to R$ is surjective at $\p$ (\ie its image avoids $\p$) for all $\gamma \in \Gamma$ (\cf \cite[Theorem 8.1.7]{MontgomeryHopfAlgebras}).
\end{remark}

It is difficult to say much more about $\p^{\frT}$ in this generality. Hence, we are going to assume that $R$ is a normal integral domain with field of fractions $K$ and arrange for $S/R$ to be a $G$-torsor over all $\p \in \Spec R$ of height $\leq 1$ (aka $G$-quasitorsor). Let us assume that each $S_{\gamma}$ is a reflexive $R$-module of generic rank $1$. Since $R$ is normal, this is to say that $S$ is an $\textbf{S}_2$ ring. In particular, we may take $S_\gamma = R(D_\gamma)$ for some divisor $D_\gamma$ on $\Spec R$ with $D_{0}=0$ and let the ring structure on $S$ be encoded into $R$-linear maps:
\[
R(D_{\gamma}) \otimes_R R(D_{\gamma'}) \xrightarrow{\mathrm{can}} R(D_{\gamma} + D_{\gamma'}) \xrightarrow{\cong, \cdot \kappa_{\gamma,\gamma'}} R(D_{\gamma+\gamma'}),
\]
for some $\kappa_{\gamma,\gamma'} \in K^{\times}$ such that $\Div \kappa_{\gamma,\gamma'} + D_{\gamma+\gamma'} = D_{\gamma}+ D_{\gamma'}$. 

Observe that the above defines a homomorphism $\Gamma \to \Cl R$. For every $\p \in \Spec R$, let us define $\Gamma_{\p} \subset \Gamma$ as the kernel of $\Gamma \to \Cl R \to \Cl R_{\p}$. That is, $\Gamma_{\p} = \{\gamma \in \Gamma \mid D_{\gamma} \text{ is Cartier at }\p\}$. Let us consider the subring $R \subset S^{(\p)} \subset S$ given by $S^{(\p)}\coloneqq\bigoplus_{\gamma \in \Gamma_{\p}} S
_{\gamma}$. Note that, for each $\gamma \in \Gamma_{\p}$ we may find $\delta_{\gamma} \in K^{\times}$ such $\Div \delta_{\gamma} \geq D_{\gamma}$ with equality at those prime divisors going through $\p$. In particular, for all $\gamma,\gamma'\in \Gamma_{\p}$ there is $u_{\gamma,\gamma'} \in R_{\p}^{\times}$ such that $\kappa_{\gamma,\gamma'} \delta_{\gamma+\gamma'} = u_{\gamma,\gamma'} \delta_{\gamma} \delta_{\gamma'}$. Let $T_{\p}$ be the $\Gamma_{\p}$-graded $R_{\p}$-algebra $\bigoplus_{\gamma \in \Gamma} R_{\p} \cdot t^{\gamma}$ subject to $t^{\gamma}\cdot t^{\gamma'} \coloneqq u_{\gamma,\gamma'}t^{\gamma+\gamma'}$. Then, $1/\delta_{\gamma} \leftrightarrow t^{\gamma}$ establishes an isomorphism of $\Gamma_{\p}$-graded $R_{\p}$-algebras between $T_{\p}$ and the localization of $S^{(\p)}$ at $\p$. Hence, we may write $R_{\p} \subset T_{\p} \subset S_{\p}$. Observe that $T_{\p}/R_{\p}$ is a $D(\Gamma/\Gamma_{\p})$-torsor and $S/R$ is a $G$-torsor over $\p$ if and only if $T_{\p}=S_{\p}$ (\ie $\Gamma_{\p}=\Gamma$). Then, \cite[Corollary 3.13]{CarvajalFiniteTorsors} yields:

\begin{proposition} \label{pro.QuasitorsoNonSingularity}
    With notation as above, $S/R$ is a $G$-quasitorsor and $\frT$ is non-singular. 
\end{proposition}

\begin{corollary}\label{transpose in divisorial stuff} With notation as above, let $\p\in \Spec R$ and write $S_{\p}=T_{\p}\oplus M$. Then, $\p T_{\p}\oplus M = \p^{\frT}S_{\p}$ and so $\eta(\p) = |\Gamma_{\p}|$. In particular:
\begin{enumerate}
    \item $S/R$ is a $G$-torsor over $\p$ if and only $\p S_{\p} = \pt S_{\p}$. In that case:
    \begin{enumerate}
        \item If $p\nmid|\Gamma|$, then $S/R$ is tamely $\frT$-ramified over $\p$ and $\p S_{\p} = \sqrt{\p S_{\p}} = \pt S_{\p}$.
        \item If $\Gamma$ is a $p$-group, $S/R$ is tamely $\frT$-ramified over $\p$ if and only if $\bigoplus_{\gamma \in \Gamma}\kappa(\p)\cdot t^{\gamma}$ is a field; where $t^{\gamma}\cdot t^{\gamma'} = \bar{u}_{\gamma,\gamma'}t^{\gamma+\gamma'}$ with $\bar{u}_{\gamma,\gamma'} \in \kappa(\p)$ the class of $u_{\gamma,\gamma'} \in R_{\p}^{\times}$.
    \end{enumerate}
    
    \item If $\Gamma_{\p} = \{0\}$ then $S_{\p}$ is a local ring with maximal ideal $\pt S_{\p}$ and residue field $\kappa(\p)$ and so $S/R$ is tamely $\frT$-ramified over $\p$. 
\end{enumerate}
\end{corollary}
\begin{proof}
    By \autoref{pro.QuasitorsoNonSingularity}, $\eta(\p)$ is the free-rank of $S_{\p}$ as an $R_{\p}$-module. In particular, since $T_{\p}$ is a direct summand of $S_{\p}$, we conclude that $\eta(\p) \geq |\Gamma_{\p}|$. Observe that $\p T_{\p} \oplus M \subset S_{\p}$ is an ideal. Indeed, this is to say that $R(D_{\gamma}) \otimes_R R(D_{\gamma'}) \to R(D_{\gamma+\gamma'})$ is not surjective at $\p$ for all $\gamma,\gamma' \in \Gamma \setminus \Gamma_{\p}$ such that $\gamma+\gamma'\in \Gamma_{\p}$. Therefore, $\p T_{\p}\oplus M \subset \p^{\frT}S_{\p}$ and the equality follows by comparing the $\kappa(\p)$-dimensions of their quotients. 
\end{proof}

\begin{question} \label{que.Computep^TForGneralQuasiTorsors}
    How can \autoref{transpose in divisorial stuff} be generalized to the following setup: $G/\kay$ is a finite group-scheme, $S/R$ is a finite $G$-quasitorsor between $\mathbf{S}_2$ integral $\kay$-domains, and $\frT \: S \to R$ is the (non-singular) trace map of \cite{CarvajalFiniteTorsors}?
\end{question}

\subsubsection{Cyclic case} Let $\Gamma$ be cyclic of order $m$ and set $D_{i}\coloneqq iD$ for $i=0,\ldots,m-1$ where $D$ is a divisor on $\Spec R$ such that $\Div \kappa = m D$ for some $\kappa \in K^{\times}$. Suppose that $m=nk$ for $n,k\in \bN_{+}$, so that $\Div \kappa=nE$ with $E \coloneqq kD$. Then, we obtain a factorization of $R$-algebras $R\subset T \subset S$ where $T=\bigoplus_{i=0}^{n-1}R(iE)$ is the cyclic cover given by $\Div \kappa = nE$. Moreover, if $E=\Div \epsilon$ for some $\epsilon \in K^{\times}$, it then follows that $S=T[t]/\langle t^n-u \rangle$ where $u \coloneqq \kappa/\epsilon^n \in R^{\times}$. In particular, by letting $k$ be the order of $\Lambda \coloneqq \ker(\Gamma \to \Cl R)$ (\ie $k$ is the index of $D\in \Cl R$), we may factor the general cyclic cover case $R \subset S$ as $R \subset T \subset S$ where $T/R$ is a cyclic cover such that $\Lambda \to \Cl R$ is injective and $S/T$ is a $D(\Gamma/\Lambda)$-torsor $S=T[t]/\langle t^n-u \rangle$ of degree $n=[\Gamma: \Lambda]$ for some unit $u \in R^{\times}$.

\begin{corollary}\label{tame t ramification in cyclic covers}
    With notation as above, if $p\nmid m$ then $S/R$ is tamely $\frT$-ramified everywhere. If $m=q$, then $S/R$ is tamely $\frT$-ramified over $\p \in \Spec R$ if and only if the following holds: the index of $D$ in $\Cl R_{\p}$ is either $q$ or: $q'=q/q''<q$ and if $q'D_{\p}=\Div_{R_{\p}} \delta$ then the residual class $\bar{u}\in \kappa(\p)$ of $u \coloneqq \kappa/\delta^{q''} \in R_{\p}^{\times}$ is not in $\kappa(\p)^p$. In that case, $S_{\kappa(\p)}=\kappa(\p)(\bar{u}^{1/q''})=\kappa(\q)$.
\end{corollary}
\begin{proof}
In general, let $k \mid m$ be the index of $D$ in $\Cl R_{\p}$. As explained above, we then obtain a factorization $R_{\p} \subset T \subset S_{\p}$ where $T/R_{\p}$ (resp. $S_{\p}/T$) is as in (b) (resp. (a)) in \autoref{transpose in divisorial stuff}. Therefore, applying \autoref{pro.TransitivityTameRamificationGeneralCase} yields that $S/R$ is tamely $\frT$-ramified over $\p$ if and only if $S_{\p}/T$ is tamely $\frT$-ramified over the maximal ideal of $T$. If $p\nmid m$, this is always the case as established in \autoref{transpose in divisorial stuff}. If $m=q$, we need to look at when the special fiber $\kappa(\p)[t]/\langle t^{q''}-u\rangle$ of $S_{\p}/T$ is a field, which is the case exactly when either $q=q'$ or $\bar{u} \in \kappa(\p)$ has no $p$-th roots (\cite[VI, \S6, Theorem 9.1]{LangAlgebra}).
\end{proof}

\begin{example} \label{does not make sense to require tame ram everywhere}
We may impose a $\bZ/q$-grading on $S=\kay[x_1,\ldots,x_n]$ by declaring $S_i$ to be set of homogeneous polynomials of degree $i \bmod q$. In particular, $R=S_{0}$ is the $q$-th Veronese subring of $S$. As a matter of fact, $\Cl R \cong \bZ/q$ and $S$ is obtained as the corresponding cyclic cover. Moreover, $R$ is factorial everywhere except at $\fram\coloneqq \langle x_1, \ldots, x_n \rangle \cap R = \langle x_1^{i_{1}} \cdots x_n^{i_n} \mid i_1 + \cdots + i_n=q \rangle$. In particular, $S/R$ is tamely $\frT$-ramified over $\fram$. However, if $\fram \neq \p \in \Spec R$, then $S_{\p} \cong R_{\p}[t]/\langle t^q-u \rangle$ for some $u \in R_{\p}^{\times}$. Then, tame $\frT$-ramification over $\p$ is determined by whether or not $u \in \kappa(\p)$ has a $p$-th root and this, in turn, is determined by whether or not the prime ideal $\sqrt{\p S}$ is generated by $\bZ/q$-homogeneous elements in $S$. 
\end{example}

\subsection{Tame ramification in separable normalizations} \label{subsection.TameRamificationSeparableCase}

We restrict our attention now to the classical setup of ramification theory. In this subsection, we always work in:

\begin{setting} \label{setup.ClassicAlgebraicNumberTheorySetup}
    Let $R$ be a normal integral domain with field of fractions $K$ and let $L/K$ be a finite separable extension inside a fixed separable closure $\bar{K}$ of $K$. Set $S\coloneqq \bar{R}^L$. In particular, \autoref{ex.TracaConmuteWithFrobenius} applies here and $\Tr_{L/K} \: L \to K$ restricts to an $R$-linear map $\Tr=\Tr_{S/R} \: S \to R $ and the transpositions of radical ideals are radical (\autoref{pro.RadicalityAndObviousTameRamification}). Fix $\p \in \Spec R$.  We keep using \autoref{notation.Ranks}. We let the Galois closure of $L/K$ be the intersection of all the Galois finite extensions of $K$ in $\bar{K}$ that contain $L \subset \bar{K}$.
\end{setting}

Our first task is to verify that our notion is indeed a notion of tame ramification in higher dimensions. The following results establish that.

\begin{proposition} \label{prop.TameRamificationParticularCases} The following statements hold: 
    \begin{enumerate}
        \item If $\height \p = 1$, then $S/R$ is tamely ramified over $\p=R(-P)$ if and only if $S_{\p}/R_{\p}$ is a tamely ramified extension of Dedekind domains in the classical sense \cite[Ch. 2]{GrothendieckMurreTameFundamentalGroup}.
        \item If $\Char \kappa(\p)$ is either $0$ or $>[L:K]$ then $S/R$ is tamely ramified over $\p$ \cite{SpeyerFrobeniusSplit}.
        \item If $S/R$ is tamely ramified over $\p$ then $\kappa(\q)/\kappa(\p)$ is finite separable for all $\q \in Y_{\q}$.
    \end{enumerate}
\end{proposition}
\begin{proof}
    Parts (b) and (c) are direct consequences of \autoref{pro.RadicalityAndObviousTameRamification}. For (a), recall that $R(-P)^{\Tr}=S(-P^{\Tr})$ where $P^{\Tr}$ is the effective part of $f^*P-\Ram$ and $\Ram=\Ram_{\Tr}$ is the ramification divisor of $S/R$ (\autoref{ex.TransposesOFDivisorialIdeals}). Let $Q_1,\ldots,Q_k$ be the prime divisors on $S$ lying over $P$, so that $\sqrt{\p S}=S(-Q_1-\cdots - Q_k)$. Then $\sqrt{\p S} \subset \p^{\frT}$ means that $Q_1 + \cdots + Q_k \geq P^{\Tr} \geq f^*P-\Ram$. Equivalently, $\Ram \geq \sum_{i=1}^k (e_i-1)Q_i$ where $e_i$ is the ramification index of the extension of DVRs $S_{Q_i}/R_P$. Now, recall that $S_{\p}/R_{\p}$ is tamely ramified if and only if the coefficient of $\Ram$ at $Q_i$ equals $e_i-1$ for all $i=1,\ldots,k$; see \cite[IV, Proposition 2.2]{Hartshorne}. In general, $P^{\Tr} = \sum_{i \in I} Q_i$ where $I \subset \{1,\ldots,k\}$ is the subset of indexes where the coefficient of $\Ram$ at $Q_i$ is $e_i-1$, so (a) follows. 
\end{proof}

\begin{corollary} \label{cor.TameRamificationInArithmeticFAMILIES}
    Let $A$ be an integral $\bZ$-algebra of finite type and suppose that $R$ and so $S$ are $A$-algebras of finite type. Let $\bar{\eta} \to \Spec A$ be a geometric generic point, \ie an algebraic closure of $\sK(A)$. Suppose that the geometric generic fibers $R_{\bar{\eta}}$ and $S_{\bar{\eta}}$ are normal integral domains. Then, there is a dense open subset $U \subset \Spec A$ such that the fiber $R_{\mu} \subset S_{\mu}$ is an everywhere tamely ramified finite extension of normal integral domains for all closed points in $\mu \in U$. 
\end{corollary}
\begin{proof}
    First, we find a dense open subset $V \subset \Spec A$ such that $R_{\mu}$ and $S_{\mu}$ are connected and geometrically normal for all $\mu \in V$. For this, use \cite[\href{https://stacks.math.columbia.edu/tag/055G}{Tag 055G}]{stacks-project} and the standard reference \cite[\S12]{EGAIV-3} for the openness of geometrically normal fibers. Next, we use that $\kappa(\mu)$ is a finite field (by \cite[Corollary 5.24]{AtiyahMacdonald}) and so perfect to conclude that $R_{\mu}$ and $S_{\mu}$ are normal. Hence, since on a normal ring connectedness and irreducibility match, $R_{\mu}$ and $S_{\mu}$ are normal integral domains. This part of the argument is essentially the one in \cite[Proposition 2.1]{PatakfalviWaldronSingularitiesGeneralFibersLMMP} (and is ``well-known to experts''). Further, we may shrink $V$ such that $R_{\mu} \to S_{\mu}$ is injective for all $\mu \in V$ (as $R_{\eta} \subset S_{\eta}$). By \autoref{prop.TameRamificationParticularCases}, we may take $U \coloneqq V \cap \Spec A[1/n]$ where $n$ is the product of primes numbers $\leq[L:K]$.
\end{proof}

We can express unramification using transpositions as well:

\begin{proposition}[{\cf \cite[\href{https://stacks.math.columbia.edu/tag/0BTF}{Tag 0BTF}]{stacks-project}}] \label{pro.EtaleCase}
    The following statements are equivalent:
    \begin{enumerate}
        \item  $S/R$ is \'etale over $\p$.
        \item  $S/R$ is unramified over $\p$.
        \item $S_{\p}/R_{\p}$ is quasi-\'etale and flat.
        \item $\p S_{\p} = \p^{\Tr} S_{\p}$.
    \end{enumerate}
\end{proposition}    
    \begin{proof}
        We may assume that $(R,\fram,\kay)$ is a local ring and $\p=\fram$. Recall that $S/R$ being quasi-\'etale means that $\Ram=0$, \ie $\Tr_{S/R}$ is non-singular. In particular, if $S/R$ is unramified then it is quasi-\'etale. We explain next why it is flat (and so why (b) implies (c)). Observe that (c) implies (a) by purity of the branch locus for flat finite covers.
        \begin{claim}
            If $S/R$ is unramified then it is flat, \ie free of rank $[L:K]$.
        \end{claim}
        \begin{proof}[Proof of claim]
            It suffices to show that $[L:K]$ is at least $\mu(\fram)=\rho(\fram)$ (\cite[II, Lemma 8.9]{Hartshorne}). Consider an exact sequence of $R$-modules
            \[
            0 \to R^{\oplus[L:K]} \to S \to T \to 0
            \]
            where $T$ is a torsion finitely generated $R$-module (\ie choose a basis of $L/K$ consisting of elements in $S$). Base change it by the henselization of $R$ to obtain
            \[
            0 \to \big(R^{\mathrm{h}}\big)^{\oplus[L:K]} \to \prod_{\fran \in Y_\m} S_{\fran}^{\mathrm{h}} \to R^{\mathrm{h}} \otimes_R T \to 0 .
            \]
            See \cite[I, \S4, Example 4.10(a), Remark 4.11]{MilneEtaleCohomology} for why $R^{\mathrm{h}}$ and the $S_{\fran}^{\mathrm{h}}$ are (noetherian) normal integral domains. Since $R^{\mathrm{h}} \otimes_R T$ is a torsion $R^{\mathrm{h}}$-module, we conclude that
            \[
            [L:K] \geq \sum_{\fran \in Y_{\fram} } \big[\sK(S_{\fran }^{\mathrm{h}}):\sK(R^{\mathrm{h}})\big].
            \]
            Therefore, it suffices to show that $
            \big[\sK(S_{\fran}^{\mathrm{h}}):\sK(R^{\mathrm{h}})\big] \geq [\kappa(\fran):\kay]$ for all $\fran \in Y_{\fram}$. That is, we have reduced to the case in which $(S,\fran,\el,L)/(R,\fram,\kay,K)$ is a local extension of henselian rings. We show next that $[\el : \kay]$ divide $[L:K]$ in that case.
            
            Since $\el/\kay$ is separable, we may write $\el=\kay(\alpha)=\kay[t]/a(t)$ where $a(t)$ is a monic, irreducible, separable polynomial. Since $R$ is henselian, there is $A(t)\in R[t]$ monic and irreducible whose reduction modulo $\fram$ is $a(t)$. In particular, $R'\coloneqq R[T]/A(t)$ is a finite \'etale $R$-algebra whose special fiber is $\el/\kay$. Likewise, since $S$ henselian, there is $s \in S$ whose residual class is $\alpha$ and $A(s)=0$. In other words, we may realize $R'$ as an $R$-subalgebra of $S$. In particular, $[\el:\kay]$ is the rank of $R'/R$ and so it divides $[L:K]$.
        \end{proof}
        It remains to explain why $\fram S = \fram^{\Tr}$ is equivalent to $S/R$ being quasi-\'etale and free. 
        Consider the following inequalities $\mu(\fram) \geq \rho(\fram) \geq \eta(\fram)$. If $S/R$ is quasi-\'etale then $\eta(\fram)$ is the free-rank of $S/R$. If $S/R$ is further flat, then  $\eta(\fram)=[L:K] = \mu(\fram)$. Conversely, suppose that $\eta(\fram) \geq \mu(\fram)$. Then, the free rank of $S/R$ and so $[L:K]$ is at least $\mu(\fram)$ and so equal to it. Then $S/R$ is free of rank $\eta(\fram)$. This further implies that $\Tr$ is non-singular.
    \end{proof}

We discuss transitivity in towers next.

\begin{proposition}[Transitivity]\label{pro.TransitivityTameRamificationSeparableCase}
    Let $M/L/K$ be a tower of finite separable extension.
    Let $T\coloneqq \bar{R}^M$. Then, $T/R$ is tamely ramified over $\p$ if and only if so is $S/R$ and $T/S$ is tamely ramified over every prime ideal of $S$ lying over $R$. In particular, if $M$ is the Galois closure of $L$, then if $T/R$ is tamely ramified over $\p$ then so is $S/R$.
\end{proposition}
\begin{proof}
    It follows from \autoref{pro.TransitivityTameRamificationGeneralCase} as here transpositions of radical ideals are radical.
\end{proof}

The following lemma, although rather complicated-looking, is the key to compare our trace-theoretic approach with the ones discussed in \cite{KerzSchmidtOnDifferentNotionsOfTameness}.

\begin{lemma}\label{pro.GaloisExtensionsAreUnivTamelyRam}
    Let $R \subset R'$ be a normal integral ring extension of $R$ with field of fractions $K'$ (which contains $K$) and let $\p' \in \Spec R'$ lie over $\p$. Suppose that $L$ and $K'$ are linearly disjoint inside a larger field $E$ (\ie $K=L\cap K'$ inside $E$). Let $L'\coloneqq LK'\subset E$ be the composite field, so that $L'/K'$ is finite separable of degree $[L:K]$. Let $S'$ be the normalization of $R'$ in $L'$. Then, $\p'^{\Tr_{S'/R'}}\cap S \subset \p^{\Tr_{S/R}}$. In particular, for every $\q$ minimal prime of $\p^{\Tr_{S/R}}$ there is a $\q'$ in $\Spec S'$ lying over $\p'$ and over $\q$ such that $\p'^{\Tr_{S'/R'}} \subset \q'$.
\end{lemma}
\begin{proof}
By linear disjointness, the restriction of $\Tr_{L'/K'}$ to $L$ is $\Tr_{L/K}$. In particular, the restriction of $\frT' \coloneqq \Tr_{S'/R'}$ to $S$ is $\frT \coloneqq \Tr_{S/R}$. Noting that
\[
\frT\big(\p'^{\frT'} \cap S\big) = \frT'\big(\p'^{\frT'} \cap S\big)\subset \frT'\big(\p'^{\frT'} \big) \cap R \subset \p' \cap R = \p,
\]
the result follows.
\end{proof}

\begin{definition}[Universal tame ramification] \label{def.UniversalTameRamification}
  We say that $S/R$ is \emph{universally tamely ramified over $\p$} if: for all fields $K'$ contained in a field $E$ that also contains $L$, the extension $S'/R'$ is tamely ramified over every prime lying over $\p$; where $R'=\bar{R}^{K'}$ and $S'=\bar{S}^{L'}$ for $L'\coloneqq LK'\subset E$. 
\end{definition}

\begin{corollary} \label{cor.UnivseralTamenessGaloisCase}
    If $L/K$ is Galois, $S/R$ universally tamely ramified over $\p$ if (and only if) it is tamely ramified over $\p$. 
\end{corollary}
\begin{proof}
    Let us work in the setup of \autoref{def.UniversalTameRamification}. Set $\tilde{K}\coloneqq K'\cap L$, $\tilde{R}\coloneqq \bar{R}^{\tilde{K}}$, and $\tilde{\p}\coloneqq \p'\cap \tilde{R}$. Note that $L/\tilde{K}$ and $L'/K'$ are also Galois (\cite[Ch.14, Proposition 19]{DummitFoote}). By \autoref{pro.TransitivityTameRamificationSeparableCase}, $S/\Tilde{R}$ is tamely ramified over $\tilde{\p}$. Then, we can apply \autoref{pro.GaloisExtensionsAreUnivTamelyRam} yielding the required result by using \autoref{rem.TheGaloisCase}.
\end{proof}

\begin{remark}
    Following \cite{KerzSchmidtOnDifferentNotionsOfTameness}, one may restrict itself in \autoref{def.UniversalTameRamification} to the case in which $R'$ belongs to a certain class of valuation rings of $K$, say: non-archimedean valuation rings, DVRs, divisorial valuation rings. Then, one has notions of \emph{valuation-tameness}, \emph{discrete-valuation-tameness}, and \emph{divisorial-tameness}; respectively. 
\end{remark}

Putting everything together, we obtain the following result comparing universal tame ramification and tame ramification as well as the role played by Galois closures.

\begin{theorem} \label{thm.UNIVERSALvsSimpleTameRamification}
     Let $M$ be the Galois closure of $L$ and $T\coloneqq \bar{R}^M$. Then, $S/R$ is universally tamely ramified over $\p$ if and only if $T/R$ is tamely ramified over $\p$.
\end{theorem}
\begin{proof}
    Suppose that $S/R$ is universally tamely ramified over $\p$.
    \begin{claim} \label{claim.Composites}
        Let $\bar{K}/L'/K$ be a finite separable extension, $S'\coloneqq \bar{R}^{L'}$, and $U\coloneqq \bar{R}^{LL'}$. Then, $U/R$ is tamely ramified over $\p$ if so is $S'/R$.
    \end{claim}
    \begin{proof}[Proof of claim]
        Since $S/R$ is universally tamely ramified over $\p$, it follows that $U/S'$ is tamely ramified over every prime of $S'$ lying over $\p$. Then, the claim follows from \autoref{pro.TransitivityTameRamificationSeparableCase}.
    \end{proof}
    Let us use the primitive element theorem to write $L=K(\alpha)$ and let $\alpha_1,\ldots,\alpha_k \in \bar{K}$ be the conjugates of $\alpha$; say $\alpha_1=\alpha$. Letting $S_i$ denote $\bar{R}^{K(\alpha_i)}$, we have that $S_i/R$ is tamely ramified over $\p$. Then, \autoref{claim.Composites} and the fact that $M=K(\alpha_1,\ldots,\alpha_k)=L_1\cdots L_k$ imply that $T/R$ is tamely ramified over $\p$.

    Conversely, assume that $T/R$ is tamely ramified over $\p \in \Spec R$. To show that $S/R$ is universally tamely ramified over $\p \in \Spec R$, work in the setup of \autoref{def.UniversalTameRamification}. Set $M'\coloneqq MK'$ and let $T'\coloneqq \bar{T}^{M'}$. By \autoref{cor.UnivseralTamenessGaloisCase} and \autoref{pro.TransitivityTameRamificationSeparableCase}, we conclude that $T'/R'$ and so $S'/R'$ is tamely ramified over every prime of $R'$ lying over $\p$.
\end{proof}

 We recall next the concept of \emph{ramification index}.

\begin{definition}[{\cite[\href{https://stacks.math.columbia.edu/tag/0BSD}{Tag 0BSD}]{stacks-project}}] \label{def.RamificationIndex}
      Suppose that $L/K$ is Galois with $G=\Gal(L/K)$. For $\q \in \Spec S$, we let $D_{\q} \coloneqq \{\sigma \in G \mid \sigma(\q) = \q\}$ be the \emph{decomposition group} of $\q$ and 
     \[
     I_{\q} \coloneqq \{\q \in D_{\q} \mid \sigma/\q = \id_{\kappa(\q)}\} = \ker\big(D_{\q}\to \Aut \big(\kappa(\q)/\kappa(\p) \big)\big)
     \] 
     be the \emph{inertia group} of $\q$. The inertia groups $I_{\q}$ are all conjugate and so $|I_{\q}|$ is independent of $\q$. We denote such number as $e_{\p}$ and refer to it as the \emph{ramification index} of $S/R$ over $\p$.
\end{definition}
\begin{remark} \label{rem.InertiaDecompositionAutExactSequence}
  With notation as in \autoref{def.RamificationIndex}, $\kappa(\q)/\kappa(\p)$ is a normal extension as $R=S^G$. See \cite[\href{https://stacks.math.columbia.edu/tag/0BRJ}{Tag 0BRJ}]{stacks-project}. Moreover, there is a short exact sequence
  \[
  1 \to I_{\q} \to D_{\q} \to A_{\q} \coloneqq \Aut \big(\kappa(\q)/\kappa(\p) \big) \to 1.
  \]
  Since $G$ acts transitively on $Y_{\p}$, every $\q \in Y_{\p}$ defines a bijection $G/D_{\q} \to Y_{\p}$ by sending a left coset $\sigma D_{\q}$ to $\sigma(\q)$. Of course, $D_{\q'}$ is the left coset that corresponds to $\q'$. In particular, $|Y_{\p}|=[G:D_{\q}]$ for all $\q \in Y_{\p}$. Hence, we may write $\Tr = \sum_{\q \in Y_{\p}}\sum_{\sigma \in D_{\q}} \sigma$ which let us see that
  \[
 \bar{\Tr}_{\q} = \sum_{\q \in D_{\q}} \sigma = e_{\p} \cdot \sum_{\sigma \in A_{\q}} \sigma
  \]
  and so that $\bar{\Tr}_{\q} = e_{\p} \cdot \Tr_{\kappa(\q)/\kappa(\p)}$ if $\kappa(\q)/\kappa(\p)$ is separable. Indeed, use \autoref{rem.BeingExplictAboutT_q} and note that if $s \in \frab \coloneqq  \bigcap_{\q \neq \q' \in Y_{\p}} \q'$ then $\Tr(s) -\sum_{\sigma \in D_{\q}} \sigma(s)=\sum_{\q \neq \q' \in Y_{\p}}\sum_{\sigma \in D_{\q'}} \sigma(s) \in R \cap \frab =\p$.
\end{remark}

\begin{proposition}[{\cite[\href{https://stacks.math.columbia.edu/tag/0BTF}{Tag 0BTF}]{stacks-project}}]
    In the setup of \autoref{pro.EtaleCase}, suppose that $L/K$ is Galois. Then, we may add $e_{\p}=1$ to the equivalent statements in  \autoref{pro.EtaleCase}.
    \begin{proof}
        Suppose that $\p S_{\p}=\p^{\Tr}S_{\p}$. Then, $[L:K]=|Y_{\q}| [\kappa(\q):\kappa(\p)]$ and so $|D_{\q}| = [\kappa(\q):\kappa(\p)]$ for all $\q$ (as $|Y_{\q}|\in \bN_+$). Since $\kappa(\q)/\kappa(\p)$ is separable, it is Galois and the result follows from \autoref{rem.InertiaDecompositionAutExactSequence}. Conversely, suppose that $e_{\p}=1$. In other words, $D_{\q}$ acts faithfully on $\kappa(\q)$ and so we obtain a tower $\kappa(\q)/\kappa(\q)^{D_{\q}}/\kappa(\p)$ such that $\kappa(\q)/\kappa(\q)^{D_{\q}}$ is Galois with Galois group $D_{\q}$; see \cite[\href{https://stacks.math.columbia.edu/tag/09I3}{Tag 09I3}]{stacks-project}. By \autoref{rem.InertiaDecompositionAutExactSequence}, there is a factorization
        \[
        \Tr_{\kappa(\q)/\kappa(\q)^{D_{\q}}} \: \kappa(\q) \xrightarrow{{\bar{\Tr}_{\q}}} \kappa(\p) \xrightarrow{\subset} \kappa(\q)^{D_{\q}}.
        \]
        which forces the displayed vertical inclusion to be an equality as $\Tr_{\kappa(\q)/\kappa(\q)^{D_{\q}}}$ is surjective. That is, $\kappa(\p) = \kappa(\q)^{D_{\q}}$. This implies that $\p^{\Tr}=\sqrt{\p S}$ and $\rho(\p)=[L:K]$. Therefore, the free rank of $S_{\p}$ as an $R_{\p}$-module is at least and so equal to $[L:K]$. That is, $S_{\p}/R_{\p}$ is flat of rank equal to the one of its special fiber. This shows that $\p^{\Tr} S_{\p}\subset \p S_{\p}$.
    \end{proof}
    
\end{proposition}

\begin{definition}[\cite{KerzSchmidtOnDifferentNotionsOfTameness}]
     With notation as in \autoref{def.RamificationIndex}, $S/R$ is \emph{numerically (resp. cohomologically) tamely ramified} over $\p$ if $\Char \kappa(\p) \nmid e_{\p}$ (resp. $\Tr_{\p}\:S_{\p} \to R_{\p}$ is surjective).
\end{definition}

\begin{proposition} \label{prop.DifferentNOTIONStameRamificationGaloisCase}
     If $L/K$ is Galois the following statements are equivalent:
     \begin{enumerate}
         \item  $S/R$ is tamely ramified over $\p$.
         \item  $S/R$ is cohomologically tamely ramified over $\p$.
         \item $S/R$ is numerically tamely ramified over $\p$.
     \end{enumerate}
\end{proposition}
\begin{proof}
The equivalence between (a) and (b) was discussed in \autoref{rem.TheGaloisCase}. By \autoref{rem.InertiaDecompositionAutExactSequence}, if $\bar{\Tr}_{\q} \neq 0$ then $e_{\p} \neq 0 \in \kappa(\p)$, \ie (a) implies (c). Conversely, suppose that $e_{\p}$ is a unit in $\kappa(\p)$.
According to \autoref{rem.InertiaDecompositionAutExactSequence}, $A_{\q} \coloneqq D_{\q}/I_{\q}$ acts faithfully on $\kappa(\q)$ and there is a factorization
 \[
        \Tr_{\kappa(\q)/\kappa(\q)^{A_{\q}}} \: \kappa(\q) \xrightarrow{ e_{\p}^{-1} \cdot\bar{\Tr}_{\q}} \kappa(\p) \xrightarrow{\subset} \kappa(\q)^{A_{\q}}.
\]
where $\kappa(\q)/\kappa(\q)^{A_{\q}}$ is Galois; see \cite[\href{https://stacks.math.columbia.edu/tag/09I3}{Tag 09I3}]{stacks-project}. In particular, $\bar{\Tr}_{\q} \neq 0$; as required.
\end{proof}

\begin{corollary}
    The locus of universal tame ramification is open and so closed under generization.
\end{corollary}
\begin{proof}
    With notation as in \autoref{thm.UNIVERSALvsSimpleTameRamification} and using \autoref{prop.DifferentNOTIONStameRamificationGaloisCase}, $\Tr_{M/K}(T) \subset R$ cuts out the locus of points $\p \in \Spec R$ over which $S/R$ is not universally tamely ramified.
\end{proof}

  \begin{remark}[Final comparison]
    Following \cite{KerzSchmidtOnDifferentNotionsOfTameness}, we may say that $S/R$ is numerically (resp. cohomologically) tamely ramified over $\p$ if so is the normalization of $S$ in the Galois closure of $L/K$. Kerz--Schmidt showed that these two notions are equivalent and are the strongest among other notions of tameness in the literature. We have shown them to be equivalent to our notion of universal tameness. We believe that our trace-theoretic approach is advantageous in at least two ways. Namely, universal tameness does not require passing to Galois closures to be defined and it makes explicit the extent to which trace tameness is much weaker than cohomological/numerical tameness. For instance, if $S/R$ is a local extension in dimension $2$, tameness means that the trace is surjective whereas universal tameness would ensure tameness over the height-$1$ prime ideals and so everywhere on $\Spec R$ (and far beyond). For example, we may take $S/R$ of degree prime to the characteristic, ensuring the surjectivity of the trace, but it can still ramify wildly over height $1$ primes. For instance, take $R\coloneqq \bF_2[x,y]_{\langle x,y \rangle}$ and $S$ its normalization in $\bF_2(x,y)[t]/\langle t^3+ t +x \rangle$. This is wildly ramified over $\langle x \rangle$.
  \end{remark}

\section{Centers of $F$-Purity and the Cartier Core Map}\label{Section CFPs and Cartier Core map}
We revisit next some of the basics of $F$-singularities using the language developped in \autoref{section tame ramification}. We focus on centers of $F$-purity and the \emph{Cartier core map} as introduced in \cite{BadillaCespedesFInvariantsSRRings,BrosowskyCartierCoreMap}. As stated in the introduction, we also study these as their own spectrum.
Although interesting in its own right, this rephrasing will be useful to understand the behavior of centers of $F$-purity under homomorphisms in \autoref{sec.FiberedHomomorphisms} and \autoref{sec.FiberedTranspositions}.

From now on, \emph{we assume all rings to be $F$-finite $\bF_p$-algebras}. In particular, $F^e \: R \to R$ is a finite cover and choosing $\phi \in \omega_{F^e}$ lands us in the setup of \autoref{def.TranspositionOfIdeals}. Notably, $\Spec F^e = \id_{\Spec R}$. We refer to the data $(R,\phi)$ of a ring $R$ and a map $\phi \in \omega_{F^e}$ as a \emph{Cartier pair}. In this section, we specialize \autoref{section tame ramification} to this specific but prominent case.

Note that $\omega_{F^e}$ has an $R$-bimodule structure such that $\phi \cdot r^q = r \cdot \phi$ for all $r\in R$. Let us set $\omega_{F^0}\coloneqq R$ so that we may generalize this $R$-bimodule structure as follows. For all $e,e'\in \bN$ and $\phi \in \omega_{F^e}$, $\phi'\in \omega_{F^{e'}}$, define $\phi \cdot \phi' \in \omega_{F^{e+e'}}$ as the composition $\phi \circ F_*^e \phi'$. In particular, for all $n\in \bN$, one defines $\phi^n \in \omega_{F^{en}}$ inductively as $\phi^0 \coloneqq 1$ and $\phi^{n+1} \coloneqq \phi \cdot \phi^n = \phi^n \cdot \phi$. If $W\subset R$ be a multiplicative subset, we let $W^{-1}\phi \in \omega_{F^e_{W^{-1}R}}$ be the only map making the following diagram commutative
\[
\xymatrixcolsep{3pc}\xymatrix{
F^e_* R \ar[r]^-{\phi} \ar[d]_-{F^e_*\lambda} & R \ar[d]^-{\lambda \coloneqq \text{localization}} \\
F^e_* W^{-1}R  \ar[r]^-{W^{-1}\phi} & W^{-1}R, 
}
\]
When $W=R\setminus \p$ for some $\p \in \Spec R$, we write $W^{-1}\phi$ as $\phi_{\p} \: F^e_* R_{\p} \to R_{\p}$ instead.

 According to \autoref{pro.SeparableVSpurelyInseparable}, for all radical ideals $\fra \in I(R)$, we have that $\fra^{\phi} \subset \fra$ if and only if $\phi$ is non-degenerate along $\fra$. However, in this specific setup, we also have the following.
 \begin{lemma}\label{lemm.CuteLemma}
     Let $(R,\phi)$ be a Cartier pair. If $\phi$ is non-degenerate along $\fra \in I(R)$ then $\fra^{\phi} \subset \fra$. 
 \end{lemma}
 \begin{proof}
Let $x\in \fra^{\phi}$. Then $x^q \in \fra^{\phi}$ and so $x \in \fra \: \mathbf{1}_{\phi}=\fra$ (see \autoref{pro.BasicPropertiesTranspositions}\autoref{key.Item}).     
 \end{proof}
 
\begin{definition}\label{definition phi ideals}
    Let $(R,\phi)$ be a Cartier pair. We define the \emph{ideals of $(R,\phi)$} as the set
    \[
   I(R,\phi) \coloneqq \{\fra \in I(R) \mid \fra_\phi\subset \fra\} = \{\fra \in I(R) \mid \fra \subset \fra^{\phi}\}
    \]
    and refer to its elements as \emph{$\phi$-ideals}.\footnote{Aka ($\phi$-)compatible ideals in the literature.} That is, a $\phi$-ideal is an ideal $\fra \in I(R)$ for which there is a (necessarily unique) $R/\fra$-linear map $\phi/\fra$ making the following diagram commutative
    \[
\xymatrix{
F^e_* R \ar[r]^-{\phi} \ar[d]_-{F^e_*\pi} & R \ar[d]^-{\pi\coloneqq \text{quotient}} \\
F^e_* R/\fra  \ar[r]^-{\phi/\fra} & R/\fra, 
}
\]
If $\fra_\phi= \fra$, we refer to $\fra$ as a \emph{fixed $\phi$-ideal}. We define the \emph{Cartier spectrum} of $(R,\phi)$ as 
\[
\CSpec(R,\phi) \coloneqq I(R,\phi) \cap \Spec R,\] 
and refer to its elements as \emph{Cartier primes}, which are the prime ideals $\p \in \Spec R$ such that $\phi$ descends to a homomorphism $\bar{\phi}_{\p} \coloneqq \phi_{\p}/\p R_{\p}\:F^e_* \kappa(\p) \to \kappa(\p)$.
If $\bar{\phi}_{\p} \neq 0$ (\ie $F^e$ is tamely $\phi$-ramified over $\p$),  one says that $\p$ is a \emph{center of $F$-purity} (\cite{SchwedeCentersOfFPurity}). We define the \emph{Schwede spectrum} of $(R,\phi)$ as
\[
\Schpec(R,\phi) \coloneqq \{\p \in \CSpec(R,\phi) \mid \bar{\phi}_{\p} \neq 0 \}.
\]
We endow $\Schpec(R,\phi)$ and $\CSpec(R,\phi)$ with the subspace topology from $\Spec R$ and we write $V_{\phi}(\fra) \coloneqq \CSpec(R,\phi) \cap V(\fra)$ and  $V_{\phi}^{\circ}(\fra) \coloneqq V(\fra) \cap \Schpec(R,\phi) = V_{\phi}(\fra) \cap \Schpec(R,\phi)$.
\end{definition}

The following properties are direct from the definitions and well-documented in the literature; see e.g. \cite[Proposition 3.3]{SmithZhang}, \cite[Lemma 2.13]{BlickleSchwedeTuckerFSigPairs1}, \cite{SchwedeTuckerTestIdealSurvey}, \cite[Proposition 6.7]{SchwedeSmithFSingularitiesBook}.

\begin{proposition}\label{first proposition properties of comp ideal}
    Let $(R,\phi)$ be a Cartier pair, $\fra\in I\rphi$, $\frab\in I(R)$, and $W\subset R$ a multiplicative subset. Then,
    \begin{enumerate}
        \item $I(R,\phi) \subset I(R)$ is a sublattice, \ie it is closed under sums and intersections and it contains $\mathbf{0}, \mathbf{1}$.
        \item  $(\fra : \frab) \in I(R,\phi)$ and consequently $\Ass_R R/\fra \subset I(R,\phi)$ and $\sqrt{\fra} \in I(R,\phi)$.
        \item  Extension of ideals along $\pi \: R \to R/\fra$ induces an isomorphism of lattices  
        \[\{\frab \in I(R,\phi) \mid \fra \subset \frab \} \to I(R/\fra, \phi/\fra),\] which restricts to homeomorphisms 
        \[ 
        V_{\phi}(\fra) \to \CSpec(R/\fra,\phi/\fra)  \text{ and } V_{\phi}^{\circ}(\fra) \to \Schpec(R/\fra,\phi/\fra).
        \]
        \item Extension of ideals along $\lambda \: R \to W^{-1}R$ yields a homomorphism of lattices \[I(R,\phi) \to I\big(W^{-1}R, W^{-1}\phi\big),\]
        which restricts to a homeomorphism 
        $$\{\p \in \CSpec(R,\phi) \mid \p \cap W = \emptyset\} \to \CSpec\big(W^{-1}R, W^{-1}\phi\big)$$
        where centers of $F$-purity correspond to centers of $F$-purity.
    \end{enumerate}
\end{proposition}

\begin{definition}[\cf {\cite{BadillaCespedesFInvariantsSRRings,BrosowskyCartierCoreMap,AberbachEnescuStructureOfFPure}}] 
Let $\rphi$ be a Cartier pair and $\fra \in I(R)$. The \emph{Cartier core} of $\fra$ is the $\phi$-ideal $\upkappa_{\fra}(R,\phi) \coloneqq  \bigcap_{n \in \bN_{+}} \fra^{\phi^n}$. Dually, we define the \emph{image ideal of $(R,\phi)$ along $\fra$} as the $\phi$-ideal $\uplambda_{\fra}(R,\phi) \coloneqq  \sum_{n \in \bN_{+}} \fra_\phin$.
\end{definition}

\begin{proposition}\label{propdef cartier core} Let $(R,\phi)$ be a Cartier pair, $X\coloneqq \Spec R$, $A \subset I(R)$, $\fra,\frab\in I(R)$, and $W\subset R$ be a multiplicative subset. Then:
\begin{enumerate}
    \item $W^{-1}\upkappa_{\fra}(R,\phi) = \upkappa_{W^{-1}\fra}(W^{-1}R,W^{-1}\phi)$.
    \item $\upkappa_{\bigcap_{\fra \in A} \fra}\rphi=\bigcap_{\fra \in A} \upkappa_{\fra}\rphi$ and $\uplambda_{\sum_{\fra \in A} \fra}\rphi=\sum_{\fra \in A} \uplambda_{\fra}\rphi$.
    \item $\fra \subset \upkappa_{\frab}(R,\phi)$ if and only if $\uplambda_{\fra}(R,\phi) \subset \frab$.
    \item $ V\big(\uplambda_{\fra}(R,\phi)\big) = \{\p \in X \mid \fra \subset \upkappa_{\p}(R,\phi)\}$. 
    \item  $\upkappa_{\fra}(R,\phi) \neq \mathbf{1} $ 
    if and only if $\mathbf{1}_\phi \not\subset \fra$.
    \item If $\mathfrak{c} \in I(R,\phi)$ is contained in $\fra$ then $\mathfrak{c} \subset \upkappa_{\fra}(R,\phi)$.
     \item If $\phi$ is non-degenerate along $\fra$ then $\upkappa_{\fra}(R,\phi) \subset \fra $ and so $\upkappa_{\fra}(R,\phi) $ is the largest $\phi$-ideal contained in $\fra$. In particular, if $\fra\in I\rphi$ then $\fra=\upkappa_{\fra}\rphi$.
 \end{enumerate}
\end{proposition}
\begin{proof}
Parts (a), (b), (c), and (d) follow from \autoref{pro.BasicPropertiesTranspositions}. See also \cite[Notation 3.2, Proposition 3.6]{BrosowskyCartierCoreMap}. For (e), $\upkappa_\fra\rphi\neq \mathbf{1}$ if and only if there is an $n\in \N_+$ such that $\fra^{\phi^n}=(\fra^\phi)^{\phi^{n-1}}\neq \mathbf{1}$. For (f), $\mathfrak{c}_\phin \subset \mathfrak{c} \subset \fra$ so $\mathfrak{c} \subset \fra^{\phi^n}$ for all $n\in \N_+$. \autoref{lemm.CuteLemma} yields (g).
\end{proof}

\begin{defprop}\label{defprop sigma} Let $\rphi$ be a Cartier pair and $\fra \in I(R,\phi)$. We let $\upsigma(\fra,\phi)\in I(R,\phi)$ denote the stable element of the descending chain of $\phi$-ideals $
    \fra \supset \fra_{\phi^1} \supset \fra_{\phi^2} \supset \cdots $. It is the largest fixed $\phi$-ideal contained in $\fra$. Moreover, 
    \[
    V(\upsigma(\fra,\phi))=V(\fra_\phi)=V(\uplambda_\fra\rphi)=\{\p \in \Spec R \mid \fra \subset \upkappa_{\p}(R,\phi)  \}.
    \]
\end{defprop}

\begin{proof} The stabilization of the chain is well-known \cite{Gabber.tStruc,BlickleSchwedeSurveyPMinusE,BlickleTestIdealsViaAlgebras}. We explain why for all $n \in \bN_{+}$ and all $\p$, if $\fra_{\phi^{n+1}} \subset \p$ then $\fra_\phin \subset \p$. It suffices to do so after localizing at $\p$ whence we may take $R$ local with maximal ideal $\p$. Then, we must show that: if $\phi^n \cdot a$ is surjective for some $a \in \fra$ then so is $\phi^{n+1}\cdot b$ for some $b\in \fra$. Let $a \in \fra$ be such that $\phi^n \cdot a \: F^{en}_*1 \mapsto 1$. Then, $\phi^n(F^{en}_* ar)= a$ for some $r \in R$. Let $b\coloneqq\phi^{n-1}(F^{en-e}_* ar)\in \fra$. Then, $\phi(F^e_* b) = a$ and $\phi^{n+1} \cdot b$ is surjective. Since $\fra_{\phi}=\uplambda_{\fra}(R,\phi)$, the rest follows from \autoref{propdef cartier core}.
\end{proof}

\begin{corollary}\label{Pro.NonDegeneracyPhiIdeals}
Let $\rphi$ be a Cartier pair and $\fra\in I\rphi$. Then, $\phi$ is non-degenerate along $\fra$ (\ie $\Ass_R R/\fra \subset \Schpec(R,\phi)$) if and only if $\fra$ is radical and $\fra=\fra^\phi$. In particular, $\phi$ is non-degenerate if and only if $R$ is reduced and $\mathbf{0}^{\phi}=\mathbf{0}$.
\end{corollary}
\begin{proof}
Suppose that $\phi$ is non-degenerate along $\fra$. By \autoref{defprop sigma}, this means that we may take $r\in W_{\fra}\cap \upsigma(R,\phi)$. In particular, for every $n \in \N_{+}$, there is $r_n\in R$ such that $r=\phi^n(\fstar^{en} r_n)$. Let $x\in \sqrt{\fra}$. Then, there is an $n\gg 0 $ such that $x^{q^{n}}\in \fra$ and so $xr=\phi^n(\fstar^{en} x^{q^{n}} r_n)\in \fra$, so $x\in \fra$. The rest follows from \autoref{pro.SeparableVSpurelyInseparable}. 
\end{proof}

\begin{definition}\label{propdef f pure locus}
Let $\rphi$ be a Cartier pair. The \emph{non-$F$-pure locus} of $(R,\phi)$ is the closed subset $V(\upsigma(R,\phi))=\{\p \in \Spec R \mid \upkappa_{\p}(R,\phi)=\mathbf{1}\}$. Its complement, $\sU(R,\phi)$, is the \emph{$F$-pure locus} of $(R,\phi)$. The pair $(R,\phi)$ is \emph{$F$-pure} if $\sU(R,\phi)=\Spec R$.  
\end{definition}

\begin{remark}
     Inside $\Spec R$, we have $\CSpec(R,\phi) = \Schpec(R,\phi) \sqcup V(\upsigma(R,\phi))$ and further $\Schpec(R,\phi)=\CSpec(R,\phi)\cap \sU(R,\phi)$. If $(R,\phi)$ is $F$-pure then $I(R,\phi)$ is a finite set of radical ideals (\cite[Corollary 5.10]{SchwedeFAdjunction}, \cite[Theorem 1.1]{KumarMehtaFiniteness}).
\end{remark}

\begin{example} \label{rem.CFPsOfHeight1}
    Let $(R,\phi)$ be a Cartier pair where $R$ is a normal integral domain and $\phi \neq 0$. Let $\Delta_{\phi}$ be the tame boundary attached to $\phi$; see \cite{SchwedeCentersOfFPurity}. Then, $\Delta_{\phi} \geq P$ if and only if $R(-P) \in \CSpec(R,\phi)$; see \eg \cite[Proposition 2.1]{CarvajalRojasStaeblerTameFundamentalGroupsPurePairs}. Then, the elements of $\CSpec(R,\phi)$ (resp. $\Schpec(R,\phi)$) of height $1$ are those corresponding to prime divisors on $\Spec R$ supporting $\Delta_{\phi}$ with coefficient $\geq 1$ (resp. $=1$). In particular, if $(R,\phi)$ is $F$-pure then the coefficients of $\Delta_{\phi}$ belong to $[0,1]\cap \bZ_{(p)}$.
\end{example}

\begin{proposition} \label{pro.CartierCoreIsPrime}
  Let $(R,\phi)$ be a Cartier pair. For $\p \in \sU(R,\phi)$, $\upkappa_{\p}(R,\phi) \in \Schpec(R,\phi)$. In particular, $V\big(\uplambda_{\fra}(R,\phi)\big) = \{\p \in \sU(R,\phi) \mid  \upkappa_{\p}(R,\phi) \in V(\fra)  \}  \sqcup V(\upsigma(R,\phi))$ for all $\fra \in I(R)$.
\end{proposition}
\begin{proof}
See \eg \cite[Theorem 3.9]{BrosowskyCartierCoreMap}.
\end{proof}

\begin{definition}
Let $\rphi$ be a Cartier pair and $\fra \in I(R)$. We define the \emph{Cartier radical} of $\fra$ as $\uprho_{\fra}(R,\phi) \coloneqq  \fra + \uplambda_{\fra}(R,\phi) \in I(R,\phi)$---the smallest $\phi$-ideal that contains $\fra$. Dually, we define the $\phi$-ideal $\upbeta_{\fra}(R,\phi)\coloneqq\fra \cap \upkappa_{\fra}(R,\phi)$---the largest $\phi$-ideal contained in $\fra$. 
\end{definition}
\begin{proposition}\label{proposition beta as an object} Let $\rphi$ be a Cartier pair and $\fra \in I(R)$. Then, $\fra \in I(R,\phi)$ if and only if $\upbeta_{\fra}(R,\phi) = \fra$ and if and only if $\uprho_\fra\rphi=\fra$. If $\phi$ is non-degenerate along $\fra$ then $\upbeta_{\fra}(R,\phi) = \upkappa_{\fra}(R,\phi)$ and $\uprho_{\fra}(R,\phi) = \uplambda_{\fra}(R,\phi)$ and these are radical ideals. Moreover, $\upbeta$ commutes with intersections whereas $\uprho$ commutes with sums. 
\end{proposition}
\begin{proof}
This is direct from the definitions and \autoref{propdef cartier core}.
\end{proof}

We readily obtain the following explaining the topological behavior of Cartier cores:

\begin{proposition}\label{proposition beta as a map}
 Let $\rphi$ be a Cartier pair and $\fra,\frab \in I(R)$. If $\p \in \Spec R$ then 
 \[
    \upbeta_{\p}(R,\phi) =  \begin{cases}
\upkappa_{\p}(R,\phi) & \text{ if } \p \in \sU(R,\phi)\\
 \p  &\text{ otherwise. }
\end{cases}
\in \CSpec(R,\phi).
    \]
    Thus, $\p \mapsto  \upbeta_{\p}(R,\phi)$ defines a map $\beta \: \Spec R \to \CSpec(R,\phi)$ that retracts the inclusion $\CSpec(R,\phi) \subset \Spec R$. Moreover, $ \fra \subset \upbeta_{\frab}(R,\phi)$ if and only if $\uprho_{\fra}(R,\phi) \subset \frab.$ 
    In particular:
    \begin{enumerate}
\item  $V\big(\uprho_{\fra}(R,\phi)\big) = \{\p \in \Spec R \mid  \upbeta_{\p}(R,\phi) \in V(\fra)  \}  = \beta^{-1} V_{\phi}(\fra)$ and so $\beta$ is continuous.
    \item $\sqrt{\uprho_{\fra}(R,\phi)} = \bigcap_{\p \in V_{\phi}(\fra)}\p$. 
    \item $V_{\phi}(\fra) = V_{\phi}(\frab)$ if and only if $\sqrt{\uprho_{\fra}(R,\phi)} =\sqrt{\uprho_{\fra}(R,\phi)}$.
        \item The mapping $\fra \mapsto V_{\phi}(\fra)$ is a bijection between radical ideals in $I(R,\phi)$ and closed subsets of $\CSpec(R,\phi)$.
        \item  The topology on $\Spec R$ whose closed subsets are $\{V(\fra) \subset \Spec R \mid \fra \in I(R,\phi)\}$ is the coarsest topology on $\Spec R$ such that $\beta$ is a continuous map.
\end{enumerate}
\end{proposition}

We continue to study $\upbeta$ topologically. This will allow us to express what we think of as good behavior of centers of $F$-purity in a topological way in \autoref{defprop.QuasiFibered}. 

\begin{definition}[Cartier--Zariski topolgy]\label{definition CZpec}
Let $\rphi$ be a Cartier pair, we refer to the topology in \autoref{proposition beta as a map} (e) as the \emph{Cartier--Zariski topology} of $\Spec R$ and let $\CZpec(R,\phi)$ be the corresponding topological space. We define $\upbeta =\upbeta(R,\phi)$ as the continuous map $\CZpec(R,\phi) \to \CSpec(R,\phi)$. We will denote the \emph{Cartier--Zariski (resp. Zariski) closure} of $S \subset \Spec R$ by $\overline{-}^{\mathrm{CZ}}$ (resp. $\overline{-}^{\mathrm{Z}}$).\footnote{Recall that  $\overline{S}^{\mathrm{Z}} = V\big(\bigcap_{\p \in S} \p\big)$ for all $S \subset \Spec R$.}
\end{definition}

One can readily compare the Cartier--Zariski and Zariski topologies via their closure operations: 
\begin{proposition}\label{pro.propertiesCZpec}Let $\rphi$ be a Cartier pair. Then, $
\overline{S}^{\mathrm{CZ}} = \overline{\upbeta(S)}^{\mathrm{Z}} = V\big(\bigcap_{\p \in S} \upbeta_{\p}\big)$ for all subsets $S\subset \Spec R$. In particular, for $\fra \in I(R)$, the following equivalences hold:
\[ 
\overline{V(\fra)}^{\mathrm{CZ}} = V(\fra) \Longleftrightarrow \upbeta(V(\fra)) \subset V(\fra) \Longleftrightarrow \sqrt{\fra}\in I(R,\phi)
\]
\end{proposition}

\begin{proposition}
    Let $(R,\phi)$ be a Cartier pair.Then, $
     V_{\phi}^{\circ}(\fra) = V_{\phi}^{\circ}\big(\upsigma(\fra,\phi)\big)$ for all $\fra \in I(R,\phi)$. Moreover, a pair of fixed $\phi$-ideals $\fra, \frab \in I(R)$ define the same closed subset on $\Schpec(R,\phi)$ if and only if they have the same radical. In particular, $\sU(R,\phi)$ is an open subset of $\CZpec(R,\phi)$ and its closed subsets are of the form $V(\frab) \cap \sU(R,\phi)$ where $\frab$ is a fixed $\phi$-ideal.
\end{proposition}
\begin{proof}
The first statement follows at once from \autoref{defprop sigma}. The rest follows from all fixed $\phi$-ideals being contained in $\upsigma(R,\phi)$.
\end{proof}

Test ideals are another central object in the theory of $F$-singularities which remarkably show up in the fibers of $\upbeta$ as we will now see.

\begin{defprop} [{\cite{SmolkinSubadditivity,SmolkingTauAlongSubscheme,TakagiHigherDimensionalAdjoint}}] Let $(R,\phi)$ be a Cartier pair such that $\phi$ is non-degenerate along $\fra \in I(R,\phi)$. Then, there exists a unique smallest $\phi$-ideal not contained in $Z_\fra$. This ideal is denoted by $\uptau_{\fra}(R,\phi)$ and is referred to as the \emph{test ideal of $(R,\phi)$ along $\fra$}. If $\uptau_{\fra}(R,\phi)=\mathbf{1}$ one says that $(R,\phi)$ is \emph{purely $F$-regular along $\fra$}. In that case, $(R,\phi)$ is $F$-pure and the minimal primes of $\fra$ are pairwise coprime and are referred to as the \emph{maximal centers of $F$-purity} of $(R,\phi)$. One says that $(R,\phi)$ is \emph{$F$-regular} if it is purely $F$-regular along $\mathbf{0}$.
\end{defprop}

\begin{remark} \label{rem.FPurityAsPureFRgularity}
    One readily sees that $\uptau_{\fra}(R,\phi) = \sum_{\p \in \Ass_R R/\fra} \uptau_{\p}(R,\phi)$. Moreover, $(R,\phi)$ is $F$-regular if and only if $\dim \CSpec(R,\phi)=0$ and if and only if it is $F$-pure (so $R$ is reduced) and $\Schpec(R,\phi)$ is the set of minimal primes of $R$. Also, if $(R,\phi)$ is an $F$-pure Cartier pair then it is purely $F$-regular along some $\fra \in I(R,\phi)$. Indeed, since $\Schpec(R,\phi)=\CSpec(R,\phi)$ is a finite set we may take $\fra$ equal the intersection of the maximal elements in there.
\end{remark}

\begin{proposition}[The fibers of $\upbeta$.] Let $\rphi$ be a Cartier pair. Then, for all $\p \in \CSpec(R,\phi)$, we have
    \[
\upbeta^{-1}(\p) = \begin{cases}
V(\p) \cap V\big(\uptau_{\p}(R,\phi)\big)^{\mathsf{c}} & \text{ if } \p \in \sU(R,\phi),\\
 \{\p\}  &\text{ otherwise.}
\end{cases}
\]
In particular, the fibers of  $\upbeta \: \CZpec(R,\phi) \to \Schpec(R,\phi)$ over $\sU(R,\phi)$ are locally closed.
\end{proposition}
\begin{proof}
If $\p\in \sU\rphi$, this follows by going modulo $\p$ and using \cite[Proposition 4.10]{SmolkinSubadditivity} and \autoref{first proposition properties of comp ideal}. If $\upsigma(R,\phi)\subset \p$ then $\upsigma(R,\phi)\subset \q$ for all $\q\in V(\p)$ and so $\upbeta_\q\rphi = \q\neq \p$. 
\end{proof}

So far we used the Cartier core to define the Cartier--Zariski topology. To conclude this section, we would like to do the converse. That is, we explain how to define the Cartier core map from the Cartier--Zariski topology.
For this, we need to review some basic point-set topology.

\begin{defprop}\label{defprop quotient continuous map}
    On a topological space $X$, we define the equivalence relation $x \sim y \Longleftrightarrow \overline{\{x\}} = \overline{\{y\}}$ and let $\upalpha_X \:X \to X/{\sim}$ be the corresponding quotient in $\mathsf{Top}$---the category of topological spaces. We may write $x \sim_X y$ if more than one topological space is involved. Given a continuous map $f\: Y \to X$, there is a commutative diagram in $\mathsf{Top}$:
     \[
\xymatrix{
X\ar[d]_-{\upalpha_X} & Y \ar[l]_-{f} \ar[d]^-{\upalpha_Y} \\
X/{\sim} & Y/{\sim} \ar[l]_-{f/{\sim}}
}
\]
In other words, we have a natural transformation $\upalpha \: \id_{\mathsf{Top}} \to (-)/{\sim}$ between the functors $\id_{\mathsf{Top}}, (-)/{\sim} \: \mathsf{Top}\to\mathsf{Top}$.
\end{defprop}
\begin{proof}
It suffices to show that if $y_1 \sim_Y y_2$ then $f(y_1) \sim_X f(y_2)$. Recall that $f$ being continuous means that $f(\overline{B}) \subset \overline{f(B)}$ for all $B\subset Y$.
Taking $B=\{y_2\}$ yields $f(\overline{\{y_2\}}) \subset \overline{\{f(y_2)\}}$. Hence, if $y_1 \sim y_2$ then $y_1 \in \overline{\{y_2\}}$ and so $f(y_1) \in \overline{\{f(y_2)\}}$. Switching $y_1$ and $y_2$ yields $f(y_1) \sim f(y_2)$.
\end{proof}

\begin{remark}
    If $Y$ is an irreducible subset of a topological space $X$ then so is $\overline{Y}$ (\cite[I, Example 1.1.4]{Hartshorne}). In particular, the closed subsets $\overline{\{x\}}$ are irreducible. One refers to $y\in \overline{\{x\}}$ as \emph{specializations} of $x$ and write $x \rightsquigarrow y$ for all $y \in \overline{\{x\}}$. One also says that $x$ is a \emph{generization} of $y$. This establishes a partial order in $X$ if $X$ is a $\text{T}_0$ space; see \cite[II, Exercise 3.17]{Hartshorne}.
\end{remark}

\begin{definition}
    A topological space $X$ is \emph{quasi-sober} (resp. \emph{sober}) if the mapping $x \mapsto \overline{\{x\}}$ is a surjection (resp. bijection) from $X$ onto the irreducible closed subsets of $X$. A topological space is \emph{quasi-Zariski} (resp. \emph{Zariski}) if it is noetherian and quasi-sober (resp. sober). See \cite[p.5]{Hartshorne} for the definition of noetherianity.
\end{definition}

\begin{example}
    If $(R,\phi)$ is a Cartier pair then $\CZpec(R,\phi)$ is a quasi-Zariski space whereas $\CSpec(R,\phi)$ and so $\Schpec(R,\phi)$ are Zariski spaces. 
\end{example}

As promised, the Cartier core can be understood as a purely topological phenomenon:

\begin{theorem} \label{pro.TopologicalDescriptionOFbETA}
    Let $(R,\phi)$ be a Cartier pair. Then, $\upbeta(R,\phi)= \upalpha_{\CZpec(R,\phi)}$.
   \end{theorem}
\begin{proof}
 We must show that: if $f\: \CZpec(R,\phi) \to X$ in $\mathsf{Top}$ is such that $f(\p) = f(\q)$ whenever $\p \sim \q$, then there is a unique continuous function $g\:\CSpec(R,\phi) \to X$ factoring $f$ through $\upbeta\rphi\:\CZpec(R,\phi) \to \CSpec\rphi$. By \autoref{pro.propertiesCZpec},
    \[
    \overline{\{\p\}}^{\mathrm{CZ}}=\overline{\{\q\}}^{\mathrm{CZ}} \Longleftrightarrow \upbeta_{\p}(R,\phi) =  \upbeta_{\q}(R,\phi)
    \]
    and so the required factorization holds in the category of sets. To do this in $\mathsf{Top}$, we must show that $Z \subset \CSpec(R,\phi)$ is closed if so is $\upbeta^{-1}Z \subset \CZpec(R,\phi)$. Let $Z \subset \CSpec(R,\phi)$ be such that $\upbeta\inv Z=V(\fra)$ for some $\fra \in I\rphi$. Note that $\upbeta_{\p} \in Z$ if and only if $\upbeta_\p \in V(\fra)$. Since $\upbeta$ is a surjection, we conclude that $Z=V_{\phi}(\fra)$ and so it is closed.
\end{proof}

We need one last lemma on topological spaces which will be useful to us when studying the behavior of $\upbeta$ under homomorphisms \autoref{defprop.QuasiFibered}.

 \begin{lemma} \label{lem.NaturalityImpliesContinuity}
Let $X$ be a quasi-Zariski space and $f \: Y \to X$ be a function such that $f(y_1) \sim f(y_2)$ if $y_1 \sim y_2$ in $Y$, \ie there is a commutative diagram of functions
\[
\xymatrixcolsep{3pc}\xymatrix{
X\ar[d]_-{\upalpha_X} & Y \ar[l]_-{f} \ar[d]^-{\upalpha_Y} \\
X/{\sim} & Y/{\sim} \ar[l]_-{g \coloneqq f/{\sim}}
}
\]
If $g$ is continuous then so is $f$. 
    \end{lemma}
    \begin{proof}
        Since $X$ is quasi-Zariski, it suffices to show that $f^{-1}\big(\overline{\{x\}}\big) \subset Y$ is closed for all $x \in X$. Note that $Z \subset \upalpha_X\inv (\upalpha_X(Z))$ is an equality if $Z\subset X$ is closed. Hence, $\upalpha_X$ is closed and: 
        \[
        f^{-1}\big(\overline{\{x\}}\big) = (\upalpha_X \circ f)^{-1} \Big(\upalpha_X\big(\overline{\{x\}}\big)\Big) =
        (\upalpha_X \circ f)^{-1} \big(\overline{\{\upalpha_X(x)\}}\big)= (g\circ \upalpha_Y)^{-1} \big(\overline{\{\upalpha_X(x)\}}\big),
        \]
        which is closed as $g$ (and $\upalpha_Y$) are continuous.
    \end{proof}

\section{Fibered Homomorphisms} \label{sec.FiberedHomomorphisms}

 A ring homomorphism $\theta \: R \to S$ defines a continuous map $f \coloneqq \Spec \theta \: \Spec S \to \Spec R$,  $\q \mapsto \q \cap R$. Moreover, this defines a functor $\Spec \: \mathsf{Rings} \to \mathsf{Top}$ from rings to topological spaces. This relies on two things: the contraction of a prime $S$-ideal is a prime $R$-ideal (\ie the map $f$ is well-defined) and $f$ is continuous as $f^{-1}\big(V(\fra)\big)=V\big(\sqrt{\fra S}\big)$ for all $R$-ideals $\fra$. We determine next to what extent the above extends works for spectra of centers of $F$-purity.

\begin{definition}
    We say that a ring homomorphism is \emph{Cohen--Seidenberg} if it satisfies the incomparability, going-up (so lying-over), and going-down properties. 
\end{definition}

\begin{definition} A \emph{homomorphism of Cartier pairs} $\theta \: (R,\phi) \to (S,\psi)$ is a commutative diagram
\[
\xymatrix{
F^e_* S \ar[r]^-{\psi} & S \\
F^e_* R \ar[u]^-{F^e_* \theta} \ar[r]^-{\phi} & R \ar[u]_-{\theta}
}
\]
where $\theta \: R \to S$ is a homomorphism of rings, $\phi \in \omega_{F^e_R}$, and $\psi \in \omega_{F^e_S}$. 
\end{definition}
\begin{defprop}\label{defprop compatible ideals contract in homomorphisms}
    Let $\theta \: (R,\phi) \to (S,\psi)$ be a homomorphism of pairs. If $\frab \in I(S,\psi)$ then $\frab \cap R  \in I(R,\phi)$. In particular, $f \coloneqq \Spec \theta \: \Spec S \to \Spec R$ restricts to a continuous map $g \coloneqq \CSpec \theta \:  \CSpec(S,\psi) \to \CSpec(R,\phi)$.
    Moreover, every $\q\in \CSpec(S,\psi)$ induces a local homomorphism $\theta_{\q}\:(R_{f(\q)},\phi_{f(\q)}) \to (S_{\q},\psi_{\q})$ and further a residual one $\bar{\theta}_{\q}\:(\kappa({f(\q)}),\bar{\phi}_{f(\q)}) \to (\kappa(\q),\bar{\psi}_{\q})$. In particular, $ g^{-1} \Schpec(R,\phi) \subset \Schpec(S,\psi)$.
\end{defprop}

\begin{example} \label{ex.NonContinuity}
It is well-known (see \eg \cite{SpeyerFrobeniusSplit,CarvajalStablerFsignaturefinitemorphisms}) that $\fra \in I(R,\phi)$ does not imply $\sqrt{\fra S} \in I(S,\psi)$. Indeed, let $R=\bF_2[x]$ and $\phi \: F_*R = R\cdot F_*1 \oplus R \cdot F_*x \to R$ be such that $F_*1 \mapsto 1$ and $F_*x \mapsto x$. Then, $\fra = \langle x \rangle \in I(R,\phi)$. Letting $K=\bF_2(x)$, consider the Artin--Schreier extension over $K$:
    \[
    L \coloneqq \bF_2(y) = K[t]/\langle t^2+xt+1 \rangle = K[t]/\langle (t/x)^2+t/x+1/x^2 \rangle.
    \]   
    Let $S\coloneqq \bar{R}^L =\bF_2[y,y^{-1}]$. Then, $F_* S = S\cdot F_*1 \oplus F_* y$ (\eg $F_*y^{-1} = y^{-1}F_*y$) and the map $\psi\:F_*S \to S$ such that $F_*1 \mapsto 1$ and $F_* y \mapsto 1+y$ is an extension of $\phi$. However, $\sqrt{\fra S} = \langle y+1 \rangle \notin I(S,\psi)$. Noteworthily, $S/R$ is wildly ramified over $\langle x \rangle$. See \cite[Example 4.1]{SchwedeTuckerExplicitlyExtending} for another example. Therefore, although $\q \mapsto\q \cap R$ defines a function $h \coloneqq \CZpec \theta  \:\CZpec(S,\psi) \to \CZpec(R,\phi)$, it is not a continuous one.
\end{example}

\begin{defprop} \label{defprop.QuasiFibered}
Let $\theta \: (R,\phi) \to (S,\psi)$ be a homomorphism of pairs. We say that it is \emph{quasi-fibered} if any of the the following equivalent statements holds:
\begin{enumerate}
    \item The map $h \coloneqq \CZpec \theta  \:\CZpec(S,\psi) \to \CZpec(R,\phi)$ is continuous.
    \item The following diagram is commutative 
    \[
\xymatrixcolsep{5pc}\xymatrix{
\CZpec(R,\phi)\ar[d]_-{\upbeta(R,\phi)} & \CZpec(S,\psi) \ar[l]_-{h=\CZpec \theta} \ar[d]^-{\upbeta(S,\psi)} \\
\CSpec(R,\phi) & \CSpec(S,\psi) \ar[l]_-{g\coloneqq \CSpec \theta}.
}
    \]
    \item 
    The inclusion $ \upbeta_{\q}(S,\psi) \cap R \subset \upbeta_{\q \cap R}(R,\phi)$ is an equality for all $\q \in \Spec S$.
    \item $\sqrt{\upbeta_{\fra}(R,\phi) S} \subset \upbeta_{\sqrt{\fra S}}(S,\psi)$ for all $\fra \in I(R)$.
    \item $\sqrt{\upbeta_{\p}(R,\phi) S} \subset \upbeta_{\sqrt{\p S}}(S,\psi)$ for all $\p \in \Spec R$.
\end{enumerate}
Furthermore, if $\theta \: R \to S$ is Cohen--Seidenberg, then the inclusion in (d) (and so in (e)) becomes an equality whenever $\fra$ is radical (\ie $\upbeta$ commutes with radical extensions). Moreover, in that case, (e) becomes equivalent to each of the following two statements:
\begin{enumerate}
    \item[(f)] $f^{-1}\{\upbeta_{\p}\rphi\} = \upbeta(S,\psi) \big(f^{-1}\{\p\}\big)$ for all $\p \in \Spec R$. 
    \item[(h)] The inclusion $\CSpec(S,\psi) \subset f^{-1}  \CSpec(R,\phi) $ is an equality.
\end{enumerate}
\end{defprop}
\begin{proof} For notational ease, when the ambient ring is clear, we may write $\upbeta$ instead of $\upbeta\rphi$ or $\upbeta\spsi$. The equivalence between (a) and (b) is a direct application of \autoref{pro.TopologicalDescriptionOFbETA} and \autoref{lem.NaturalityImpliesContinuity} after observing that $g=h/{\sim}$. Indeed, note that $h/{\sim}$ is explicitly given by the mapping $\q \mapsto \upbeta_{\q \cap R}$ for all $\q \in \CSpec(S,\psi)$. Since $\q \cap R  \in \CSpec(R,\phi)$, then $h/{\sim}=g$. 

The equivalence of (b) with (c) is straightforward. To see why (c) implies (d), notice that $\upbeta_{\sqrt{\fra S}}=\bigcap_{\fra \subset \q \cap R} \upbeta_{\q}$ whereas $\sqrt{\upbeta_{\fra} S}=\bigcap_{\upbeta_{\fra} \subset \mathfrak{q} \cap R} \mathfrak{q}$. Hence, (d) means that $\fra \subset \q \cap R \Longrightarrow \upbeta_{\fra} \subset \upbeta_{\q} \cap R$ for all $\q \in \Spec S$. This is implied by (c) as $\fra \subset \q \cap R$ implies that $\upbeta_{\fra} \subset \upbeta_{\q \cap R} \subset \upbeta_{\q} \cap R$.

Next, we explain why (e) implies (a).  Observe that (e) implies $\sqrt{\upbeta_{\p} S} \subset \upbeta_{\sqrt{\p S}} \subset \sqrt{\p S}$. In particular, if $\p \in \CSpec(R,\phi)$ then $\sqrt{\pS} = \upbeta_{\sqrt{\p S}}$ and so $\sqrt{\p S} \in I(S,\psi)$. In other words, $h^{-1}\big(V(\p)\big)=V(\sqrt{\p S})$ is Cartier-Zariski closed and so $h$ is continuous.

The above proves the equivalence between the first five statements. We explain now why (d) and (e) are equalities if the Cohen--Seidenberg properties hold and $\fra$ is radical. It suffices to show the equality in (e). Fix $\p\in \Spec R$, then the equality follows if we are able to show that for $\q \in \Spec S$, if $\upbeta_{\p} \subset \q \cap R$ then $\upbeta_{\sqrt{\p S}} \subset \q$. Let $\q \in \Spec S$ be such that $\upbeta_{\p} \subset \q \cap R$.
By the going-down property, there is $\q'\subset \q$ lying over $\upbeta_{\p}$. By the going-up property, there is $\q''\supset \q'$ lying over $\p$ and so that $\sqrt{\p S} \subset \q''$. By the commutativity of the square in (b) (\ie by having inclusions $\upbeta_{\mathfrak{r} \cap R} \subset \upbeta_{\mathfrak{r}} \cap R$ for all $\mathfrak{r}\in \Spec S$), we conclude that $\upbeta_{\p}$ is a generization of both $\upbeta_{\q''}\cap R$ and $\upbeta_{\q'} \cap R$. It then follows that $\upbeta_{\q'}$ lies over $\upbeta_{\p}$. In particular, $\upbeta_{\q'}\cap R = \upbeta_{\p} \subset \upbeta_{\q''} \cap R$, which is further an equality as $\upbeta_{\q''} \supset \upbeta_{\q'}$. By the incomparability property, it follows that $\upbeta_{\q''}=\upbeta_{\q'} \subset \q'$. Using now that $\sqrt{\p S} \subset \q''$, we get that $\upbeta_{\sqrt{\p S}} \subset \q' \subset \q$; as desired.

Finally, (f) is saying that all $\q\in \Spec S$ lying over $\upbeta_\p$ have to be compatible, which follows from the equality of (e). Similarly, (h) is saying that all $\q\in \Spec S$ lying over an element of $\CSpec\rphi$ have to be compatible, so the equivalence follows.
\end{proof}

\begin{defprop} \label{defprop.Fibered}
    Let $\theta \: (R,\phi) \to (S,\psi)$ be a quasi-fibered homomorphism of pairs. We say that it is \emph{fibered} if any of the following equivalent statements holds:
    \begin{enumerate}

        \item The inclusion $h^{-1} \sU(R,\phi) \subset \sU(S,\psi)$ is an equality, \ie $h=\CZpec \theta$ restricts to a continuous map $\sU(\theta)
        \: \sU(S,\psi) \to \sU(R,\phi)$.
        \item The inclusion $g^{-1} \Schpec(R,\phi) \subset \Schpec(S,\psi)$ is an equality, \ie $g=\CSpec \theta$ restricts to a continous map $\Schpec \theta \:\Schpec(S,\psi) \to \Schpec(R,\phi)$).
    \end{enumerate}
    In that case, we obtain a commutative square in $\mathsf{Top}:$
     \[
\xymatrixcolsep{5pc}\xymatrix{
\sU(R,\phi)\ar[d]_-{\upbeta(R,\phi)} & \sU(S,\psi) \ar[l]_-{\sU(\theta)} \ar[d]^-{\upbeta(S,\psi)} \\
\Schpec(R,\phi) & \Schpec(S,\psi) \ar[l]_-{\Schpec \theta}
}
    \]
Furthermore, if $\theta \: R \to S$ is Cohen--Seidenberg then the above conditions are equivalent to:
\begin{enumerate}
    \item[(c)] 
    The inclusion $f^{-1}\Schpec(R,\phi)\subset \Schpec\spsi$ is an equality.
\end{enumerate}
\end{defprop}
\begin{proof}
Using \autoref{propdef f pure locus}, these are all equivalent to $\sqrt{\upsigma(R,\phi) S}=\sqrt{\upsigma(S,\psi)}$ (\ie a center of $F$-purity contracts to a center of $F$-purity). 
\end{proof}

\begin{theorem} \label{thm.FibrationsHomos}
    Let $\theta \: (R,\phi) \to (S,\psi)$ be a homomorphism of Cartier pairs. The following statements are equivalent:
    \begin{enumerate}
        \item It is fibered.
        \item The inclusion $\upkappa_{\q}(S,\psi) \cap R \subset \upkappa_{\q \cap R}(R,\phi)$ is an equality for all $\q \in \Spec S$.
        \item $\sqrt{\upkappa_{\fra}(R,\phi) S} \subset \upkappa_{\sqrt{\fra S}}(S,\psi)$ for all $\fra \in I(R)$.
    \end{enumerate}
    Moreover, if $\theta \: R \to S$ is Cohen--Seidenberg we may write an equality in (c) when $\fra = \sqrt{\fra}$.
\end{theorem}
\begin{proof}
Observe that both (c) in \autoref{defprop.QuasiFibered} and our current condition (b) are unconditionally true if $\q \notin \sU(S,\psi)$. That is, they are only interesting for what they say about $\q \in \sU(S,\psi)$. Note that (b) says $\upbeta_{\q}(S,\psi) \cap R \supset \upkappa_{\q \cap R}(R,\phi)$ for all $\q\in\sU(S,\psi)$. Therefore, (b) means that for all $\q \in \sU(S,\phi)$ we have $\q \cap R \in \sU(R,\phi)$ and $\upbeta_{\q}(S,\psi) \cap R \supset \upbeta_{\q \cap R}(R,\phi)$. This explains why (a) and (b) are equivalent. An analogous analysis lets us conclude the equivalence between (a) and (c). Indeed, note that (c) implies
\[
\sqrt{\upbeta_\fra S}=\sqrt{(\upkappa_\fra \cap \fra) S}\subset \sqrt{\upkappa_\fra S \cap \fra S} =\sqrt{\upkappa_\fra S}\cap \sqrt{\fra S}\subset \upkappa_{\sqrt{\fra S}} \cap \sqrt{\fra S}=\upbeta_{\sqrt{\fra S}}
\]
Further, plugging in $\fra = \p \notin \sU(R,\phi)$ in (c) yields $S=\upkappa_{\sqrt{\p S}}$ which means that $\q \notin \sU(S,\psi)$ for all $\q$ such that $\p \subset \q \cap R$. That is, (c) implies $\sqrt{\upbeta_\fra S} \subset \upbeta_{\sqrt{\fra S}}$ for all $\fra \in I(R)$ and that $\q \in \sU(S,\psi)$ contracts to $\q\cap R \in \sU(R,\phi)$. The converse holds as $\upkappa_{\sqrt{\fra S}}=\bigcap_{\q \in A} \upbeta_{\q}$ where $A \coloneqq f^{-1}V(\fra) \cap \sU(S,\psi)$ and for $\q\in A$ we have $\p \coloneqq \q \cap R \in \sU(R,\phi)$ and $\sqrt{\upkappa_{\fra}S} \subset \sqrt{\upbeta_{\p}S} \subset \upbeta_{\q}$.

When $\theta$ is Cohen--Seidenberg, the equality follows from \autoref{defprop.QuasiFibered}. Indeed, it suffices to test it for $\fra=\p$ prime, in which case it is only non-trivial if $\p \in \sU(R,\phi)$. It then follows from knowing that all minimal primes of $\sqrt{\p S}$ are in $\sU(S,\psi)$ as then $\upkappa_{\sqrt{\pS}} = \upbeta_{\sqrt{\p S}}$. 
\end{proof}

The following justifies the terminology ``fibered.''

\begin{corollary} \label{cor.FibereJustification}
    Let $\theta \: (R,\phi) \to (S,\psi)$ be a homomorphism of pairs and let $\bar{S}_{\p}$ denote the reduced fiber over $\p \in \Spec R$. Then, $\theta$ is quasi-fibered if and only if $\psi$ descends to an $\bar{S}_{\p}$ linear map $\bar{\psi}_{\p} \:F_*^e  \bar{S}_{\p} \to \bar{S}_{\p}$ for all $\p \in \CSpec(R,\phi)$. In that case, there is a canonical commutative diagram of Cartier pairs:
\[
\xymatrix{
(S,\psi) \ar[r] & \big(\bar{S}_{\p},\bar{\psi}_{\p}\big)\\
(R,\phi) \ar[u]^-{\theta} \ar [r] & \big(\kappa(\p),\bar{\phi}_{\p}\big) \ar[u]_-{\bar{\theta}_{\p}}
}
\]
Moreover, $\CSpec\big(\bar{S}_{\p},\bar{\psi}_{\p}\big) \to \CSpec(S,\psi)$ induces a homeomorphism onto $g^{-1}\{\p\}$ where $g= \CSpec \theta$ and centers of $F$-purity correspond to centers of $F$-purity. 

If  $\p \in \Schpec(R,\phi)$ then $\big(\bar{S}_{\p},\bar{\psi}_{\p}\big)$ is $F$-pure and so: $\bar{\psi}_{\p} \neq 0$ if $\bar{S}_{\p} \neq 0$. On the other hand, $\theta$ is fibered if and only if: $\bar{\psi}_{\p}\neq 0$ implies $\p \in \Schpec(R,\phi)$.
\end{corollary}

\begin{proposition} \label{pro.FormalUnramificationGivesFiberness}
A homomorphism of pairs $\theta \: (R,\phi) \to (S,\psi)$ is fibered if $\theta \: R \to S$ is formally unramified. If $\theta$ is fibered then $\kappa(\q)/\kappa(\p)$ is separable for all $\q \in \Schpec(S,\psi)$. 
\end{proposition}
\begin{proof}
   Use \autoref{cor.FirstCriterionSeparability} for the last statement. We explain next why if $\theta$ is formally unramified (and $S$ is $F$-finite by assumption) then its relative Frobenius $F_{\theta}^e \:S \otimes_R F^e_* R \to F^e_* S$ is surjective. Recall that a morphism of schemes is a closed immersion if and only if it is proper, unramified, and universally injective \cite[\href{https://stacks.math.columbia.edu/tag/04XV}{Tag 04XV}]{stacks-project}. In particular, an affine morphism of schemes is a closed immersion if and only if it is finite, unramified, and universally injective (since a flat closed immersion is open, this is used to show that an open immersion is nothing but an \'etale and universally injective morphism). Thus, we must show that $\Spec F_{\theta}^e$ is finite, unramified, and universally injective. Its finiteness follows from $S$ being $F$-finite. It is then unramified as $\Omega_{F^e_{\theta}} = \Omega_{\theta}$, which is zero as $\theta$ is formally unramified. Finally, $\Spec F^e_{\theta}$ is a universal homeomorphism (\cite[\href{https://stacks.math.columbia.edu/tag/0CCB}{Tag 0CCB}]{stacks-project}) and so universally injective.

    We must show that for all $\q \in \Spec R$ the inclusion $\upkappa_{\q}(S,\psi) \cap R \subset \upkappa_{\p}(R,\phi)$ is an equality, where $\p \coloneqq \q \cap R$. In fact, we prove that the containment $\q^{\psi^n}\cap R \subset (\q\cap R)^{\phin}$ is an equality for all $n  \in \bN$ (the case $n=0$ is trivial). In other words, we show that if $r \in R$ is such that the image of $\phi^n\cdot r$ is inside $\p$ then the image of $\psi^n \cdot r$ is inside $\q$. This follows right away from $F_{\theta}^e \:S \otimes_R F^e_* R \to F^e_* S$ being surjective as then the image of $\psi^n \cdot r$ is inside $\p S \subset \q$. 
\end{proof}

\begin{remark}
    Let $R \to S$ be formally unramified. If $\psi_1,\psi_2 \in \omega_{F^e_S} $ coincide on $F^e_*R$ then $\psi_1=\psi_2$. If $R \to S$ is formally \'etale, then every $\phi \in \omega_{F^e_R}$ extends uniquely to a map $\psi \in \omega_{F^e_S} $.
\end{remark}

\subsection{Categorical formulation} \label{sec.CategoricalFormulation}
In \autoref{defprop.QuasiFibered} and \autoref{defprop.Fibered}, the diagrams suggest that $\upbeta$ is a natural transformation. This can be made precise by defining a suitable category to interpret $\CZpec, \CSpec, \sU, \Schpec$ as functors from it into $\mathsf{Top}$. We define such category next.

\begin{remark}[{\cite{SchwedeFAdjunction,SchwedeCentersOfFPurity}}]
Let $R$ be a normal integral domain (or more generally a $\textbf{G}_1+\textbf{S}_2$ ring). To any non-degenerate $\phi\in \omega_{F^e}$ there corresponds an effective $\bZ_{(p)}$-divisor $\Delta_{\phi}$ on $X=\Spec R$ such that $K_R+\Delta_{\phi}\sim_{\bZ_{(p)}} 0$. Moreover, any such boundary is obtained in this way and two homomorphisms $\phi\in \omega_{F^e}$ and $\psi \in \omega_{F^d}$ yield the same boundary if and only if there is $n \in \langle e \rangle \cap \langle d \rangle \subset \bZ$ such that $\phi^{n/e} =\psi^{n/d} \cdot u$ for some unit $u\in R$. More generally, an effective $\bZ_{(p)}$-divisor $\Delta$ on $X$ such that $K_R+\Delta$ is $\bZ_{(p)}$-Cartier (say a \emph{tame boundary} on $X$) corresponds to a non-degenerate $R$-linear map $\varphi \: F^e_* R(D) \to R$ where $D$ is a Cartier divisor on $X$. A pair of non-degenerate $R$-linear maps $\varphi \: F^e_* R(D) \to R$ and $\psi \: F^d_* R(E) \to R$ where $D$ and $E$ are Cartier divisors on $X$ correspond to the same tame boundary if and only if there is $n \in \langle e \rangle \cap \langle d \rangle \subset \bZ$ such that $\varphi^{n/e} =\psi^{n/d} \circ F^{n}_*\upsilon$ where $\upsilon \: R(D) \to R(E)$ is an isomorphism of $R$-modules.
\end{remark}
\begin{remark} \label{rem.GeneralBoundaries}
Centers of $F$-purity are objects that one can more generally attach to $R$-linear maps of the form $\varphi \: F^e_* R(D) \to R$ for some Cartier divisor $D$. Indeed, a prime ideal $\p \subset R$ is a center of $F$-purity of $\varphi \: F^e_* R(D) \to R$ if $\varphi\big(F^e_* \p R(D)\big) \subset \p$ and $\varphi_{\p}$ is surjective. All we have done so far readily generalizes to this seemingly more general setup. However, $\p$ is a center of $F$-purity for $\varphi \: F^e_* R(D) \to R$ if and only if it is so for all those $\phi$'s with $D=0$ such that $\Delta_{\phi} \geq \Delta_{\varphi}$. Thus, we do not lose generality by reducing to the case in which $D=0$. Also, one may always work sufficiently locally when studying singularities and assume $D=0$ anyways. Thus, for the sake of simplicity, we restrict ourselves to the case $D=0$.
\end{remark}

\begin{definition}
    Let $R$ be a ring. We say that $\phi \in \omega_{F^e}$ and  $\psi \in \omega_{F^{d}}$ are \emph{boundary equivalent} if there is $n \in \langle e \rangle \cap \langle d \rangle \subset \bZ$ such that $\phi^{n/e} =\psi^{n/d} \cdot u$ for some unit $u\in R$. This induces an equivalence relation $\sim$ on $\bigsqcup_{e \in \bN_{+}} \omega_{F^e}$. We refer to an equivalence class as a \emph{tame boundary} on $R$. We let $\bigsqcup_{e \in \bN_{+}} \omega_{F^e} \to \sB_R$; $\phi \mapsto \Delta_{\phi}$, denote the corresponding quotient. That is, $\Delta_{\phi}$ denotes the  tame boundary of $\phi$.
\end{definition}

 One readily verifies that $I(R,\phi)\subset I(R,\phi^n)$ for all $n\in \bN_+$. The converse inclusion holds, for example, if $\phi$ is surjective. In general: 
\begin{defprop}\label{defprop boundary equivalence and F singularities}
    Let $R$ be a ring. Let $\omega_{F^e}\ni\phi \sim \psi \in \omega_{F^d}$ and $\Delta \in \sB_R$ be the corresponding boundary. Letting $I^{1/2}(R)$ be the radical ideals of $R$, we have:
    \begin{enumerate}
        \item $\upbeta_{\fra}(R,\phi)=\upbeta_{\fra}(R,\psi) \eqqcolon \upbeta_{\fra}(R,\Delta)$ for all $\fra \in I^{1/2}(R)$.
        \item $I(R,\phi)\cap I^{1/2}(R) =  I(R,\psi) \cap I^{1/2}(R) \eqqcolon I^{1/2}(R,\Delta)$.
        \item $\uptau_{\fra}(R,\phi)=\uptau_{\fra}(R,\psi) \eqqcolon \uptau_{\fra}(R,\Delta)$ for all non-degenerate $\fra \in I^{1/2}(R,\Delta)$.
\end{enumerate}
    Thus, $\upbeta(R,\Delta) \: \CZpec(R,\Delta) \to \CSpec(R,\Delta)$, $\sU(R,\Delta)$, and $\Schpec(R,\Delta)$ are well-defined.
\end{defprop}
\begin{proof}
    For (a), since $\fra$ and so $\upbeta_\fra$ are radical, it suffices to show that $ \upbeta_{\fra}(R,\phi) \subset \upbeta_{\fra}(R,\phi^n) \subset \sqrt{\upbeta_{\fra}(R,\phi)}$ for all $1< n \in \bN$. This is a consequence of the following.
    \begin{claim} \label{claim.DifferentTestElements}
        For every $a \in \bN$, if $r \in \fra^{\phi^{an}}$ then $r^{q^{n-1}} \in \fra^{\phi^{an+b}}$ for all $0\leq b\leq n-1$.
    \end{claim}
    \begin{proof}[Proof of claim]
        Note that:
        \begin{align*}
        \phi^{an+b}\big(F^{e(an+b)}_* r^{q^{n-1}}R\big) = \phi^{an}\Big(F^{ean}_* \phi^b\big(F^{eb}_* r^{q^{n-1}}  \big)R \Big) &= \phi^{an}\Big(F^{ean}_* r\phi^b\big(F^{eb}_* r^{q^{n-1-b}} \big) R\Big)\\
        &\subset \phi^{an}(F^{ean}_*rR)
        \end{align*}
        which is further contained in $\fra$ if $r \in \fra^{\phi^{an}}$. This shows the claim. \end{proof}
Part (b) follows from (a) and \autoref{proposition beta as an object}. For (c), it suffices to show $\uptau_\fra(R,\phi)=\uptau_\fra(R,\phin)$ for all $n \in \bN_{+}$. The inclusion $\supset$ is straightforward. The converse follows from \autoref{claim.DifferentTestElements} after taking $r$ a test element along $\fra$ and using the description of test ideals in terms of test elements in \cite[Lemma 4.12]{SmolkinSubadditivity}.
\end{proof}

\begin{definition}
    We refer to a pair $(R,\Delta)$ where $R$ is ring  and $\Delta \in \sB_R$ as a \emph{tame log ring}. A homomorphism of tame log rings $(R,\Delta_R) \to (S,\Delta_S)$ is a homomorphism of rings $\theta \: R \to S$ such that there exists $\phi\in \Delta_R$ and $\psi \in \Delta_S$ such that $\theta\: (R,\phi) \to (S,\psi)$ is a fibered homomorphism of Cartier pairs (\ie for all $\phi \in \Delta_R$ there is $n\gg 0$ such that $\phi^n$ lifts to a map $\psi$ with $\Delta_{\psi} = \Delta_S$ defining a fibered homomorphism). We denote the corresponding category by $\mathsf{TLRings}$---the \emph{category of tame log rings}.
\end{definition}

\begin{definition}
    We define four functors $\CZpec, \CSpec,\sU, \Schpec \:\mathsf{TLRings} \to \mathsf{Top}$.
\end{definition}

Our promised categorical formulation is the following. 

\begin{theorem}\label{theorem beta is retraction}
With notation as above, $\upbeta$ defines a natural retraction $\upbeta \: \CZpec \to \CSpec$ of $\CSpec \subset \CZpec$. Likewise, $\upbeta \: \sU \to \Schpec$ is a natural retraction of $\Schpec \subset \sU$.
\end{theorem}

\begin{remark}
    The commutative diagram in \autoref{cor.FibereJustification} defines a fibered coproduct if $\p \in \Schpec(R,\phi)$. The reason is that if $(A,\alpha)$ is a fibered extension of $(\kappa(\p),\bar{\phi}_{\p})$ then it is $F$-pure and so reduced. Thus, if it is further a fibered over $(S,\psi)$, then $S_{\p} \to A$ factors through $S_{\p} \to \bar{S}_{\p}$. However, general fibered coproducts do not exist in $\mathsf{TLRings}$. Indeed, let $R\coloneqq \F_2[x]$, $\phi\in \omega_{F_R}$,
    $S\coloneqq \F_2[y]$, and $\psi\in \omega_{F_S}$ be such that $\phi(\fstar 1)=\psi(\fstar 1)=1\in \F_2$. If the fibered coproduct $(R,\phi)\tensor_{(\F_2,\id)} (S,\psi)$ were to exist in $\mathsf{TLRings}$, one would have that there is a unique extension $F_* \bF_2[x,y] \to \bF_2[x,y]$ of both $\phi$ and $\psi$. However, there are as many extensions as elements in $\bF_2[x,y]$, for all we need to do is to specify is where to send $F_*xy$.
\end{remark}

\section{Fibered Transpositions} \label{sec.FiberedTranspositions}
In this section, we figure out how bad the ramification can be so that a crepant map $\theta\: (R,\Delta_R) \to (S,\Delta_S)$ between tame log rings is fibered. In particular, we prove the Main Theorem from the introduction.

\begin{definition}[{\cite[Definition 5.3]{SchwedeTuckerTestIdealFiniteMaps}, \cite[\S3]{CarvajalStablerFsignaturefinitemorphisms}}]\label{def.Transposition}
With notation as in \autoref{def.TranspositionOfIdeals}, we refer to a commutative diagram
\[
\xymatrix{
F^e_* S \ar[d]_-{F^e_* \frT} \ar[r]^-{\psi} & S  \ar[d]^-{\frT} \\
F^e_* R  \ar[r]^-{\phi} & R
}
\]
with $\frT$ non-degenerate as a \emph{transposition of Cartier pairs} $(\theta,\frT)\:(R,\phi) \to (S,\psi)$. We may write $\phit$ instead of $\psi$ and refer to it as a \emph{$\frT$-transpose} of $\phi$.
\end{definition}

\begin{proposition} \label{trace commutes with frobenius} Let $\thetat\colon\rphi\to\sphit$ be a transposition of Cartier pairs that is also a homomorphism of Cartier pairs. If $\phi$ is non-degenerate then $\frT$ commutes with $F^e$.
\end{proposition}
\begin{proof}
Observe that the following equalities hold for all $r \in R$, $s\in S$:
\[
\phi(F^e_* (\frT(s^q)r))=\phi(F^e_* \frT(s^q r)) = \frT(\psi(F^e_* rs^q)) 
 = \frT(s\psi(F^e_* r))
 = \phi( F^e_* r)\frT(s) 
 = \phi( F^e_* (\frT(s)^q r)).
\]
Hence, $\phi \cdot s'=0$ for $s'\coloneqq \frT(s^q)-\frT(s)^q$ and so $s'=0$ as $(R,\phi)$ is non-degenerate.
\end{proof}
\begin{proposition}[\cf {\cite[Corollary 4.2]{SchwedeTuckerTestIdealFiniteMaps}}] \label{equiv transposition and extension}
With notation as in \autoref{def.Transposition}, suppose that $\frT$ is strongly non-degenerate and commutes with $F^e$. Then, $(R,\phi) \to (S,\psi)$ is a homomorphism if it is a transposition. The converse holds if $R$ and $S$ are integral.
\end{proposition}
\begin{proof}
Suppose that $(R,\phi) \to (S,\psi)$ is a $\frT$-transposition. Let $r\in R$ and $s\in S$. Then:
\[ \frT(\psi(\fstare r)s) = \phi(\fstare \frT(rs^q)) 
 = \phi(\fstare (r\frT(s^q)))
 = \phi(\fstare r(\frT(s))^q) 
 = \frT(s)\phi(\fstare r)
 = \frT(\phi(\fstare r)s).
 \]
Since $\frT$ is (strongly) non-degenerate, $\psi(\fstare r)=\phi(\fstare r)$ for all $r\in R$; as required.

Conversely, suppose that $(R,\phi) \to (S,\psi)$ is a homomorphism where $S$ and $R$ are integral. In particular, $S/R$ is generically separable by \autoref{second criterion separability}. Then, to prove that $\phi \circ F^e_*\frT - \frT \circ \psi \: F^e_* S \to R$ is zero, it suffices to do it after composing it with $F^e_{\theta}\: S \otimes_R F^e_* R \to F^e_* S$ and further with both $S \to  S \otimes_R F^e_* R$ and  $F^e_* R \to S \otimes_R F^e_* R$, in which case it is direct.
\end{proof}

With notation as in \autoref{def.Transposition}, we consider the map 
\[h\coloneqq\CZpec(\theta,\frT) \: \CZpec(S,\phit) \to \CZpec(R,\phi),\quad \q \mapsto \q \cap R.\] However, $\q \mapsto \q \cap R$ no longer yields a map $g \: \CSpec(S,\phit) \to \CSpec(R,\phi)$ as the contraction of a $\phit$-ideal need not be a $\phi$-ideal. In some cases, however, one can ensure that the contraction of a $\phit$-ideal is a $\phi$-ideal; see \autoref{section contracting compatible stuff}. Hence, we do not get the inclusion $\upbeta_{\q}\cap R\subset \upbeta_{\q \cap R}$ for free. To fix this, we must define 
\[
\CSpec(\theta,\frT) \: \CSpec(S,\phit) \to \CSpec(R,\phi), \quad g\coloneqq h/{\sim} \: \q \mapsto \q \cap R \mapsto \upbeta_{\q \cap R}\rphi.
\] 

\begin{proposition}\label{Pro.EquivalencePreQuasiFibered}
Let $(\theta, \frT) \: (R,\phi) \to (S,\phit)$ be a transposition of Cartier pairs. Then, the following statements are equivalent: 
\begin{enumerate}
    \item The map $\CZpec(\theta,\frT) \:\CZpec(S,\phit) \to \CZpec(R,\phi)$ is continuous.
    \item The following diagram is conmutative
    \[
\xymatrixcolsep{5pc}\xymatrix{
\CZpec(R,\phi)\ar[d]_-{\upbeta(R,\phi)} & \CZpec(S,\phit) \ar[l]_-{\CZpec(\theta,\frT)} \ar[d]^-{\upbeta(S,\phit)} \\
\CSpec(R,\phi) & \CSpec(S,\phit) \ar[l]_-{\CSpec(\theta,\frT)}
}
    \]
    \item $\upbeta_{\q \cap R}\rphi \subset \upbeta_{\q}\sphit \cap R$ for all $\q \in \Spec S$.
    \item $\sqrt{\upbeta_{\fra}\rphi S}  \subset \upbeta_{\sqrt{\fra S}}\sphit$ for all $\fra \in I(R)$.
    \item $\sqrt{\upbeta_{\p}\rphi S}  \subset \upbeta_{\sqrt{\p S}}\sphit$ for all $\p \in \Spec R$.
\end{enumerate}    
    Furthermore, if $\theta \: R \to S$ is Cohen--Seidenberg (\ie going-down holds, \eg either it is flat or $R$ and $S$ are integral domains with $R$ normal)
then the inclusion in (d) (and so in (e)) is an equality if $\fra$ is radical. In that case, (e) is equivalent to each of the following two statements:
\begin{enumerate}
     \item[(f)] $f^{-1}\{\upbeta_{\p}\rphi\} = \upbeta\sphit \big(f^{-1}\{\p\}\big)$ for all $\p \in \Spec R$.
    \item[(h)] $f^{-1}  \CSpec(R,\phi)\subset \CSpec\sphit$.
\end{enumerate}
\end{proposition}
\begin{proof}
The proof of \autoref{defprop.QuasiFibered} applies \emph{mutatis mutandis} here. For instance, if (b) holds then $\upbeta_{\q\cap R}\rphi=\upbeta_{\upbeta_\q\sphit \cap R}\rphi\subset\upbeta_\q\sphit \cap R$ for all $\q$, and so (c) holds.
\end{proof}

\begin{definition} \label{def.FiberedQuasiFiberedTransposition}
A transposition of pairs $\thetat\colon\rphi\to\sphit$ is \emph{quasi-fibered} if:
\begin{enumerate}
    \item any of the equivalent conditions of \autoref{Pro.EquivalencePreQuasiFibered} holds,
    \item $\pt$ is radical for all $\p\in \CSpec\rphi$, so that there is an induced commutative diagram
\[
\xymatrix{
F^e_* \bar{S}_{\kappa(\p)} \ar[d]_-{F^e_* \bar{\frT}_{\p} } \ar[r]^-{\bar{\phi}_{\p}^{\frT}} &  \bar{S}_{\kappa(\p)}  \ar[d]^-{ \bar{\frT}_{\p}} \\
F^e_* \kappa(\p)  \ar[r]^-{\bar{\phi}_{\p}} & \kappa(\p)
}
\]
\item and $\bar{\phi}_\p\neq 0$ implies $\bar{\phi}_\p\transp$ is $F$-pure.
\end{enumerate}
We say that $\thetat$ is \emph{fibered} if moreover
\begin{enumerate}[resume]
    \item $\bar{\phi}_\p=0 $ implies $\bar{\phi}_\p\transp=0$ for all $\p\in \CSpec\rphi$.
\end{enumerate}
\end{definition}

\begin{remark}
In \autoref{def.FiberedQuasiFiberedTransposition}, 
(a) and (d) combined are equivalent to each of the following:
\begin{enumerate}
     \item $h$ is continuous and it restricts to a map $\sU\thetat \: \sU(S,\phit) \to \sU(R,\phi)$.
    \item  $h$ restricts to a map $\sU\thetat \: \sU(S,\phit) \to \sU(R,\phi)$, so that $\CSpec(\theta,\frT)$ restricts to a map $\Schpec\thetat\:\Schpec(S,\phit) \to \Schpec(R,\phi)$, and the following diagram is commutative
    \[
\xymatrixcolsep{5pc}\xymatrix{
\sU(R,\phi)\ar[d]_-{\upbeta(R,\phi)} & \sU(S,\phit) \ar[l]_-{\sU(\theta,\frT)} \ar[d]^-{\upbeta(S,\phit)} \\
\Schpec(R,\phi) & \Schpec(S,\phit) \ar[l]_-{\Schpec(\theta,\frT)}
}
    \]
    \item $\upkappa_{\q \cap R}\rphi \subset \upkappa_{\q}\sphit \cap R$ for all $\q \in \Spec S$.
    \item $\sqrt{\upkappa_{\fra}\rphi S}  \subset \upkappa_{\sqrt{\fra S}}\sphit$ for all $\fra \in I(R)$.
\end{enumerate}    
    When $\theta \: R \to S$ is Cohen--Seidenberg, the inclusion in (d) is an equality for $\fra$ radical.
\end{remark}

\begin{corollary}
Let $\thetat\colon \rphi\to\sphit$ be a transposition of Cartier pairs 
that is also a homomorphism of Cartier pairs. Then, it is quasi-fibered  (resp. fibered) as a transposition if and only if it is quasi-fibered (resp. fibered) as a homomorphism.
\end{corollary}
\begin{proof}
Apply \autoref{pro.SeparableVSpurelyInseparable}, \autoref{cor.FibereJustification}.
\end{proof}

\begin{remark}\label{Rem.CommDiagNotTransposition}
Let $(\theta, \frT) \: (R,\phi) \to (S,\phit)$ be a quasi-fibered transposition. The diagram in (b) is not referred to as a transposition as $\bar{\frT}_{\p}$ may be zero (\ie $\pt = \mathbf{1}$). Nevertheless, if $\bar{\frT}_{\p} \neq 0$ then we would have (c) for free.  Likewise, for all $\p \in \CSpec(R,\phi)$, by \autoref{notation.Ranks}, there are induced commutative diagrams on residue fields 
\[
\xymatrix{
F^e_* \kappa(\q) \ar[d]_-{F^e_* \bar{\frT}_{\q} } \ar[r]^-{\bar{\phi}_{\q}^{\frT}} & \kappa(\q)  \ar[d]^-{ \bar{\frT}_{\q}} \\
F^e_* \kappa(\p)  \ar[r]^-{\bar{\phi}_{\p}} & \kappa(\p)
}
\]
for all $\q\in Y_\p$. However, this is not necessarily a transposition as $\bar{\frT}_{\q}$ may vanish. In fact, this is a transposition for all $\q\in Y_\p$ if and only if $\theta$ is tamely $\frT$-ramified over $\p$ by \autoref{cor.DifferentWaysToThinkOfTameRamification}. In that case, it would follow that $(\theta,\frT)$ is fibered.
\end{remark}

Using the transitivity of transpositions (\ie \autoref{pro.TransitivityTransposition}), we see right away that Cartier cores commute with transpositions:

\begin{proposition}\label{pro.betaCommutesWithTransposeAndApplications}
    Let $(\theta,\frT) \: (R,\phi) \to (S,\phit)$ be a transposition of Cartier pairs. Then, $ \upbeta_{\fra}\rphi^{\frT} = \upbeta_{\fra^{\frT}}\sphit$ for all $\fra\in I(R)$.
    In particular, if $\fra \in I(R,\phi)$ then $\fra^{\frT} \in I(S,\phit)$.
\end{proposition}
\begin{proof} Let $\fra\in  I(R)$.  Observe that $\phi^{\frT, n}=(\phit)^n$ is a $\frT$-transposition of $\phi^n$ for all $n\in \bN$, so we may write $\phi^{n\frT}=(\phit)^n$. By \autoref{pro.TransitivityTransposition}, we then see that $(\fra^{\phin})^{\frT}= \fra^{\phin \circ \frT}=\fra^{\frT \circ \phi^{n\frT}} = (\fra^{\frT})^{\phi^{n\frT}}$.
Taking the intersection over $n \in \bN$ gives the desired equality.
\end{proof}

\begin{remark}\label{rem tame ram in purely insep}
Let $\thetat\colon\rphi\to\sphit$ be a purely inseparable transposition of Cartier pairs and $\p\in \CSpec\rphi$. If $\pt$ is proper (\eg $\frT$ is surjective) then the previous proposition implies that $\q\coloneqq\sqrt{\pS}=\sqrt{\pt}\in \CSpec\sphit$. In fact, using \autoref{Pro.NonDegeneracyPhiIdeals}, one can see that $\theta$ is tamely $\frT$-ramified over $\p$ if $\upsigma(\mathbf{1}_S,\phit)\nsubset \q$ \ie $\q\in \Schpec\sphit$.
\end{remark}

The following lemma is implicit in the proof of \cite[Theorem 5.12]{CarvajalStablerFsignaturefinitemorphisms}.

\begin{lemma} \label{lem.TriviaInlcusionTestIdeals}
     Let $(\theta,\frT) \: (R,\phi) \to (S,\phit)$ be a transposition of Cartier pairs and $\q \in \Schpec (S,\phit)$ be such that $\p \coloneqq \q \cap R \in \Schpec(R,\phi)$. Then $\uptau_{\p}(R,\phi) \supset \frT\big(\uptau_{\q}(S,\phit)\big)$.
\end{lemma}
\begin{proof}
    Let $c \in S$ be such that for all $s\in S \setminus \q$ there is $n$ and $s'\in S$ such that $\phi^{\frT,n}(F^{en}_*s s') = c$ \ie $c$ is a test element along $\q$. It suffices to prove that, for every $s''\in S$, $\frT(s''c)$ has the following property: for all $r \in R\setminus \p$ there is $n\in \N$ and $r'$ such that $\phi^n(F^{en}_* rr')=\frT(s''c)$. Let $s''\in S$ and $r \in R \setminus \p$, so that $r \in S \setminus \q$. There is $n\in \N$ and $s'\in S$ such that $\phi^{\frT,n}(F^{en}_*r s') = c$ and so $\phi^{\frT,n}(F^{en}_*r s' s''^{q^n}) = s''c$. Applying $\frT$ to this last equality yields $
    \phi^n\big(F^{en}_* r \frT(s' s''^{q^n})\big) = \frT(s''c)$
    so $r' \coloneqq \frT(s' s''^{q^n}) $ works.
\end{proof}

\begin{proposition}\label{pro.minimalPrimesOfptInCSpec}
     Let $(\theta,\frT) \: (R,\phi) \to (S,\phit)$ be a transposition of Cartier pairs and $\p\in \Schpec\rphi$. Then, the minimal primes of $\p^{\frT}$ (if any) belong to $Y_{\p} \cap \CSpec(R,\phit)$. If $\pt$ is radical then its minimal primes belong to $Y_{\p} \cap \Schpec(R,\phit)$. If a minimal prime $\q$ of $\p^{\frT}$ is in $\Schpec(R,\phit)$ then  $\uptau_\p\rphi=\frT(\uptau_\q\sphit)$. Conversely, if $\q \in Y_{\p} \cap \Schpec(R,\phit)$ is such that $\uptau_\p\rphi=\frT(\uptau_\q\sphit)$ then $\q$ is a minimal prime of $\p^{\frT}$.
\end{proposition}
\begin{proof}
The first statement follows from \autoref{pro.betaCommutesWithTransposeAndApplications} and \autoref{first proposition properties of comp ideal}. For the second one, suppose that $\q$ is a minimal prime of $\p^{\frT}$ that is not a center of $F$-purity. By \autoref{MinimalPrimesTranspostions}, there is $s\in S \setminus \p^{\frT}$ such that $\frT(s\q)\subset \p$. By assumption, $\phi^{\frT}(F^e_* S) \subset \q$ and so $\frT(s\phi^{\frT}(F^e_*S))\subset \frT(s\q)\subset \p$. In particular, $\phi(F^e_* \frT(s^qS)) \subset \p$ and so $\frT(s^qS) \subset \p$ as $\p \in \Schpec(R,\phi)$. Hence, $s^q \in \p^{\frT}$ and so $\p^{\frT}$ is not radical. For the other claims, note that for all $\q \in Y_{\p} \cap \Schpec(R,\phit)$:
\[
\p^{\frT} \subset \q \Longleftrightarrow \uptau_{\q}(S,\phit) \not\subset \p^{\frT} \Longleftrightarrow \frT\big(\uptau_{\q}(S,\phit)\big) \not\subset \p   \Longleftrightarrow  \uptau_{\p}(R,\phi) \subset \frT\big(\uptau_{\q}(S,\phit)\big).
\]
Then one applies \autoref{lem.TriviaInlcusionTestIdeals}.
\end{proof}

\begin{corollary}\label{Cor.fiberedAtAPoint}
Let $(\theta,\frT)\: \rphi \to \sphit$ be a transposition of Cartier pairs and $\p \in \Schpec(R,\phi)$. Then, $S/R$ is tamely $\frT$-ramified over $\p$ if and only if $Y_{\p} \subset \Schpec(S,\phit)$ and $\frT(\uptau_{\q}(R,\phit))=\uptau_{\p}(R,\phi)$ for all $\q \in Y_{\p}$. In that case, we obtain residual transpositions:
   \[
\xymatrix{
F^e_* \kappa(\q) \ar[d]_-{F^e_*  \Bar{\frT}_{\q}} \ar[r]^-{\bar{\phi}_{\q}^{\frT}} & \kappa(\q)  \ar[d]^-{\bar{\frT}_{\q}} \\
F^e_* \kappa(\p)  \ar[r]^-{\bar{\phi}_{\p}} & \kappa(\p)
}
\]
\end{corollary}
\begin{proof}
The forward implication follows at once from \autoref{pro.minimalPrimesOfptInCSpec}. The converse follows if $\p^{\frT}$ is radical. For this, note that the minimal primes of $\pt$ are those of $\sqrt{\pS}$ and so $\pt$ is radical by \autoref{Pro.NonDegeneracyPhiIdeals} and \autoref{pro.betaCommutesWithTransposeAndApplications}.
\end{proof}

Putting everything together, we obtain the following: 

\begin{corollary}[Main Theorem] \label{cor.MainTheorem}
    Let $(\theta,\frT)\: \rphi \to \sphit$ be a transposition of Cartier pairs. Consider the following statements.
    \begin{enumerate}
        \item $\theta$ is tamely $\frT$-ramified over $\CSpec(R,\phi)$.
        \item $(\theta,\frT)$ is fibered and $\uptau_{\p}\rphi=\frT(\uptau_\q\sphit)$ for all $\p\in \Schpec\rphi$ and $\q \in Y_{\p}$.
        \item $\theta$ is tamely $\frT$-ramified over $\Schpec(R,\phi)$.
    \end{enumerate}
    Then, $(a)\Longrightarrow (b) \Longrightarrow (c)$. In particular, they are all equivalent if $(R,\phi)$ is $F$-pure.
\end{corollary} 
\begin{proof}
It follows directly from \autoref{Rem.CommDiagNotTransposition} and \autoref{Cor.fiberedAtAPoint}.
\end{proof}

\begin{corollary} \label{cor.GaloisCase}
    Work in \autoref{setup.ClassicAlgebraicNumberTheorySetup} and let $(R,\phi) \to (S,\phi^{\Tr})$ be a homomorphism of $F$-pure Cartier pairs. If $S/R$ is tamely ramified over $\Schpec(R,\phi)$ then $\Tr\: S \to R$ is surjective. The converse holds if $L/K$ is Galois (in which case $S/R$ is tamely ramified everywhere).
\end{corollary}
\begin{proof}
By \autoref{rem.FPurityAsPureFRgularity}), $(R,\phi)$ is purely $F$-regular along $\fra$ where the minimal primes of $\fra$ are the maximal centers of $F$-purity of $(R,\phi)$. For every $\p$ maximal center of $F$-purity of $(R,\phi)$ one applies \autoref{cor.MainTheorem} to a $\q$ lying over $\p$ to conclude the desired surjectivity.
\end{proof}

\begin{lemma}[{\cf \cite[Remark 5.14]{CarvajalStablerFsignaturefinitemorphisms}}]
     Working in \autoref{setup.ClassicAlgebraicNumberTheorySetup}, $(R,\phi) \to (S,\phi^{\Tr})$ be a homomorphism of Cartier pairs and $\p \in \Schpec(R,\phi)$ be of height $\leq 1$. Then, $S/R$ is tamely ramified over $\p$ if (and only if) $Y_{\p} \subset \Schpec(R,\phi^{\Tr})$.
\end{lemma}
\begin{proof}
If $\height \p = 0$ (\ie $\p=0$) the result is simply \autoref{cor.FirstCriterionSeparability}. Suppose $\height \p = 1$. Recalling \autoref{rem.CFPsOfHeight1}, the result then follows from the remarks made in the proof \autoref{prop.TameRamificationParticularCases}.
\end{proof}

\begin{corollary}[Main Result] \label{cor.MainResult}
     Work in \autoref{setup.ClassicAlgebraicNumberTheorySetup} and let $\theta \:(R,\phi) \to (S,\phi^{\Tr})$ be a homomorphism of $F$-pure Cartier pairs. Then, $S/R$ is tamely ramified over $\Schpec(R,\phi)$ if (and only if) $\theta$ is fibered and for all $\q \in \Schpec(R,\phi^{\Tr})$ of height $\geq 2$,  $\Tr(\uptau_\q(S,\phi^{\Tr}))=\uptau_{\q\cap R}\rphi$.
\end{corollary}

\begin{example}
    The test ideals transformation rule in \autoref{cor.MainResult} cannot be dropped. Indeed, let $R\coloneqq \bar{\bF}_p[x,y,z]/\langle x^3+y^3+z^3\rangle$ with $p \equiv 1 \bmod 3$. Note that $R$ is a standard graded Gorenstein normal domain. In fact, $R$ is the affine cone of the ordinary/$F$-split elliptic curve $E\subset \bP^2$ cut out by $x^3+y^3+z^3=0$. Let $\phi \: F^e_* R \to R$ be a Frobenius trace for $R$, which is surjective as $E$ is $F$-split. One readily sees that $\Schpec(R,\phi)=\{\mathbf{0},\fram\}$ where $\fram \coloneqq \langle x, y, z \rangle$, so that we may assume that $\phi$ is an $F$-splitting. To construct $S$, we may proceed as follows. 
    
    First, recall that $F$ induces a $p$-linear action on the graded $\bar{\bF}_p$-module $H^2_{\fram}(R)=H^1(U,\sO_U)$, where $U \coloneqq \Spec R \setminus \{\fram\}$. In fact, by the cone construction abstract nonsense, we have a $\bG_{\mathrm{m}}$-bundle map $\pi \: U \to E$ such that $\pi_* \sO_U = \bigoplus_{i\in \bZ} \sO_E(i)$ and $\pi^{\#} \:\sO_E \to \pi_* \sO_U$ is the canonical structural map onto degree $0$. Moreover, this realizes an isomorphism $H^1(U,\sO_U)=\bigoplus_{i \in \bZ}H^1(E,\sO_E(i))$ of graded $\bar{\bF}_p$-modules and the Frobenius action is given by the canonical maps $\sO_E(i) \to F_* \sO_E(pi)$. In particular, the $F$-action is not graded as $F$ maps $[H^2_{\fram}(R)]_i$ into $[H^2_{\fram}(R)]_{pi}$. Nonetheless, $[H^2_{\fram}(R)]_0 = H^1(E,\sO_E) \cong \bar{\bF}_p$ is an $F$-stable submodule and moreover:
    \[
    H \coloneqq \{\xi \in H_{\fram}^2(R) \mid F \xi = \xi \} =  \{\xi \in H^1(E,\sO_E) \mid F \xi = \xi \} \cong \bF_p.
    \]
    where the latter isomorphism is a consequence of $R$ and $E$ being $F$-split. By Artin--Schreier theory \cite[p.127]{MilneEtaleCohomology}, to a $0 \neq \xi \in H$ there corresponds a non-trivial $\bZ/p$-torsor $\xi \: C \to E$ whose pullback along $\pi \: U \to E$ gives a non-trivial $\bZ/p$-torsor $\xi \: V \to U$. By the Hurwitz formula, $C$ is an elliptic curve. In fact, $\xi\: C \to E$ is $E$'s Verschiebung (auto)morphism.
    
    Letting $S$ be the normalization of $R$ in the function field of $V$, we see that $(S,\fran,\bar{\bF}_p)$ is a local normal integral finite extension of $(R,\fram,\bar{\bF}_p)$ whose spectrum $f\: \Spec S \to \Spec R$ is such that $f_U=\xi\: V \to U$. In particular, $S$ is Gorenstein and $\phi$ extends to a unique Frobenius trace $\psi \: F_*S \to S$ on $S$, which is necessarily an $F$-splitting. Moreover, $(S,\psi)$ is not $F$-regular as $H^2_{\fran}(S)=H^1(V,\sO_V)$ contains $H^1(C,\sO_{C}) \cong \bar{\bF}_p$ as an $F$-stable submodule (for $\pi$ pulls back to a $\bG_{\mathrm{m}}$-bundle $V \to C$ along $\xi \: C \to E$). Therefore, $\Schpec(S,\psi) = \{\mathbf{0},\fram\}$. In conclusion, $(R,\phi)\to (S,\psi)$ is a fibered extension but it is not tamely ramified over $\fram$ as the ramification index is $e_{\fram}=p$.
\end{example}

\subsection{On the work of Speyer} \label{sec.SpeyerCharatcteristicZero}
    We may rephrase Speyer's theorem \cite[Theorem 3]{SpeyerFrobeniusSplit} in our terminology as follows. 
    \begin{theorem}[{Speyer \cite[Theorem 3]{SpeyerFrobeniusSplit}}]
         Let $\theta \:R \to S$ be a finite cover between $\bZ$-algebras of finite type. Then, there is a dense open $U \subset \Spec \bZ$ such that for all $p\in U$: every homomorphism of $F$-pure Cartier pairs $\theta_{p}\coloneqq \theta \otimes \bF_p \:(R_{p},\phi) \to (S_{p},\psi)$ is fibered.
    \end{theorem}
    
    It is unclear to the authors whether an analogous statement for tame ramification in arithmetic families hold and whether $\bZ$ can be replaced by any $\bZ$-algebra of finite type (other than rings of integers in number fields, see \cite[Remark 4]{SpeyerFrobeniusSplit}). The difficulty in either case is that $R$ and $S$ are not assumed to be geometrically normal. Under this assumption, we obtain the following version of Speyer's result using \autoref{cor.TameRamificationInArithmeticFAMILIES}.

\begin{corollary} \label{cor.SpeyerResultNormalVersion}
    With notation and assumptions as in \autoref{cor.TameRamificationInArithmeticFAMILIES}, every homomorphism of Cartier pairs $(R_{\mu},\phi) \to (S_{\mu},\psi)$ is fibered for all closed points $\mu \in U$.
\end{corollary}

\begin{remark}
One might be tempted to use this to conclude an analogous statement of our Main Result for log canonical centers in characteristic zero via reductions to positive characteristic. However, the issue is that the converse of \cite[Theorem 6.8]{SchwedeCentersOfFPurity} is not known beyond the divisorial case; see \cite[Theorem 2.23]{CarvajalRojasStaeblerTameFundamentalGroupsPurePairs}. The difficulties are discussed in \cite[Theorem 6.2, Question 7.1]{DasSchwedeFdifferencanoncialbundleformula}.
\end{remark}

\subsection{Diagonalizable quotients, revisited}\label{section contracting compatible stuff}

We resume where we left in \autoref{setting workout example}. Since $\frT$ is non-singular, every map $\phi \:F^{e}_*R\to R$ admits a unique $\frT$-transpose $\phi^{\frT}\:F^e_*S \to S$ \cite[Theorem 3.13]{CarvajalStablerFsignaturefinitemorphisms}. We want to analyze the behavior of centers of $F$-purity in this case under the assumption that $(R,\fram,\kay)$ is local and $\Gamma \to \Cl R$ is injective, which we refer to as the Veronese local case. Then, $(S,\fram S,\kay)$ is an $\mathbf{S}_2$ local integral domain. Moreover, $(R,\phi)$ is $F$-pure if and only if so is $(S,\phit)$ by \cite[Scholium 4.9]{CarvajalFiniteTorsors}. In that case:

\begin{theorem} \label{thm.MainTheoremPurelyInseparableStuff}
    With notation as above, $\theta$ is tamely $\frT$-ramified over $\Schpec\rphi$ and so fibered. Furthermore, the following equivalences hold for all $\p \in \Spec R$:
    \[
    \p \in \Schpec\rphi \Longleftrightarrow Y_\p \subset \Schpec\sphit \Longleftrightarrow Y_{\p} \cap \Schpec\sphit \neq \emptyset
    \]
    Thus, $\Spec S \to \Spec R$ restricts to a homeomorphism $ \Schpec\sphit\to \Schpec\rphi$ if $\Gamma$ is a $p$-group. 
\end{theorem}

\begin{remark}
This generalizes \cite[Lemmas 3.7, 3.8]{PolstraSchwedeCompatibleIdealsGorensteinRings} (in the normal integral case) as we do not need the uniformly compatible hypothesis nor the cyclicity of the cover. This also answers positively the question they asked in \cite[Remark 3.6]{PolstraSchwedeCompatibleIdealsGorensteinRings}. Furthermore, one readily sees that if every $I(S,\phi^{\frT})$ is a trace ideal (in the terminology of \cite{PolstraSchwedeCompatibleIdealsGorensteinRings}) then so is every ideal in $I(R,\phi)$. In particular, we obtain another way to see that their main result in the $\bQ$-Gorenstein case follows from the Gorenstein case even if $p$ divides the Gorenstein index.
\end{remark}

To show \autoref{thm.MainTheoremPurelyInseparableStuff}, we need some remarks in the not necessarily Veronese local case.

\begin{lemma}\label{lemma stuff goes to 0}
Let $\thetat\colon\rphi\to\sphit$ be a transposition, $s\in S$, and assume that $q\geq |\Gamma|=p^n$. Then, $\phi^{\frT}(F^e_*s)=0$ if $\frT(s)=0$.
\end{lemma}
\begin{proof}
Since $\frT$ is non-singular (\autoref{pro.QuasitorsoNonSingularity}), it suffices to show that $\frT(s'\phit(\fstare s))=0$ for all $s'\in S$. Note that $\frT(s'\phit(\fstare s)) = \frT(\phit(\fstare s'^q s)) = \phi(\fstare \frT(s'^q s)) = \phi(\fstare s'^q \frT(s)) = 0$ as $s'^q\in S_0=R$ (which uses that $q\geq |\Gamma|$).
\end{proof}

\begin{proposition}\label{small case contract keeps compatibility}
Let $\rphi\to\sphit$ be a transposition and $\q\in \Schpec\sphit$ with $\Gamma$ a $p$-group. Then, $\q$ is homogeneous and $\q\cap R\in \Schpec\rphi$.
\end{proposition}
\begin{proof} 
We may replace $\phi$ by any sufficiently large power $\phi^n$ and assume that $q \geq |\Gamma|$ (\autoref{defprop boundary equivalence and F singularities}). We claim that $\upkappa_{\q'}$ is homogeneous for all $\q' \in \Spec S$. Let $s \in S$ and $s_{\gamma} \coloneqq \pi_{\gamma}(s)$. \autoref{lemma stuff goes to 0} implies:
\[
\phi^n(F^e_*s_{\gamma}S)\subset \q' \Longleftrightarrow \phi^n(F^e_*s_{\gamma}S_{-\gamma})\subset \q' \Longleftrightarrow \phi^n(F^e_*sS_{-\gamma})\subset \q'
\]
Then $\upkappa_{\q'}$ is homogeneous.
In particular, if $\q\in \Schpec\sphit$, $\q\cap R=\frT(\q)$ is compatible.

It remains to show that $\upsigma(R,\phi)\not\subset \q\cap R$. Suppose the contrary.
By \autoref{transpose is equivariant}, $(\q\cap R)\transp=\q$ as $\q$ is homogeneous. Since $\frT(\upsigma(S,\phit))\subset\upsigma(R,\phi)$, then
$\upsigma(S,\phit) \subset \upsigma(R,\phi)\transp \subset (\q\cap R)\transp = \q$, which contradicts $\q\in \Schpec\sphit$.
\end{proof}

\begin{proof}[Proof of \autoref{thm.MainTheoremPurelyInseparableStuff}]
Write $\Gamma = \Lambda \oplus \Upsilon$ where $\Upsilon$ is a $p$-group and $p\nmid |\Lambda|$. Then, we have a factorization $R\subset T \subset S$ where $T=\bigoplus_{\lambda \in \Lambda } S_{\lambda}$ and $T/R$ is quasi-\'etale. In particular, $T$ is a normal integral domain. Then, $S=\bigoplus_{\upsilon \in \Upsilon} S'_{\upsilon}$ where $S'_{\upsilon}\coloneqq \bigoplus_{\lambda \in \Lambda } S_{\lambda+\upsilon} = T(f^*D_{\upsilon})$ and $f$ is the induced finite cover $\Spec T \to \Spec R$ between normal integral domains. Therefore, it suffices to show the proposition when $p\nmid |\Gamma|$ and $\Gamma$ is a $p$-group separately. If $p\nmid |\Gamma|$, this follows from \autoref{transpose in divisorial stuff}. If $\Gamma=q'$, this is \autoref{rem tame ram in purely insep} and \autoref{small case contract keeps compatibility}.
\end{proof}

\begin{question}
    In the same vein of \autoref{que.Computep^TForGneralQuasiTorsors}, how can we generalize \autoref{thm.MainTheoremPurelyInseparableStuff} to more general finite $G$-quasitorsors by (infinitesimal) group-schemes $G$?
\end{question}

\begin{remark}
    Let $X$ be a normal variety over a perfect field $\kay$ of characteristic $p$. Then, to a \emph{foliation} $\sF \subset \sT_{X/\kay}$ there corresponds a purely inseparable cover of height-$1$ $f\:X \to Y$ to a normal variety; see \cite[\S2.4]{PatakfalviWaldronSingularitiesGeneralFibersLMMP}, \cite{BrantnerWaldronPurelyInsepGaloisTheoryFundamentalThm} and the references therein. Under this correspondence, $\omega_f \cong \sO_X((1-p)K_{\sF})$ where $-K_{\sF}$ is the determinant of $\sF$. For example, if $\sF=\sT_X$ then $f=F_X$ and $K_{\sF}=K_X$. In particular, to a section of $\omega_f$ we may attach an effective $\bZ_{(p)}$-divisor $\Delta$ such that $K_{\sF}+\Delta \sim_{\bZ_{(p)}}0$. 
    It would be interesting to understand how our work generalizes to this setup.
\end{remark}

\bibliographystyle{skalpha}
\bibliography{MainBib}
\end{document}